\pdfoutput=1
\RequirePackage{ifpdf}
\ifpdf % We are running pdfTeX in pdf mode
\documentclass[pdftex]{sigma}
\else
\documentclass{sigma}
\fi

\numberwithin{equation}{section}

\newtheorem{Theorem}{Theorem}[section]
\newtheorem{Corollary}[Theorem]{Corollary}
\newtheorem{Lemma}[Theorem]{Lemma}
\newtheorem{Proposition}[Theorem]{Proposition}
\newtheorem{Condition}[Theorem]{Condition}
 { \theoremstyle{definition}
\newtheorem{Definition}[Theorem]{Definition}
\newtheorem{Remark}[Theorem]{Remark} }

\begin{document}

\allowdisplaybreaks

\newcommand{\arXivNumber}{1612.01486}

\renewcommand{\PaperNumber}{040}

\FirstPageHeading

\ShortArticleName{A System of Dif\/ferential Equations on the Torus}

\ArticleName{A Linear System of Dif\/ferential Equations Related\\ to Vector-Valued Jack Polynomials on the Torus}

\Author{Charles F.~DUNKL}
\AuthorNameForHeading{C.F.~Dunkl}

\Address{Department of Mathematics, University of Virginia,\\ PO Box 400137, Charlottesville VA 22904-4137, USA}
\Email{\href{mailto:cfd5z@virginia.edu}{cfd5z@virginia.edu}}
\URLaddress{\url{http://people.virginia.edu/~cfd5z/}}

\ArticleDates{Received December 11, 2016, in f\/inal form June 02, 2017; Published online June 08, 2017}

\Abstract{For each irreducible module of the symmetric group $\mathcal{S}_{N}$ there is a set of para\-metrized nonsymmetric Jack polynomials in $N$ variables taking values in the module. These polynomials are simultaneous eigenfunctions of a~commutative set of operators, self-adjoint with respect to two Hermitian
forms, one called the contravariant form and the other is with respect to a~matrix-valued measure on the $N$-torus. The latter is valid for the parameter
lying in an interval about zero which depends on the module. The author in a~previous paper [\textit{SIGMA} \textbf{12} (2016), 033, 27~pages] proved the existence of the measure and that its absolutely continuous part satisf\/ies a system of linear dif\/ferential equations. In this paper the system is analyzed in detail. The $N$-torus is divided into $(N-1)!$ connected components by the hyperplanes $x_{i}=x_{j}$, $i<j$, which are the singularities of the system. The main result is that the orthogonality measure has no singular part with respect to Haar measure, and thus is given by a matrix function times Haar measure. This function is analytic on each of the connected components.}

\Keywords{nonsymmetric Jack polynomials; matrix-valued weight function; symmetric group modules}

\Classification{33C52; 32W50; 35F35; 20C30; 42B05}

\vspace{-3mm}

{\small \renewcommand{\baselinestretch}{0.90} \tableofcontents}

\section{Introduction}

The Jack polynomials form a parametrized basis of symmetric polynomials. A~special case of these consists of the Schur polynomials, important in the character theory of the symmetric groups. By means of a commutative algebra of dif\/ferential-dif\/ference operators the theory was extended to nonsymmetric Jack
polynomials, again a parametrized basis but now for all polynomials in~$N$ variables. These polynomials are orthogonal for several dif\/ferent inner products, and in each case they are simultaneous eigenfunctions of a~commutative set of self-adjoint operators. These inner products are invariant under permutations of the coordinates, that is, the symmetric group. One of these inner products is that of $L^{2}\big(\mathbb{T}^{N},K_{\kappa}(x) \mathrm{d}m(x)\big)$, where
\begin{gather*}
\mathbb{T}^{N} :=\big\{ x\in\mathbb{C}^{N}\colon \vert x_{j} \vert =1,\, 1\leq j\leq N\big\} ,\\
\mathrm{d}m(x) =(2\pi) ^{-N}\mathrm{d}\theta_{1}\cdots\mathrm{d}\theta_{N},
\qquad x_{j}=\exp( \mathrm{i}\theta _{j}) , \qquad -\pi<\theta_{j}\leq\pi, \qquad 1\leq j\leq N,\\
K_{\kappa}(x) =\prod_{1\leq i<j\leq N}\vert x_{i}-x_{j}\vert ^{2\kappa}, \qquad \kappa>-\frac{1}{N};
\end{gather*}
def\/ining the $N$-torus, the Haar measure on the torus, and the weight function respectively. Beerends and Opdam \cite{Beerends/Opdam1993} discovered this
orthogonality property of symmetric Jack polyno\-mials. Opdam \cite{Opdam1995} established orthogonality structures on the torus for trigonometric
polynomials associated with Weyl groups; the nonsymmetric Jack polynomials form a special case.

Grif\/feth \cite{Griffeth2010} constructed vector-valued Jack polynomials for the family $G(n,p,N) $ of complex ref\/lection groups. These are the groups of permutation matrices (exactly one nonzero entry in each row and each column) whose nonzero entries are $n^{\rm th}$ roots of unity and the product of these entries is a $(n/p)^{\rm th}$ root of unity. The symmetric groups and the hyperoctahedral groups are the special cases $G(1,1,N) $ and $G(2,1,N)$ respectively. The term ``vector-valued'' means that the polynomials take values in irreducible modules of the underlying group, and the action of the group is on the range as well as the domain of the polynomials. The author \cite{Dunkl2010} together with Luque \cite{Dunkl/Luque2011} investigated the symmetric group case more intensively. The basic setup is an irreducible representation of the symmetric group, specif\/ied by a~partition $\tau$ of~$N$, and a~parameter~$\kappa$ restricted to an interval determined by the partition, namely $-1/h_{\tau}<\kappa<1/h_{\tau}$ where $h_{\tau}$ is the maximum hook-length of the partition~$\tau$. More recently~\cite{Dunkl2016} we showed that there does exist a positive matrix measure on the torus for which the nonsymmetric vector-valued Jack polynomials (henceforth NSJP's) form an orthogonal set. The proof depends on a matrix-version of Bochner's theorem about the relation between positive measures on a compact abelian group and positive-def\/inite functions on the dual group, which is a discrete abelian group. In the present situation the torus is the compact (multiplicative) group and the dual is $\mathbb{Z}^{N}$. By using known properties of the NSJP's we produced a positive-def\/inite matrix function on $\mathbb{Z}^{N}$ and this implied the existence of the desired orthogonality measure. Additionally we showed that the part of the measure supported by $\mathbb{T}_{\rm reg}^{N}:=\mathbb{T}^{N}\backslash\bigcup_{i<j} \{ x\colon x_{i}=x_{j} \} $ is absolutely continuous with respect to the Haar measure $\mathrm{d}m$ and satisf\/ies a f\/irst-order dif\/ferential system. In this paper we complete the description of the measure by proving there is no singular part. The idea is to use the functional equations satisf\/ied by the inner product to establish a~correspondence to the dif\/ferential system. The main reason for the argument being so complicated is that the ``obvious'' integration-by-parts argument which works smoothly for the scalar case with $\kappa>1$ has great dif\/f\/iculty with the singularities of the measure of the form $\vert x_{i}-x_{j}\vert ^{-2\vert \kappa\vert }$. We use a Cauchy principal-value argument based on a weak continuity condition across the faces $\{x\colon x_{i}=x_{j}\} $ (as an over-simplif\/ied one-dimensional example consider the integral $\int_{-1}^{1}\frac{\mathrm{d}}{\mathrm{d}x}f(x) \mathrm{d}x$ with $f(x) =\vert 2x+x^{2}\vert ^{-1/4}$: the integral is divergent but the principal value $\lim\limits_{\varepsilon\rightarrow0_{+}}\big\{ \int_{-1}^{-\varepsilon}+\int_{\varepsilon}^{1}\big\} f^{\prime}(x) \mathrm{d}x=f(1) -f(-1) +\lim\limits_{\varepsilon \rightarrow0_{+}}\{ f(-\varepsilon) -f(( \varepsilon))\} $ and $f(-\varepsilon) -f(\varepsilon) =O\big( \varepsilon^{3/4}\big) $ hence the limit exists).

The dif\/ferential system is a two-sided version of a Knizhnik--Zamolodchikov equation (see \cite{Felder/Veselov2007}) modif\/ied to have solutions
homogeneous of degree zero, that is, constant on circles $\{( ux_{1},\ldots$, $ux_{N}) \colon \vert u\vert =1\} $. The purpose of the latter condition is to allow solutions analytic on connected components of~$\mathbb{T}_{\rm reg}^{N}$. Denote the degree of $\tau$ by $n_{\tau}$. The solutions of the dif\/ferential system are locally analytic $n_{\tau}\times n_{\tau}$ matrix functions with initial condition given by a constant matrix. That is, the solution space is of dimension $n_{\tau}^{2}$ but only one solution can provide the desired weight function. Part of the analysis deals with conditions specifying this solution~-- they turn out to be commutation relations involving certain group elements. In the subsequent discussion it is shown that the weight function property holds for a very small interval of $\kappa$ values if these relations are satisf\/ied. This is combined with the existence theorem of the positive-def\/inite matrix measure to f\/inally demonstrate that the measure has no singular part for any $\kappa$ in $-1/h_{\tau}<\kappa<1/h_{\tau}$.

In a subsequent development \cite{Dunkl2017} it is shown that the square root of the matrix weight function multiplied by vector-valued symmetric Jack polynomials provides novel wavefunctions of the Calogero--Sutherland quantum mechanical model of identical particles on a circle with $1/r^{2}$ interactions.

Here is an outline of the contents of the individual sections:
\begin{itemize}\itemsep=0pt
\item Section \ref{mods}: a short description of the representation of the symmetric group associated to a partition; the def\/inition of Dunkl operators for vector-valued polynomials and the def\/inition of nonsymmetric Jack polynomials (NSJP's) as simultaneous eigenvectors of a~commutative set of operators; and the Hermitian form given by an integral over the torus, for which the NSJP's form an orthogonal basis.

\item Section \ref{difsys}: the def\/inition of the linear system of dif\/ferential equations which will be demonstrated to have a unique matrix solution $L(x) $ such that $L(x) ^{\ast}L(x) \mathrm{d}m(x) $ is the weight function for the Hermitian form; the proof that the system is Frobenius integrable and the analyticity and monodromy properties of the solutions on the torus.

\item Section \ref{byparts}: the use of the dif\/ferential equation to relate the Hermitian form to $L(x) ^{\ast}L(x) $ by means of integration by parts; the result of this is to isolate the role of the singularities in the process of proving the orthogonality of the NSJP's with respect to $L^{\ast}L\mathrm{d}m$.

\item Section \ref{locps}: deriving power series expansions of $L(x) $ near the singular set $\bigcup_{i<j}\big\{ x\in\mathbb{T}^{N}\colon$ $x_{i}=x_{j}\big\}$, in particular near the set $\{x\colon x_{N-1}=x_{N}\} $; description of commutation properties of the coef\/f\/icients with respect to the ref\/lection $(N-1,N) $; the behavior of $L$ across the mirror $\{ x\colon x_{N-1}=x_{N}\} $.

\item Section \ref{bnds}: the derivation of global bounds on $L(x)$ and local bounds on the coef\/f\/icients of the power series, needed to analyze convergence properties of the integration by parts.

\item Section \ref{suffco}: the proof of a suf\/f\/icient condition for the validity of the Hermitian form; the condition is partly that~$\kappa$ lies in a small interval around $0$ and that the boundary value of~$L(x)$ satisf\/ies a commutativity condition; the proof involves very detailed analysis of bounds on~$L$, since the local bounds have to be integrated over the entire torus.

\item Section \ref{orthmu}: further analysis of the orthogonality measure constructed in \cite{Dunkl2016}, in particular the proof of the formal dif\/ferential system satisf\/ied by the Fourier--Stieltjes (Laurent) series of the measure; this is used to show that the measure has no singular part on the open faces, such as
\begin{gather*}
\big\{ \big( e^{\mathrm{i}\theta_{1}},e^{\mathrm{i}\theta_{2}},\ldots,e^{\mathrm{i}\theta_{N-1}},e^{\mathrm{i}\theta_{N-1}}\big)\colon \theta_{1}<\theta_{2}<\cdots<\theta_{N-2}<\theta_{N-1}<\theta_{1}+2\pi\big\} ;
\end{gather*}
in turn this property is shown to imply the validity of the suf\/f\/icient condition set up in Section~\ref{suffco}.

\item Section \ref{anlcmat}: analyticity properties of the solutions of matrix equations with analytic coef\/f\/i\-cients; the results are used to extend the validity of the Hermitian form to the desired interval $-1/h_{\tau}<\kappa<1/h_{\tau}$ from the smaller interval found in Section~\ref{suffco}.
\end{itemize}

\section{Modules of the symmetric group}\label{mods}

The \textit{symmetric group} $\mathcal{S}_{N}$, the set of permutations of $\{1,2,\ldots,N\} $, acts on $\mathbb{C}^{N}$ by permutation of coordinates. For $\alpha\in\mathbb{Z}^{N}$ the norm is $\vert \alpha\vert :=\sum\limits_{i=1}^{N}\vert \alpha_{i}\vert $ and the monomial is $x^{\alpha}:= \prod\limits_{i=1}^{N} x_{i}^{\alpha_{i}}$. Denote $\mathbb{N}_{0}:=\{0,1,2,\ldots\}$. The space of polynomials $\mathcal{P}:= \operatorname{span}_{\mathbb{C}}\big\{ x^{\alpha}\colon \alpha \in\mathbb{N}_{0}^{N}\big\} $. Elements of $\operatorname{span}_{\mathbb{C}}\big\{ x^{\alpha}\colon \alpha\in\mathbb{Z}^{N}\big\} $ are called \textit{Laurent} polynomials. The action of $\mathcal{S}_{N}$ is extended to polynomials by $wp(x) =p(xw) $ where $( xw) _{i}=x_{w(i) }$ (consider $x$ as a row vector and~$w$ as a~permutation matrix, $[w] _{ij}=\delta_{i,w(j)}$, then $xw=x[w] $). This is a representation of $\mathcal{S}_{N}$, that is, $w_{1}(w_{2}p) (x) =(w_{2}p) (xw_{1}) =p(xw_{1}w_{2}) =(w_{1}w_{2}) p(x) $ for all $w_{1},w_{2}\in\mathcal{S}_{N}$.

Furthermore $\mathcal{S}_{N}$ is generated by ref\/lections in the mirrors $\{x\colon x_{i}=x_{j}\} $ for $1\leq i<j\leq N$. These are \textit{transpositions}, denoted by $(i,j)$, so that $x(i,j) $ denotes the result of interchanging~$x_{i}$ and~$x_{j}$. Def\/ine the $\mathcal{S}_{N}$-action on~$\alpha\in\mathbb{Z}^{N}$ so that $(xw) ^{\alpha}=x^{w\alpha}$
\begin{gather*}
(xw) ^{\alpha}=\prod_{i=1}^{N}x_{w(i) }^{\alpha_{i}}=\prod_{j=1}^{N}x_{j}^{\alpha_{w^{-1}(j) }},
\end{gather*}
that is $(w\alpha) _{i}=\alpha_{w^{-1}(i) }$ (take $\alpha$ as a column vector, then $w\alpha=[w] \alpha$).

The \textit{simple reflections} $s_{i}:=(i,i+1)$, $1\leq i\leq N-1$, generate $\mathcal{S}_{N}$. They are the key devices for applying inductive methods, and satisfy the \textit{braid} relations:
\begin{gather*}
s_{i}s_{j} =s_{j}s_{i},\qquad \vert i-j\vert \geq2;\\
s_{i}s_{i+1}s_{i} =s_{i+1}s_{i}s_{i+1}.
\end{gather*}

We consider the situation where the group $\mathcal{S}_{N}$ acts on the range as well as on the domain of the polynomials. We use vector spaces, called
$\mathcal{S}_{N}$-modules, on which $\mathcal{S}_{N}$ has an irreducible unitary (orthogonal) representation: $\tau\colon \mathcal{S}_{N}\rightarrow
O_{m}(\mathbb{R}) $ $\big(\tau(w) ^{-1}=\tau\big( w^{-1}\big) =\tau(w) ^{T}\big)$. See James and Kerber~\cite{James/Kerber1981} for representation theory, including a modern discussion of Young's methods.

Denote the set of \textit{partitions}
\begin{gather*}
\mathbb{N}_{0}^{N,+}:=\big\{ \lambda\in\mathbb{N}_{0}^{N}\colon \lambda_{1}\geq\lambda_{2}\geq\cdots\geq\lambda_{N}\big\} .
\end{gather*}
We identify $\tau$ with a partition of $N$ given the same label, that is $\tau\in\mathbb{N}_{0}^{N,+}$ and $\vert\tau\vert =N$. The length of~$\tau$ is $\ell(\tau) :=\max \{ i\colon \tau_{i}>0 \} $. There is a Ferrers diagram of shape~$\tau$ (also given the same label), with boxes at points $(i,j) $ with $1\leq i\leq\ell(\tau) $ and $1\leq j\leq\tau_{i}$. A \textit{tableau} of shape~$\tau$ is a~f\/illing of the boxes with numbers, and a \textit{reverse standard Young tableau} (RSYT) is a~f\/illing with the numbers $\{1,2,\ldots,N\} $ so that the entries decrease in each row and each column. We exclude the one-dimensional representations corresponding to one-row $(N) $ or one-column $(1,1,\ldots,1) $ partitions (the trivial and determinant representations, respectively). We need the important quantity $h_{\tau}:=\tau_{1}+\ell(\tau) -1$, the maximum hook-length of the diagram (the \textit{hook-length} of the node~$(i,j) \in \tau$ is def\/ined to be $\tau_{i}-j+\# \{ k\colon i<k\leq\ell(\tau) \&j\leq\tau_{k} \} +1$). Denote the set of RSYT's of shape~$\tau$ by $\mathcal{Y}(\tau)$ and let
\begin{gather*}
V_{\tau}=\operatorname{span} \{ T\colon T\in\mathcal{Y}(\tau)\}
\end{gather*}
(the f\/ield is $\mathbb{C}(\kappa) $) with orthogonal basis $\mathcal{Y}(\tau) $. For $1\leq i\leq N$ and $T\in \mathcal{Y}(\tau) $ the entry $i$ is at coordinates $( rw(i,T) ,cm(i,T)) $ and the \textit{content} is $c(i,T) :=cm(i,T) -rw(i,T) $. Each $T\in\mathcal{Y}(\tau) $ is uniquely determined by its \textit{content vector} $[ c(i,T)] _{i=1}^{N}$. Let $S_{1}(\tau) =\sum\limits_{i=1}^{N}c(i,T) $ (this sum depends only on $\tau$) and $\gamma:=S_{1}(\tau)/N$. The $\mathcal{S}_{N}$-invariant inner product on $V_{\tau}$ is def\/ined by
\begin{gather*}
\langle T,T^{\prime}\rangle _{0}:=\delta_{T,T^{\prime}}\times \prod_{\substack{1\leq i<j\leq N,\\ c(i,T) \leq c(j,T) -2}}\left( 1-\frac{1}{( c(i,T) -c(j,T)) ^{2}}\right) , \qquad T,T^{\prime}\in\mathcal{Y} (\tau) .
\end{gather*}
It is unique up to multiplication by a constant.

The \textit{Jucys--Murphy} elements $\sum\limits_{j=i+1}^{N}(i,j)$ satisfy $\sum\limits_{j=i+1}^{N}\tau((i,j)) T=c(i,T) T$ and thus the central element $\sum\limits_{1\leq i<j\leq N}(i,j) $ satisf\/ies $\sum\limits_{1\leq i<j\leq N} \tau((i,j)) T=S_{1}(\tau) T$ for each $T\in\mathcal{Y}(\tau) $. The basis is ordered such that the vectors $T$ with $c(N-1,T) =-1$ appear f\/irst (that is, $cm(N-1,T) =1$, $rw(N-1,T) =2$). This results in the matrix representation of $\tau((N-1,N)) $ being
\begin{gather*}
\left[
\begin{matrix}
-I_{m_{\tau}} & O\\
O & I_{n_{\tau}-m_{\tau}}
\end{matrix}
\right] ,
\end{gather*}
where $n_{\tau}:=\dim V_{\tau}=\#\mathcal{Y}(\tau) $ and $m_{\tau}$ is given by $\operatorname{tr}( \tau( (N-1,N))) =n_{\tau}-2m_{\tau}$. From the sum $\sum\limits_{i<j}\tau ((i,j)) =S_{1}(\tau) I$ it follows that $\binom{N}{2}\operatorname{tr}( \tau((N-1,N))) =S_{1}(\tau) n_{\tau}$ and $m_{\tau}= n_{\tau}\big( \frac{1}{2}-\frac{S_{1}(\tau) }{N(N-1) }\big)$. (The transpositions are conjugate to each other implying the traces are equal.)

\subsection{Jack polynomials}

The main concerns of this paper are measures and matrix functions on the torus associated to $\mathcal{P}_{\tau}:=\mathcal{P}\otimes V_{\tau}$, the space of
$V_{\tau}$-valued polynomials, which is equipped with the $\mathcal{S}_{N}$ action:
\begin{gather*}
w\left( x^{\alpha}\otimes T\right) =(xw) ^{\alpha}\otimes\tau(w) T, \qquad \alpha\in\mathbb{N}_{0}^{N}, \qquad T\in \mathcal{Y}(\tau) ,\\
wp(x) =\tau(w) p(xw), \qquad p\in\mathcal{P}_{\tau},
\end{gather*}
extended by linearity. There is a parameter $\kappa$ which may be generic/transcendental or complex.

\begin{Definition}
The \textit{Dunkl} and \textit{Cherednik--Dunkl} operators are ($1\leq i\leq N$, $p\in\mathcal{P}_{\tau}$)
\begin{gather*}
\mathcal{D}_{i}p(x) :=\frac{\partial}{\partial x_{i}}p(x) +\kappa\sum_{j\neq i}\tau((i,j)) \frac{p(x) -p(x(i,j))}{x_{i}-x_{j}},\\
\mathcal{U}_{i}p(x) :=\mathcal{D}_{i}( x_{i}p(x)) -\kappa\sum_{j=1}^{i-1}\tau( (i,j)) p(x(i,j)) .
\end{gather*}
\end{Definition}

The commutation relations analogous to the scalar case hold:
\begin{gather*}
\mathcal{D}_{i}\mathcal{D}_{j} =\mathcal{D}_{j}\mathcal{D}_{i}, \qquad \mathcal{U}_{i}\mathcal{U}_{j}=\mathcal{U}_{j}\mathcal{U}_{i}, \qquad 1\leq i,j\leq
N;\\
w\mathcal{D}_{i} =\mathcal{D}_{w(i) }w, \qquad \forall\, w\in \mathcal{S}_{N}; \qquad s_{j}\mathcal{U}_{i}=\mathcal{U}_{i}s_{j}, \qquad j\neq i-1,i;\\
s_{i}\mathcal{U}_{i}s_{i} =\mathcal{U}_{i+1}+\kappa s_{i}, \qquad \mathcal{U}_{i}s_{i}=s_{i}\mathcal{U}_{i+1}+\kappa, \qquad \mathcal{U}_{i+1}s_{i}=s_{i}\mathcal{U}_{i}-\kappa.
\end{gather*}
The simultaneous eigenfunctions of $\{\mathcal{U}_{i}\} $ are called (vector-valued) nonsymmetric Jack polynomials (NSJP). For generic~$\kappa$ these eigenfunctions form a basis of $\mathcal{P}_{\tau}$ (this property fails for certain rational numbers outside the interval $-1/h_{\tau}<\kappa<1/h_{\tau}$). There is a partial order on $\mathbb{N}_{0}^{N}\times\mathcal{Y}(\tau) $ for which the NSJP's have a triangular expression with leading term indexed by $(\alpha,T) \in\mathbb{N}_{0}^{N}\times\mathcal{Y}(\tau) $. The polynomial with this label is denoted by $\zeta_{\alpha,T}$, homogeneous of degree $\sum\limits_{i=1}^{N}\alpha_{i}$ and satisf\/ies
\begin{gather*}
\mathcal{U}_{i}\zeta_{\alpha,T} =\left( \alpha_{i}+1+\kappa c\left(r_{\alpha}(i) ,T\right) \right) \zeta_{\alpha,T}, \qquad 1\leq i\leq N,\\
r_{\alpha}(i) :=\#\{ j\colon \alpha_{j}>\alpha_{i}\}+\# \{ j\colon 1\leq j\leq i,\alpha_{j}=\alpha_{i} \} ;
\end{gather*}
the rank function $r_{\alpha}\in\mathcal{S}_{N}$ and $r_{\alpha}=I$ if and only if $\alpha$ is a partition. The vector
\begin{gather*}
[ \alpha_{i}+1+\kappa c( r_{\alpha}(i) ,T)] _{i=1}^{N}
\end{gather*} is called the spectral vector for $(\alpha,T) $. The NSJP structure can be extended to Laurent polynomials. Let $e_{N}:=\prod\limits_{i=1}^{N}x_{i}$ and $\boldsymbol{1}:=(1,1,\ldots,1) \in\mathbb{N}_{0}^{N}$, then $r_{\alpha+m\boldsymbol{1}}=r_{\alpha}$ for any $\alpha\in\mathbb{N}_{0}^{N}$ and $m\in\mathbb{Z}$. The commutation $\mathcal{U}_{i}( e_{N}^{m}p) =e_{N}^{m}( m+\mathcal{U}_{i}) p$ for $1\leq i\leq
N$ and $p\in\mathcal{P}_{\tau}$ imply that $e_{N}^{m}\zeta_{\alpha,T}$ and $\zeta_{\alpha+m\boldsymbol{1},T}$ have the same spectral vector for any
$m\in\mathbb{N}_{0}$. They also have the same leading term (see \cite[Section~2.2]{Dunkl2016}) and hence $e_{N}^{m}\zeta_{\alpha,T}=\zeta_{\alpha
+m\boldsymbol{1},T}$ for $\alpha\in\mathbb{N}_{0}^{N}$. This fact allows the def\/inition of~$\zeta_{\alpha,T}$ for any $\alpha\in\mathbb{Z}^{N}$: let
$m=-\min_{i}\alpha_{i}$ then $\alpha+m\boldsymbol{1}\in\mathbb{N}_{0}^{N}$ and set $\zeta_{\alpha,T}:=e_{N}^{-m}\zeta_{\alpha+m\boldsymbol{1},T}$.

For a complex vector space $V$ a Hermitian form is a mapping $\langle \cdot,\cdot\rangle \colon V\otimes V\rightarrow\mathbb{C}$ such that $\langle u,cv\rangle =c\langle u,v\rangle $, $\langle u,v_{1}+v_{2}\rangle =\langle u,v_{1}\rangle+\langle u,v_{2}\rangle $ and $\langle u,v\rangle
=\overline{\langle v,u\rangle }$ for $u,v_{1},v_{2}\in V$, $c\in\mathbb{C}$. The form is positive semidef\/inite if $\langle u,u\rangle \geq0$ for all $u\in V$. The concern of this paper is with a~particular Hermitian form on $\mathcal{P}_{\tau}$ which has the properties (for all $f,g\in\mathcal{P}_{\tau},T,T^{\prime}\in\mathcal{Y}(\tau)$, $w\in\mathcal{S}_{N}$, $1\leq i\leq N$):
\begin{gather}
\langle 1\otimes T,1\otimes T^{\prime} \rangle =\langle T,T^{\prime} \rangle _{0},\label{admforms}\\
\langle wf,wg \rangle =\langle f,g\rangle,\nonumber\\
\langle x_{i}\mathcal{D}_{i}f,g\rangle =\langle f,x_{i}\mathcal{D}_{i}g\rangle ,\nonumber\\
\langle x_{i}f,x_{i}g\rangle =\langle f,g\rangle.\nonumber
\end{gather}
The commutation $\mathcal{U}_{i}=\mathcal{D}_{i}x_{i}-\kappa\sum\limits_{j<i}(i,j) =x_{i}\mathcal{D}_{i}+1+\kappa\sum\limits_{j>i}(i,j) $ together with $\langle (i,j) f,g\rangle =\langle f,(i,j)g\rangle $ show that $\langle \mathcal{U}_{i}f,g\rangle= \langle f,\mathcal{U}_{i}g\rangle $ for all~$i$. Thus uniqueness of the spectral vectors (for all but a certain set of rational~$\kappa$ values) implies that $\langle \zeta_{\alpha,T},\zeta_{\beta,T^{\prime}
}\rangle =0$ whenever $(\alpha,T) \neq(\beta,T^{\prime}) $. In particular polynomials homogeneous of dif\/ferent degrees are mutually orthogonal, by the basis property of $\{ \zeta_{\alpha,T}\} $. For this particular Hermitian form, multiplication by any $x_{i}$ is an isometry for all $1\leq i\leq N$. This
involves an integral over the torus. The equations~(\ref{admforms}) determine the form uniquely (up to a multiplicative constant if the f\/irst condition is removed).

Denote $\mathbb{C}_{\times}:=\mathbb{C}\backslash\{0\} $ and $\mathbb{C}_{\rm reg}^{N}:=\mathbb{C}_{\times}^{N}\backslash\bigcup\limits_{i<j}\{x\colon x_{i}=x_{j}\}$. The torus is a compact multiplicative abelian group. The notations for the torus and its Haar measure in terms of
polar coordinates are%
\begin{gather*}
\mathbb{T}^{N} :=\big\{ x\in\mathbb{C}^{N}\colon \vert x_{j} \vert =1,\, 1\leq j\leq N\big\} ,\\
\mathrm{d}m(x) =(2\pi) ^{-N}\mathrm{d}\theta_{1}\cdots\mathrm{d}\theta_{N}, \qquad x_{j}=\exp ( \mathrm{i}\theta _{j}) , \qquad -\pi<\theta_{j}\leq\pi, \qquad 1\leq j\leq N.
\end{gather*}

Let $\mathbb{T}_{\rm reg}^{N}:=\mathbb{T}^{N}\cap\mathbb{C}_{\rm reg}^{N}$, then $\mathbb{T}_{\rm reg}^{N}$ has $(N-1) !$ connected components and each component is homotopic to a circle (if~$x$ is in some component then so is $ux= ( ux_{1},\ldots,ux_{N} ) $ for each $u\in\mathbb{T}$).

\begin{Definition}
Let $\omega:=\exp\frac{2\pi\mathrm{i}}{N}$ and $x_{0}:=\big( 1,\omega,\ldots,\omega^{N-1}\big)$. Denote the connected component of~$\mathbb{T}_{\rm reg}^{N}$ containing $x_{0}$ by $\mathcal{C}_{0}$.
\end{Definition}

Thus $\mathcal{C}_{0}$ is the set consisting of $\big( e^{\mathrm{i}\theta_{1}},\ldots,e^{\mathrm{i}\theta_{N}}\big) $ with $\theta_{1}<\theta_{2}< \cdots<\theta_{N}<\theta_{1}+2\pi$.

In \cite{Dunkl2016} we showed that if $-1/h_{\tau}<\kappa<1/h_{\tau}$ then there exists a positive matrix-valued measure~$\mathrm{d}\mu$ on~$\mathbb{T}^{N}$ such that for $f,g\in C^{(1) }\big( \mathbb{T}^{N};V_{\tau}\big) $, $w\in\mathcal{S}_{N}$, $1\leq i\leq N$, %
\begin{gather*}
\int_{\mathbb{T}^{N}}f(x) ^{\ast}\mathrm{d}\mu(x)g(x) =\int_{\mathbb{T}^{N}}f(xw) ^{\ast}\tau(w) ^{-1}\mathrm{d}\mu(x) \tau(w) g(xw) ,\\
\int_{\mathbb{T}^{N}}( x_{i}\mathcal{D}_{i}f(x))^{\ast}\mathrm{d}\mu(x) g(x) =\int_{\mathbb{T}^{N}}f(x) ^{\ast}\mathrm{d}\mu(x)x_{i}\mathcal{D}_{i}g(x) ,\\
\int_{\mathbb{T}^{N}}(1\otimes T) ^{\ast}\mathrm{d}\mu (x) (1\otimes T) =\langle T,T\rangle_{0}, \qquad T\in\mathcal{Y}(\tau) .
\end{gather*}
We introduced the notation
\begin{gather*}
f(x) ^{\ast}\mathrm{d}\mu(x) g(x):=\sum\limits_{T,T^{\prime}\in\mathcal{Y}(\tau) }\overline {f(x) _{T}}g(x) _{T^{\prime}}\mathrm{d}\mu_{T,T^{\prime}}(x),
\end{gather*}
where $f,g\in\mathcal{P}_{\tau}$ have the components $(f_{T}),(g_{T}) $ with respect to the orthonormal basis
\begin{gather*}
\big\{ \langle T,T\rangle _{0}^{-1/2}T\colon T\in\mathcal{Y}(\tau)\big\}.
\end{gather*}
Thus the Hermitian form $\langle f,g\rangle =\int_{\mathbb{T}^{N}}f(x) ^{\ast}\mathrm{d}\mu(x) g(x) $ satisf\/ies~(\ref{admforms}). Furthermore we showed that
\begin{gather*}
\mathrm{d}\mu(x) =\mathrm{d}\mu_{s}(x) +L(x) ^{\ast}H(\mathcal{C}) L(x)\mathrm{d}m(x) ,
\end{gather*}
where the singular part $\mu_{s}$ is the restriction of $\mu$ to $\mathbb{T}^{N}\backslash\mathbb{T}_{\rm reg}^{N}$, $H(\mathcal{C}) $ is constant and positive-def\/inite on each connected component $\mathcal{C}$ of $\mathbb{T}_{\rm reg}^{N}$ and $L(x) $ is a matrix function solving a system of dif\/ferential equations. That system is the subject of this paper. In a way the main problem is to show that $\mu$ has no singular part.

\section{The dif\/ferential system}\label{difsys}

Consider the system (with $\partial_{i}:=\frac{\partial}{\partial x_{i}}$, $1\leq i\leq N$) for $n_{\tau}\times n_{\tau}$ matrix functions $L(x) $
\begin{gather}
\partial_{i}L(x) =\kappa L(x) \left\{\sum_{j\neq i}\frac{1}{x_{i}-x_{j}}\tau((i,j))-\frac{\gamma}{x_{i}}I\right\} ,\qquad 1\leq i\leq N,\label{Lsys}\\
\gamma :=\frac{S_{1}(\tau) }{N}=\frac{1}{2N}\sum_{i=1}^{\ell(\tau) }\tau_{i}(\tau_{i}-2i+1) .\nonumber
\end{gather}
The ef\/fect of the term $\frac{\gamma}{x_{i}}I$ is to make $L(x)$ homogeneous of degree zero, that is, $\sum\limits_{i=1}^{N}x_{i}\partial_{i}L(x) =0$. The dif\/ferential system is def\/ined on $\mathbb{C}_{\rm reg}^{N}$, Frobenius integrable and analytic, thus any local solution can be continued analytically to any point in $\mathbb{C}_{\rm reg}^{N}$. Dif\/ferent paths may produce dif\/ferent values; if the analytic continuation is done along a closed path then the resultant solution is a constant matrix multiple of the original solution, called the monodromy matrix, however if the closed path is contained in a simply connected subset of $\mathbb{C}_{\rm reg}^{N}$ then there is no change.

Integrability means that $\partial_{i}( \kappa L(x) A_{j}(x)) =\partial_{j}( \kappa L(x)A_{i}(x)) $ for $i\neq j$, writing the system as
$\partial_{i}L(x) =\kappa L(x) A_{i} (x) $, $1\leq i\leq N$ (where $A_{i}(x) $ is def\/ined by equation~(\ref{Lsys})). The condition becomes%
\begin{gather*}
\kappa^{2}L(x) A_{i}(x) A_{j}(x) +\kappa L(x) \partial_{i}A_{j}(x) =\kappa^{2}L(x) A_{j}(x) A_{i}(x) +\kappa L(x) \partial_{j}A_{i}(x) .
\end{gather*}
Since $\partial_{i}A_{j}(x) =\frac{\tau((i,j)) }{( x_{i}-x_{j}) ^{2}}=\partial_{j}A_{i}(x) $ it suf\/f\/ices to show that $A_{i}(x)^{\prime}:=\sum\limits_{k\neq i}\frac{\tau(( i,k)) }{x_{i}-x_{k}}$ and $A_{j}(x) ^{\prime}:=\sum\limits_{\ell\neq j}\frac{\tau((j,\ell)) }{x_{j}-x_{\ell}}$ commute with each other. The product $A_{i}(x)^{\prime}A_{j}(x) ^{\prime}$ is a sum of $-\frac{1}{(x_{i}-x_{j}) ^{2}}I$, terms of the form $\frac{\tau((i,k) ( j,\ell)) }{( x_{i}-x_{k})( x_{j}-x_{\ell}) }+\frac{\tau(( i,\ell)( j,k)) }{( x_{i}-x_{\ell})(x_{j}-x_{k}) }$ for $\{ i,k\} \cap\{ j,\ell\}
=\varnothing$, and terms involving the $3$-cycles $(i,j,k) $ and $(j,i,k) $ occurring as
\begin{gather*}
 \frac{\tau ( (i,j) (j,k)) }{(x_{i}-x_{j}) (x_{j}-x_{k}) }+\frac{\tau((i,k) (j,i)) }{(x_{i}-x_{k})
(x_{j}-x_{i}) }+\frac{\tau ( (i,k) (j,k)) }{(x_{i}-x_{k}) (x_{j}-x_{k})}\\
\qquad{} =\frac{\tau((i,j,k)) }{( x_{i}-x_{k}) (x_{j}-x_{k}) }+\frac{\tau((j,i,k)) }{(x_{i}-x_{k})( x_{j}-x_{k}) },
\end{gather*}
(because $(i,j) (j,k) =(i,k) (j,i) =(i,j,k) $ and $(i,k)(j,k) =(j,i,k) $) and the latter two terms are symmetric in~$i$,~$j$.

We consider only fundamental solutions, that is, $\det L(x)\neq0$. Recall Jacobi's identity:
\begin{gather*}
\frac{\partial}{\partial t}\det F(t) =\operatorname{tr}\left(\operatorname{adj}(F(t)) \frac{\partial}{\partial t}F(t) \right),
\end{gather*}
where $F(t) $ is a dif\/ferentiable matrix function and $\operatorname{adj}(F(t)) F(t) =\det F(t) I$, that is, $\operatorname{adj} ( F(t)) =\{ \det F(t) \} F(t) ^{-1}$ when $F(t) $ is invertible; thus
\begin{gather*}
\partial_{i}\det L(x) =\operatorname{tr}\big( \{ \det L(x)\} L(x) ^{-1}\partial_{i}L(x)
\big) =\kappa\det L(x) \operatorname{tr}A_{i}(x) .
\end{gather*}
This can be solved: from $\sum\limits_{i<j}\tau((i,j)) =S_{1}(\tau) I$ it follows that $\operatorname{tr}( \tau((i,j))) =\binom{N}{2}^{-1}S_{1}(\tau) n_{\tau}=\frac{2}{N-1}\gamma n_{\tau}$ (and $n_{\tau}=\#\mathcal{Y}(\tau) $). We obtain the system
\begin{gather*}
\partial_{i}\det L(x) =\kappa\gamma n_{\tau}\left\{ \frac {2}{N-1}\sum_{j\neq i}\frac{1}{x_{i}-x_{j}}-\frac{1}{x_{i}}\right\} \det
L(x) , \qquad 1\leq i\leq N.
\end{gather*}
By direct verif\/ication
\begin{gather*}
\det L(x) =c\prod\limits_{1\leq i<j\leq N}\left( -\frac{(x_{i}-x_{j}) ^{2}}{x_{i}x_{j}}\right) ^{\kappa\lambda/2},\qquad \lambda:=\frac{\gamma n_{\tau}}{2(N-1) }=\operatorname{tr}(\tau(( 1,2))) ,
\end{gather*}
is a local solution for any $c\in\mathbb{C}_{\times}$, if $x_{k}=e^{\mathrm{i}\theta_{k}}$, $1\leq k\leq N$ then $-\frac{( x_{i}-x_{j})^{2}}{x_{i}x_{j}}=4\sin^{2}\frac{\theta_{i}-\theta_{j}}{2}$ (with the principal branch of the power function, positive on positive reals).
This implies $\det L(x) \neq0$ for $x\in\mathbb{C}_{\rm reg}^{N}$ (and of course $\det L(x) $ is homogeneous of degree zero).

\begin{Proposition} \label{L(xw)}If $L(x) $ is a solution of \eqref{Lsys} in some connected open subset $U$ of $\mathbb{C}_{\rm reg}^{N}$ then $L(xw)\tau(w) ^{-1}$ is a solution in $Uw^{-1}$.
\end{Proposition}

\begin{proof} First let $w=(j,k) $ for some f\/ixed $j$, $k$. If $i\neq j,k$ then replace $x$ by $x(j,k) $ in $\partial_{i}L$ to obtain%
\begin{gather*}
 \partial_{i}L(x(j,k)) \tau(( j,k)) =\kappa L(x(j,k)) \left\{ \sum_{\ell\neq
i,j,k}\frac{\tau((i,\ell)) }{x_{i}-x_{\ell}}+\frac{\tau((i,j)) }{x_{i}-x_{k}}+\frac
{\tau((i,k)) }{x_{i}-x_{j}}-\frac{\gamma}{x_{i}}I\right\} \tau(j,k) \\
\hphantom{\partial_{i}L(x(j,k)) \tau(( j,k))}{} =\kappa L(x(j,k)) \tau(j,k)
\left\{ \sum_{\ell\neq i,j,k}\frac{\tau ( (i,\ell)) }{x_{i}-x_{\ell}}+\frac{\tau((i,k))
}{x_{i}-x_{k}}+\frac{\tau((i,j)) }{x_{i}-x_{j}}-\frac{\gamma}{x_{i}}I\right\} ,
\end{gather*}
because $(i,j) (j,k) =(j,k) (i,k) $. Next let $w=(i,j) $, then $\partial_{i}[L(x(i,j))] =(\partial_{j}L)(x(i,j)) $ and
\begin{gather*}
\partial_{i}[ L(x(i,j)) \tau((i,j))] =\kappa L( x(i,j)) \left\{ \sum_{\ell\neq i,j}\frac{\tau ( (j,\ell)) }{x_{i}-x_{\ell}}+\frac{\tau((i,j))
}{x_{i}-x_{j}}-\frac{\gamma}{x_{i}}I\right\} \tau( (i,j)) \\
 \hphantom{\partial_{i}[ L(x(i,j)) \tau((i,j))]}{} =\kappa L(x(i,j)) \tau( (i,j)) \left\{ \sum_{\ell\neq i,j}\frac{\tau((i,\ell)) }{x_{i}-x_{\ell}}+\frac{\tau((i,j)) }{x_{i}-x_{j}}-\frac{\gamma}{x_{i}}I\right\} ,
\end{gather*}
by use of $(j,\ell) (i,j) =(i,j)(i,\ell) $. Arguing by induction suppose $L(xw) \tau(w) ^{-1}$ is a solution then so is $L( x(j,k) w) \tau(w) ^{-1}\tau((j,k)) =L( x(j,k) w) \tau((j,k) w) ^{-1}$, for any$(j,k) $, that is, the statement holds for $w^{\prime}=(j,k) w$.
\end{proof}

Let $w_{0}:= ( 1,2,3,\ldots, N ) = ( 12 )(23)\cdots(N-1,N) $ denote the $N$-cycle and let $\langle w_{0}\rangle $ denote the cyclic group generated by~$w_{0}$. There are two components of $\mathbb{T}_{\rm reg}^{N}$ which are set-wise invariant under $\langle w_{0}\rangle $ namely $\mathcal{C}_{0}$
and the reverse $\{ \theta_{N}<\theta_{N-1}<\cdots<\theta_{1}<\theta _{N}+2\pi\} $. Indeed $\langle w_{0}\rangle $ is the stabilizer of $\mathcal{C}_{0}$ as a subgroup of $\mathcal{S}_{N}$.

Henceforth we use $L(x) $ to denote the solution of (\ref{Lsys}) in $\mathcal{C}_{0}$ which satisf\/ies $L(x_{0}) =I$.

\begin{Proposition} Suppose $x\in\mathcal{C}_{0}$ and $m\in\mathbb{Z}$ then $L( xw_{0}^{m}) =\tau(w_{0}) ^{-m}L(x) \tau(w_{0}) ^{m}$.
\end{Proposition}

\begin{proof}Consider the solution $L(xw_{0}) \tau(w_{0}) ^{-1}$ which agrees with $\Xi L(x) $ for all $x\in \mathcal{C}_{0}$ for some f\/ixed matrix $\Xi$. In particular for $x=x_{0}$ where $x_{0}w_{0}=\big( \omega,\ldots,\omega^{N-1},1\big) =\omega x_{0}$ (recall $(xw) _{i}=x_{w(i) }$) we obtain $\Xi L(x_{0}) =L ( x_{0}w_{0} ) \tau(w_{0})
^{-1}=L(\omega x_{0}) \tau(w_{0}) ^{-1}=L(x_{0}) \tau(w_{0}) ^{-1}$; because $L(x) $ is homogeneous of degree zero. Thus $\Xi=\tau(w_{0}) ^{-1}$. Repeated use of the relation shows $L( xw_{0}^{m}) =\tau(w_{0}) ^{-m}L(x) \tau(w_{0}) ^{m}$.
\end{proof}

Because of its frequent use denote $\upsilon:=\tau(w_{0}) $ (the letter $\upsilon$ occurs in the Greek word \textit{cycle}).

\begin{Definition}
For $w\in\mathcal{S}_{N}$ set $\nu(w) :=\upsilon^{1-w(1) }$. For any $x\in\mathbb{T}_{\rm reg}^{N}$ there is a unique $w_{x}$ such that $w_{x}(1) =1$ and $xw_{x}^{-1}\in\mathcal{C}_{0}$. Set $M(w,x) :=\nu(w_{x}w) $.
\end{Definition}

As a consequence $\nu ( w_{0}^{m}w ) =\upsilon^{-m}\nu(w) $ for any $w\in\mathcal{S}_{N}$ and $m\in\mathbb{Z}$; since
$w_{0}^{m}w(1) -1=(w(1) +m-1) \operatorname{mod}N$. Also $M(I,x) =I$. There is a 1-1 correspondence $w\mapsto\mathcal{C}_{0}w$ between $ \{ w\in \mathcal{S}_{N}\colon w(1) =1 \} $ and the connected components
of $\mathbb{T}_{\rm reg}^{N}$.

\begin{Proposition}
For any $w_{1},w_{2}\in\mathcal{S}_{N}$ and $x\in\mathbb{T}_{\rm reg}^{N}$%
\begin{gather*}
M( w_{1}w_{2},x) =M( w_{2},xw_{1}) M (w_{1},x) .
\end{gather*}
\end{Proposition}

\begin{proof}
By def\/inition $M(w_{1}w_{2},x) =\nu(w_{x}w_{1}w_{2}) $ and $M(w_{1},x) =\nu(w_{x}w_{1})
=\upsilon^{-m}$ where $w_{x}w_{1}(1) =m+1$. Let $w_{3}=w_{xw_{1}}$, that is, $w_{3}(1) =1$ and $xw_{1}w_{3}^{-1}\in\mathcal{C}_{0}$. From $\big( xw_{x}^{-1}\big) \big( w_{x}w_{1}w_{3}^{-1}\big)$ $\in\mathcal{C}_{0}$ it follows that $w_{x}w_{1}w_{3}^{-1}\in\langle w_{0}\rangle $, in particular $w_{x}w_{1}w_{3}^{-1}=w_{0}^{m}$ because $w_{x}w_{1}w_{3}^{-1}(1) =w_{x}w_{1}(1) =m+1=w_{0}^{m}(1) $. Thus $M(w_{2},xw_{1}) =\nu( w_{3}w_{2}) =\nu\big( w_{0}^{-m}w_{x}w_{1}w_{2}\big) =\upsilon^{m}\nu(w_{x}w_{1}w_{2})$, and $\upsilon^{m}=M(w_{1},x) ^{-1}$. This completes the proof.
\end{proof}

\begin{Corollary}
Suppose $w\in\mathcal{S}_{N}$ and $x\in\mathbb{T}_{\rm reg}^{N}$ then $M\big(w^{-1},xw\big) =M(w,x) ^{-1}$.
\end{Corollary}

\begin{proof}
Indeed $M\left( w^{-1},xw\right) M(w,x) =M\big(ww^{-1},x\big) =I$.
\end{proof}

We can now extend~$L(x) $ to all of $\mathbb{T}_{\rm reg}^{N}$ from its values on $\mathcal{C}_{0}$.

\begin{Definition}\label{DefL(x)T}For $x\in\mathbb{T}_{\rm reg}^{N}$ let%
\begin{gather*}
L(x) :=L\big( xw_{x}^{-1}\big) \tau(w_{x}) .
\end{gather*}
\end{Definition}

\begin{Proposition}\label{L(xw)M}For any $x\in\mathbb{T}_{\rm reg}^{N}$ and $w\in\mathcal{S}_{N}$
\begin{gather*}
L(xw) =M(w,x) L(x) \tau (w) .
\end{gather*}
\end{Proposition}

\begin{proof}
Let $w_{1}=w_{xw}$, that is, $w_{1}(1) =1$ and $xww_{1}^{-1}\in\mathcal{C}_{0}$, then by def\/inition $L(xw) =L\big(xww_{1}^{-1}\big) \tau(w_{1}) $ and $L\big( xw_{x}^{-1}\big) =L(x) \tau(w_{x}) ^{-1}$. Let $m=w_{x}w(1) -1$. Since $w_{x}ww_{1}^{-1}$ f\/ixes $\mathcal{C}_{0}$ and $w_{x}ww_{1}^{-1}(1) =w_{x}w(1) =m+1$ it follows that $w_{x}ww_{1}^{-1}=w_{0}^{m}$. Thus $w_{1}=w_{0}^{-m}w_{x}w$,
\begin{gather*}
L\big( xww_{1}^{-1}\big) \tau(w_{1}) =L\big(xw_{x}^{-1}w_{0}^{m}\big) \tau\big( w_{0}^{-m}w_{x}w\big)
=\upsilon^{-m}L\big( xw_{x}^{-1}\big) \upsilon^{m}\tau\big( w_{0}^{-m}w_{x}w\big) \\
\hphantom{L\big( xww_{1}^{-1}\big) \tau(w_{1})}{} =\upsilon^{-m}L\big( xw_{x}^{-1}\big) \tau(w_{x}w) =\upsilon^{-m}L(x) \tau(w)
\end{gather*}
and $M(w,x) =\nu(w_{x}w) =\upsilon^{-m}$.
\end{proof}

\subsection[The adjoint operation on Laurent polynomials and $L(x)$]{The adjoint operation on Laurent polynomials and $\boldsymbol{L(x)}$}

The purpose is to def\/ine an operation which agrees with taking complex conjugates of functions and Hermitian adjoints of matrix functions when restricted to $\mathbb{T}^{N}$, and which preserves analyticity. The parameter $\kappa$ is treated as real in this context even where it may be complex (to preserve analyticity in $\kappa$). For $x\in\mathbb{C}_{\times}^{N}$ def\/ine $\phi x:=\big( x_{1}^{-1},x_{2}^{-1},\ldots,x_{N}^{-1}\big) $, then
$\phi(xw) =(\phi x) w$ for all $w\in \mathcal{S}_{N}$.

\begin{Definition}\label{defadj}\quad
\begin{enumerate}\itemsep=0pt
\item[(1)] If $f(x) =\sum\limits_{\alpha\in\mathbb{Z}^{N}}c_{\alpha}x^{\alpha}$ is a Laurent polynomial then $f^{\ast}(x) :=\sum\limits_{\alpha\in\mathbb{Z}^{N}}\overline{c_{\alpha}}x^{-\alpha}$.
\item[(2)] If $f(x) =\sum\limits_{\alpha\in \mathbb{Z}^{N}}A_{\alpha}x^{\alpha}$ is a Laurent polynomial with matrix coef\/f\/icients then $f^{\ast}(x) :=\sum\limits_{\alpha \in\mathbb{Z}^{N}}A_{\alpha}^{\ast}x^{-\alpha}$.
\item[(3)] if $F(x) $ is a matrix-valued function analytic in an open subset $U$ of $\mathbb{C}_{\times}^{N}$ then $F^{\ast}(x) :=\overline{( F(\overline{\phi x})) }^{T}\!$ and $F^{\ast}$ is analytic on $\phi U$, that is, if $F(x)\! =\![ F_{ij}(x) ]_{i,j=1}^{N}$ then $F^{\ast}(x)\! =\! [ \overline{F_{ji}( \overline{\phi x}) }] _{i,j=1}^{N}\!$ (for example if $F_{12}(x) =c_{1}\kappa x_{1}x_{3}^{-1}+c_{2}x_{2}^{2}x_{3}^{-1}x_{4}^{-1}$ then $F_{21}^{\ast}(x) =\overline{c_{1}}\kappa x_{1}^{-1}x_{3}+\overline{c_{2}}x_{2}^{-2}x_{3}x_{4}$).
\end{enumerate}
\end{Definition}

Loosely speaking $F^{\ast}(x) $ is obtained by replacing $x$ by $\phi x$, conjugating the complex constants and transposing. The fundamental chamber $\mathcal{C}_{0}$ is mapped by $\phi$ onto $\big\{ \big( e^{\mathrm{i}\theta_{j}}\big) _{j=1}^{N}\colon \theta_{1}>\theta_{2}>\cdots$ $>\theta_{N}>\theta_{1}-2\pi\big\} $, again set-wise invariant under~$w_{0}$. Using $\frac{\mathrm{d}}{\mathrm{d}t}\big\{ f\big( \frac{1}{t}\big) \big\}
=-\frac{1}{t^{2}}\big( \frac{d}{dt}f\big) \big( \frac{1}{t}\big)$ we obtain the system
\begin{gather*}
\partial_{i}L(\phi x) =\kappa L(\phi x) \left\{\sum_{j\neq i}\frac{x_{j}}{x_{i}}\frac{\tau ( (i,j)) }{x_{i}-x_{j}}+\frac{\gamma}{x_{i}}\right\} , \qquad 1\leq i\leq N.
\end{gather*}
Transposing this system leads to (note $\tau(w) ^{T}=\tau (w) ^{\ast}=\tau\big( w^{-1}\big) $)
\begin{gather*}
\partial_{i}L(\phi x) ^{T}=\kappa\left\{ \sum_{j\neq i}\frac{x_{j}}{x_{i}}\frac{\tau((i,j)) }
{x_{i}-x_{j}}+\frac{\gamma}{x_{i}}\right\} L(\phi x)^{T}, \qquad 1\leq i\leq N.
\end{gather*}
Now use part (3) of Def\/inition \ref{defadj} and set up the system whose solution of
\begin{gather}
\partial_{i}L^{\ast}(x) =\kappa\left\{ \sum_{j\neq i}\frac{x_{j}}{x_{i}}\frac{\tau((i,j)) }{x_{i}-x_{j}}+\frac{\gamma}{x_{i}}\right\} L^{\ast}(x) , \qquad 1\leq
i\leq N.\label{L*sys}
\end{gather}
satisfying $L^{\ast}(x_{0}) =I$ is denoted by $L^{\ast}(x)$. The constants in the system are all real so replacing complex constants by their complex conjugates preserves solutions of the system. The ef\/fect is that $L(x) ^{\ast}$ agrees with the Hermitian adjoint of $L(x) $ for $x\in\mathcal{C}_{0}$ (for real $\kappa$). The goal here is to establish conditions on a constant Hermitian matrix $H$ so that $K(x) :=L^{\ast}(x) HL(x) $ has desirable properties, such as $K(xw) =\tau(w)^{-1}K(x) \tau(w) $ and $K(x) \geq0$ (i.e., positive def\/inite).

Similarly to the above $\tau((i,j)) L^{\ast}(x(i,j)) $ is also a solution of (\ref{L*sys}), implying that $\tau(w) L^{\ast}(xw) $ is a~solution for any $w\in\mathcal{S}_{N}$, the inductive step is
\begin{gather*}
\tau((i,j)) \tau(w) L( x(i,j) w) ^{\ast}=\tau((i,j)w) L(x(i,j)w) ^{\ast}.
\end{gather*} Also $L^{\ast}(x_{0}w_{0}) =L^{\ast}\big( \omega^{-1}x_{0}\big) =L^{\ast} (x_{0}) =I$ (thus there is a matrix $\widetilde{\Xi}$ such that
$\tau(w_{0}) L^{\ast}(xw_{0}) =L^{\ast} ( \phi x) \widetilde{\Xi}$ for all $x\in\mathcal{C}_{0}$, and $\widetilde{\Xi}=\tau(w_{0}) =\upsilon$. In analogy to~$L$ for $x\in\mathbb{T}_{\rm reg}^{N}$ and the same $w_{x}$ as above let $L ( \phi x_{0}) ^{T}=I$, $L(\phi x) ^{T}:=\tau ( w_{x}) ^{-1}L\big( \phi xw_{x}^{-1} \big) ^{T}$ (since $\phi xw_{x}^{-1}\in\phi\mathcal{C}_{0})$. Then $L(\phi xw) ^{T}=\tau(w) ^{-1}L(\phi x) ^{T}M(w,x)^{-1}$.

For any nonsingular constant matrix $C$ the function $CL(x) $ also satisf\/ies (\ref{Lsys}) and the function $K(x) :=L^{\ast}(x) C^{\ast}CL(x) $ satisf\/ies the system
\begin{gather}
x_{i}\partial_{i}K(x) =\kappa\sum_{j\neq i}\left\{ \frac{x_{j}}{x_{i}-x_{j}}\tau((i,j)) K(x) +K(x) \tau((i,j)) \frac{x_{i}
}{x_{i}-x_{j}}\right\} , \qquad 1\leq i\leq N.\label{Kdieq}
\end{gather}
This formulation can be slightly generalized by replacing $C^{\ast}C$ by a~Hermitian matrix~$H$ (not necessarily positive-def\/inite) without changing the equation.

For the purpose of realizing the form (\ref{admforms}) we want $K$ to satisfy $K(xw) =\tau(w) ^{-1}K(x)\tau(w) $, that is,
\begin{gather*}
K(xw) =\tau(w) ^{-1}L^{\ast}(x)M(w,x) ^{-1}HM(w,x) L(x) \tau(w) \\
\hphantom{K(xw)}{} =\tau(w) ^{-1}L^{\ast}(x) \upsilon^{m}H\upsilon^{-m}L(x) \tau(w)
\end{gather*}
(from Proposition \ref{L(xw)M}), where $m=w_{x}w(1) -1$. The condition is equivalent to
\begin{gather*}
\upsilon H=H\upsilon,
\end{gather*}
which is now added to the hypotheses, summarized here:

\begin{Condition}\label{hypoLH}
$L(x) $ is the solution of \eqref{Lsys} such that $L(x_{0}) =I$ and $L(x) =L\big( xw_{x}^{-1}\big) \tau(w_{x}) $ for $x\in\mathbb{T}_{\rm reg}^{N}$ where $w_{x}(1) =1$ and $xw_{x}^{-1}\in\mathcal{C}_{0}$; $L^{\ast}(x) $ is the solution of~\eqref{L*sys} satisfying $L^{\ast}(x_{0}) =I$, $K(x) =L^{\ast }(x) HL(x) $ is a solution of~\eqref{Kdieq} and~$H$ satisfies $H^{\ast}=H$, $\upsilon H=H\upsilon$.
\end{Condition}

\section{Integration by parts}\label{byparts}

In this section we establish the relation between the dif\/ferential system and the abstract relation $\langle x_{i}\mathcal{D}_{i}f,g\rangle
=\langle f,x_{i}\mathcal{D}_{i}g \rangle $ holding for $1\leq i\leq N$ and $f,g\in C^{1}\big( \mathbb{T}^{N};V_{\tau}\big) $. We demonstrate how close $L$ is to providing the desired inner product, by performing an integration-by-parts over an $\mathcal{S}_{N}$-invariant closed set $\subset\mathbb{T}_{\rm reg}^{N}$. Here $L(x) $ and $H$ satisfy the hypotheses listed in Condition~\ref{hypoLH} above. We use the identity $x_{i}\partial_{i}f^{\ast}(x) =-( x_{i}\partial_{i}f) ^{\ast}(x) $. For $\delta>0$ let
\begin{gather*}
\Omega_{\delta}:=\Big\{ x\in\mathbb{T}^{N}\colon \min\limits_{1\leq i<j\leq N}\vert x_{i}-x_{j}\vert \geq\delta\Big\} .
\end{gather*}
This set is invariant under $\mathcal{S}_{N}$ and $K(x) $ is bounded and smooth on it. Thus the following integrals exist.

\begin{Proposition}\label{xdfKg-fKxdg}Suppose $H$ satisfies Condition~{\rm \ref{hypoLH}} then for $f,g\in C^{1}\big( \mathbb{T}^{N};V_{\tau}\big) $ and $1\leq i\leq N$
\begin{gather*}
\int_{\Omega_{\delta}}\big\{ {-}( x_{i}\mathcal{D}_{i}f(x)) ^{\ast}K(x) g(x) +f(x) ^{\ast}K(x) x_{i}\mathcal{D}_{i}g(x)\big\} \mathrm{d}m(x) \\
\qquad {} =\int_{\Omega_{\delta}}x_{i}\partial_{i}\{ f(x) ^{\ast}K(x) g(x)\} \mathrm{d}m(x).
\end{gather*}
\end{Proposition}

\begin{proof}
By def\/inition
\begin{gather*}
x_{i}\mathcal{D}_{i}g(x) =x_{i}\partial_{i}g(x) +\kappa\sum_{j\neq i}\frac{x_{i}}{x_{i}-x_{j}}\tau((i,j))( g(x) -g( x(i,j))), \\
( x_{i}\mathcal{D}_{i}f(x) ) ^{\ast}
=-x_{i}\partial_{i}f(x) ^{\ast}+\kappa\sum_{j\neq i}\frac{x_{j}}{x_{j}-x_{i}}( f(x) ^{\ast}-f( x(i,j)) ^{\ast}) \tau((i,j)) .
\end{gather*}
Thus
\begin{gather}
 -( x_{i}\mathcal{D}_{i}f(x) ) ^{\ast}K(x) g(x) +f(x) ^{\ast}K(x)x_{i}\mathcal{D}_{i}g(x) \nonumber\\
\qquad{} =x_{i}\partial_{i}f(x) ^{\ast}+x_{i}\partial_{i}g(x) \nonumber\\
\qquad\quad{} +\kappa f(x) ^{\ast}\sum_{j\neq i}\left\{ \frac{x_{j}}{x_{i}-x_{j}}\tau((i,j)) K(x)+K(x) \tau((i,j)) \frac{x_{i}}{x_{i}-x_{j}}\right\} g(x) \nonumber\\
\qquad\quad{} -\kappa\sum_{j\neq i}\frac{1}{x_{i}-x_{j}}\big\{ x_{j}f( x(i,j)) ^{\ast}\tau((i,j)) K(x) g(x) +x_{i}f(x) ^{\ast}K(x) \tau((i,j)) g( x(i,j))\big\} \nonumber\\
 \qquad{} =x_{i}\partial_{i}\{ f(x) ^{\ast}K(x)g(x) \} \label{dfKg}\\
 \qquad\quad{} -\kappa\sum_{j\neq i}\frac{1}{x_{i}-x_{j}}\big\{ x_{j}f ( x (i,j )) ^{\ast}\tau((i,j)) K(x) g(x) +x_{i}f(x) ^{\ast}K(x) \tau((i,j)) g( x(i,j))\big\}.\nonumber
\end{gather}
For each pair $\{i,j\} $ the term inside $\{\cdot\} $ is invariant under $x\mapsto x(i,j)$, because $K(x(i,j)) =\tau( (i,j)) K(x) \tau((i,j)) $, and
$x_{i}-x_{j}$ changes sign under this transformation. Thus
\begin{gather*}
\int_{\Omega_{\delta}}\frac{x_{j}f(x(i,j)) ^{\ast}\tau((i,j)) K(x) g(x) +x_{i}f(x) ^{\ast}K(x) \tau((i,j)) g(x(i,j)) }{x_{i}-x_{j}}\mathrm{d}m(x) =0
\end{gather*}
for each $j\neq i$ because $\Omega_{\delta}$ and $\mathrm{d}m$ are invariant under $(i,j)$.
\end{proof}

Observe the value of $\kappa$ is not involved in the proof. Since $x_{j}\partial_{j}=-\mathrm{i}\frac{\partial}{\partial\theta_{j}}$ when $x_{j}=e^{\mathrm{i}\theta_{j}}$ and $\mathrm{d}m(x) = (2\pi) ^{-N}\mathrm{d}\theta_{1}\cdots\mathrm{d}\theta_{N}$ one step of integration can be directly evaluated. Consider the case $i=N$ and for a f\/ixed $(N-1) $-tuple $( \theta_{1},\ldots,\theta_{N-1}) $ with $\theta_{1}<\theta_{2}<\cdots <\theta_{N-1}<\theta_{1}+2\pi$ such that $\big\vert e^{\mathrm{i}\theta_{j}}-e^{\mathrm{i}\theta_{i}}\big\vert \geq\delta$ the integral over $\theta_{N}$ is over a union of closed intervals. These are the complement of $\bigcup\limits_{1\leq j\leq N-1} \{ \theta\colon \theta_{j}-\delta^{\prime}<\theta<\theta_{j}
+\delta^{\prime}\} $ in the circle, where $\sin\frac{\delta^{\prime}}{2}=\frac{\delta}{2}$. This results in an alternating sum of values of $f^{\ast}Kg$ at the end-points of the closed intervals. Analyzing the resulting integral (over $( \theta_{1},\ldots,\theta_{N-1}) $ with respect to $\mathrm{d}\theta_{1} \cdots \mathrm{d}\theta_{N-1}$) is one of the key steps in showing that a given $K$ provides the desired inner product. In other parts of this paper we f\/ind that $H$ must satisfy another commuting relation.

\section{Local power series near the singular set}\label{locps}

In this section assume $\kappa\notin\mathbb{Z+}\frac{1}{2}$. We consider the
system (\ref{Lsys}) in a neighborhood of the face $ \{ x\colon x_{N-1}=x_{N} \} $ of $\mathcal{C}_{0}$. We use a coordinate system which treats the singularity in a simple way. For a more concise notation def\/ine
\begin{gather*}
x(u,z) =( x_{1},x_{2},\ldots,x_{N-2},u-z,u+z)\in\mathbb{C}_{\times}^{N}%
\end{gather*}

We consider the system in terms of the variable $x(u,z) $ subject to the conditions that the points $x_{1},x_{2},\ldots,x_{N-2},u$ are pairwise distinct and $\vert z\vert <\min\limits_{1\leq j\leq N-2}\vert x_{j}-u \vert $, also $\vert z\vert <\vert u\vert$, $\operatorname{Im}\frac{z}{u}>0$ (these conditions imply
$\arg(u-z) <\arg(u+z) $). This allows power series
expansions in $z$.

For $z_{1},z_{2}\in\mathbb{C}_{\times}$ let
\begin{gather*}
\rho(z_{1},z_{2}) :=\left[
\begin{matrix}
z_{1}I_{m_{\tau}} & O\\
O & z_{2}I_{n_{\tau}-m_{\tau}} \end{matrix}
\right] .
\end{gather*}
Let $\sigma:=\tau((N-1,N)) =\rho( -1,1) $. We analyze the local solution $L( x(u-z,u+z)) $ with an initial condition specif\/ied later. We obtain the dif\/ferential system (using $\partial_{z}:=\frac{\partial}{\partial z}$, $\partial_{u}:=\frac{\partial}{\partial u}$)
\begin{gather*}
\partial_{z}L(x) =\partial_{N}L-\partial_{N-1}L\\
\hphantom{\partial_{z}L(x)}{} =\kappa L\left\{ \sum_{j=1}^{N-2}\left( \frac{\tau((j,N)) }{u-x_{j}+z}-\frac{\tau( (j,N-1)) }{u-x_{j}-z}\right) +\frac{\tau(N-1,N) }{z}-\frac{\gamma}{u+z}I+\frac{\gamma}{u-z}I\right\} ,
\\
\partial_{u}L(x) =\partial_{N}L+\partial_{N-1}L\\
\hphantom{\partial_{u}L(x)}{} =\kappa L\left\{ \sum_{j=1}^{N-2}\left( \frac{\tau((j,N)) }{u-x_{j}+z}+\frac{\tau( (j,N-1)) }{u-x_{j}-z}\right) -\frac{\gamma}{u+z}I-\frac{\gamma}{u-z}I\right\} ,
\\
\partial_{j}L(x) =\kappa L(x) \left\{\sum_{i=1,i\neq j}^{N-2}\frac{\tau((i,j)) }{x_{j}-x_{i}}-\frac{\gamma}{x_{j}}I+\frac{\tau( (j,N-1)) }{x_{j}-u+z}+\frac{\tau( ( j,N)) }{x_{j}-u-z}\right\} , \\
\hphantom{\partial_{j}L(x) =}{} 1\leq j\leq N-2.
\end{gather*}

Using the expansion $\frac{1}{t-z}=\sum\limits_{n=0}^{\infty}\frac{z^{n}}{t^{n+1}}$ for $\vert z\vert <\vert t\vert $ we let
\begin{gather*}
\beta_{n}(x(u,0)) :=\sum_{j=1}^{N-2}\frac {\tau((j,N)) }{(u-x_{j}) ^{n+1}}
\end{gather*}
for $n=0,1,2,\ldots$and express the equations as (since $\sigma\tau (( j,N)) \sigma=\tau((j,N-1))$)
\begin{gather*}
\partial_{z}L(x) =\kappa L(x) \left\{
\sum_{n=0}^{\infty}\big\{ (-1) ^{n}\beta_{n}( x(u,0)) -\sigma\beta_{n}(x(u,0)) \sigma\big\} z^{n}+\frac{\sigma}{z}-\frac{\gamma}{u+z}I+\frac{\gamma}{u-z}I\right\}, \\
\partial_{u}L(x) =\kappa L(x) \left\{\sum_{n=0}^{\infty}\big\{ (-1) ^{n}\beta_{n}( x(u,0)) +\sigma\beta_{n}(x(u,0))\sigma\big\} z^{n}-\frac{\gamma}{u+z}I-\frac{\gamma}{u-z}I\right\} ,\\
\partial_{j}L(x) =\kappa L(x) \left\{\sum_{i=1,i\neq j}^{N-2}\frac{\tau((i,j)) }{x_{j}-x_{i}}-\frac{\gamma}{x_{j}}I-\sum_{n=0}^{\infty}\frac{\tau(
(j,N-1)) +(-1) ^{n}\tau((j,N)) }{(u-x_{j}) ^{n+1}}z^{n}\right\} ,\\
\hphantom{\partial_{j}L(x) =}{} 1\leq j\leq N-2.
\end{gather*}
Set
\begin{gather*}
B_{n}(x(u,0)) =(-1) ^{n}\beta _{n}(x(u,0)) -\sigma\beta_{n}( x(u,0)) \sigma, \qquad n=0,1,2,\ldots.
\end{gather*}
Note $\sigma B_{n}x(u,0) \sigma=(-1) ^{n+1}B_{n}(x(u,0)) $. Suggested by the relation
\begin{gather*}
\frac{\partial}{\partial z}\rho \big( z^{-\kappa},z^{\kappa}\big)=\frac{\kappa}{z}\rho\big({}-z^{-\kappa},z^{\kappa}\big)=\frac{\kappa }{z}\rho\big(z^{-\kappa},z^{\kappa}\big) \sigma
\end{gather*} we look for a solution of the form
\begin{gather}
L(x) =\left( \big( u^{2}-z^{2}\big) \prod_{j=1}^{N-2}x_{j}\right) ^{-\gamma\kappa}\rho\big( z^{-\kappa},z^{\kappa}\big)
\sum_{n=0}^{\infty}\alpha_{n}(x(u,0))z^{n},\label{Lzseries}
\end{gather}
where each $a_{n}(x(u,0)) $ is matrix-valued and analytic in $x(u,0) $, and the initial condition is $\alpha _{0}\big( x^{(0) }\big) =I$, where $x^{(0) }$
is a base point, chosen as $\big( 1,\omega,\omega^{2},\ldots,\omega^{N-3},\omega^{-3/2},\omega^{-3/2}\big) $ (that is, $u=\omega^{-3/2}$, $z=0$),
where $\omega:=e^{2\pi\mathrm{i}/N}$. Implicitly restrict $(x_{1},\ldots,x_{N-1},u) $ to a simply connected open subset of $\mathbb{C}_{\rm reg}^{N-1}$ containing $\big( 1,\omega,\omega^{2},\ldots,\omega^{N-3},\omega^{-3/2}\big)$. Substitute~(\ref{Lzseries}) in the~$\partial_{z}$ equation (suppressing the $x(u,0) $ argument in the $\alpha_{n}$'s)
\begin{gather*}
\partial_{z}L =\kappa\gamma\left( \frac{1}{u+z}-\frac{1}{u-z}\right)\left( \big( u^{2}-z^{2}\big) \prod_{j=1}^{N-2}x_{j}\right)
^{-\gamma\kappa}\rho\big( z^{-\kappa},z^{\kappa}\big) \sum_{n=0}^{\infty}\alpha_{n}z^{n}\\
\hphantom{\partial_{z}L =}{} +\left( \big( u^{2}-z^{2}\big) \prod_{j=1}^{N-2}x_{j}\right)^{-\gamma\kappa}\frac{\kappa}{z}\rho\big( z^{-\kappa},z^{\kappa}\big) \sigma\sum_{n=0}^{\infty}\alpha_{n}z^{n}\\
\hphantom{\partial_{z}L =}{} +\left( \big( u^{2}-z^{2}\big) \prod_{j=1}^{N-2}x_{j}\right)^{-\gamma\kappa}\rho\big( z^{-\kappa},z^{\kappa}\big) \sum_{n=1}^{\infty}n\alpha_{n}z^{n-1}\\
\hphantom{\partial_{z}L}{}
 =\kappa\left( \big( u^{2}-z^{2}\big) \prod_{j=1}^{N-2}x_{j}\right) ^{-\gamma\kappa}\rho\big( z^{-\kappa},z^{\kappa}\big) \\
\hphantom{\partial_{z}L =}{}
 \times\sum_{n=0}^{\infty}\alpha_{n}z^{n}\left\{ \sum_{m=0}^{\infty}B_{m}(u) z^{m}+\frac{\sigma}{z}-\gamma\left( \frac{1}{u+z}-\frac{1}{u-z}\right) \right\} ,
\end{gather*}
which simplif\/ies to
\begin{gather}
\frac{\kappa}{z}\sum_{n=0}^{\infty}( \sigma\alpha_{n}-\alpha_{n}\sigma ) z^{n}+\sum_{n=1}^{\infty}n\alpha_{n}z^{n-1}=\kappa\sum
_{n=0}^{\infty}\alpha_{n}z^{n}\sum_{m=0}^{\infty}B_{m} ( x (u,0 ) ) z^{m}.\label{eqnDz}
\end{gather}
The equations for $\partial_{u}$ and $\partial_{j}$ simplify to
\begin{gather*}
 \left( \big( u^{2}-z^{2}\big) \prod_{j=1}^{N-2}x_{j}\right) ^{-\gamma\kappa}\left\{ \sum_{n=0}^{\infty}\partial_{u}\alpha_{n}z^{n}
-\kappa\gamma\left( \frac{1}{u+v}+\frac{1}{u-v}\right) \sum_{n=0}^{\infty }\alpha_{n}z^{n}\right\} \\
\qquad{} =\kappa\left( \big( u^{2}-z^{2}\big) \prod_{j=1}^{N-2}x_{j}\right)^{-\gamma\kappa}\sum_{n=0}^{\infty}\alpha_{n}z^{n}\\
\qquad\quad{} \times\left\{ \sum_{m=0}^{\infty}\big\{ (-1) ^{m}\beta
_{m}(x(u,0)) +\sigma\beta_{m}( x(u,0)) \sigma\big\} z^{m}-\frac{\gamma}{u+v}I-\frac{\gamma}{u-v}I\right\} ,
\end{gather*}
leading to (with $1\leq j\leq N-2$)
\begin{gather*}
\sum_{n=0}^{\infty}\partial_{u}\alpha_{n}(x(u,0))z^{n} =\kappa\sum_{n=0}^{\infty}\alpha_{n}( x(u,0)) z^{n}\sum_{m=0}^{\infty}\big\{ (-1) ^{m}\beta_{m}(x(u,0)) +\sigma\beta_{m} ( x (u,0)) \sigma\big\} z^{m},\\
\sum_{n=0}^{\infty}\partial_{j}\alpha_{n}(x(u,0))z^{n} =\kappa\sum_{n=0}^{\infty}\alpha_{n}( x(u,0)) z^{n}\\
\hphantom{\sum_{n=0}^{\infty}\partial_{j}\alpha_{n}(x(u,0))z^{n} =}{}
\times\left\{ \sum_{i=1,i\notin j}^{N-2}\frac{\tau( (i,j)) }{x_{j}-x_{i}}-\sum_{m=0}^{\infty}\frac{\tau((j,N-1)) +(-1) ^{m}\tau((j,N)) }{(u-x_{j}) ^{m+1}}z^{m}\right\}.
\end{gather*}
We only need the equations for $\alpha_{0}( x(u,0)) $ (that is, the coef\/f\/icient of $z^{0}$) to initialize the~$\partial_{z}$ equation (this is valid because the system is Frobenius integrable):
\begin{gather}
\partial_{u}\alpha_{0}(x(u,0)) =\kappa \alpha_{0}(x(u,0)) \big\{ \beta_{0}(x(u,0)) +\sigma\beta_{0}( x(u,0)) \sigma\big\} ,\label{dua0x}\\
\partial_{j}\alpha_{0}(x(u,0)) =\kappa \alpha_{0}(x(u,0)) \left\{ \sum_{i=1,i\notin j}^{N-2}\frac{\tau((i,j)) }{x_{j}-x_{i}} -\frac{\tau((j,N-1)) +\tau((
j,N)) }{(u-x_{j}) }\right\} ,\nonumber\\
\hphantom{\partial_{j}\alpha_{0}(x(u,0)) =}{} 2\leq j\leq N-2.\nonumber
\end{gather}

\begin{Lemma}
$\sigma\alpha_{0}(x(u,0)) \sigma=\alpha_{0}(x(u,0)) $ and $\alpha_{0}( x(u,0)) $ is invertible.
\end{Lemma}

\begin{proof}
By hypothesis $\alpha_{0}\big( x^{(0) }\big) =I$. The right hand sides of the system are invariant under the transformation $Q\mapsto \sigma Q\sigma$ thus $\alpha_{0}(x(u,0)) $ and $\sigma\alpha_{0}(x(u,0)) \sigma$ satisfy the same system. They agree at the base-point $x^{(0) }$, hence everywhere in the domain. By Jacobi's identity the determinant satisf\/ies (where $\lambda:=\operatorname{tr}(\sigma) =n_{\tau}-2m_{\tau}$)
\begin{gather*}
\partial_{u}\det\alpha_{0}(x(u,0)) =\kappa\det\alpha_{0}(x(u,0)) \operatorname{tr} \{ \beta_{0}(x(u,0)) +\sigma\beta
_{0}(x(u,0)) \sigma \} \\
\hphantom{\partial_{u}\det\alpha_{0}(x(u,0))}{} =2\kappa\det\alpha_{0}(x(u,0)) \mathrm{\lambda }\sum_{j=1}^{N-2}\frac{1}{(u-x_{j}) },\\
\partial_{j}\det\alpha_{0}(x(u,0)) =\kappa\lambda\det\alpha_{0}(x(u,0)) \left\{\sum_{i=1,i\notin j}^{N-2}\frac{1}{x_{j}-x_{i}}-\frac{2}{u-x_{j}}\right\},\qquad 1\leq j\leq N-2,\\
\det\alpha_{0}(x(u,0)) =\prod_{1\leq i<j\leq N-2}\left( \frac{x_{i}-x_{j}}{x_{i}^{(0) }-x_{j}^{(0) }}\right) ^{\lambda\kappa}\prod_{i=1}^{N-2}\left( \frac{x_{i}-u}{x_{i}^{(0) }-x_{N-1}^{(0) }}\right)^{2\lambda\kappa},
\end{gather*}
the multiplicative constant follows from $\alpha_{0}\big( x^{(0) }\big) =I$. Thus $\alpha_{0}(x(u,0))$ is nonsingular in its domain.
\end{proof}

We turn to the inductive def\/inition of $\{ \alpha_{n}( x(u,0) ) \colon n\geq1\} $.

In terms of the block decomposition $( m_{\tau}+( n_{\tau}-m_{\tau})) \times( m_{\tau}+( n_{\tau}-m_{\tau})) $ (henceforth called the $\sigma$-\emph{block
decomposition}) of a matrix
\begin{gather*}
\alpha=\left[
\begin{matrix}
\alpha_{11} & \alpha_{12}\\
\alpha_{21} & \alpha_{22}%
\end{matrix}
\right]
\end{gather*}
$\sigma\alpha\sigma=\alpha$ if and only if $\alpha_{12}=O=\alpha_{21}$ and $\sigma\alpha\sigma=-\alpha$ if and only if $\alpha_{11}=O=\alpha_{22}$. Write the $\sigma$-block decomposition of $\alpha_{n}(u) $ as
\begin{gather*}
\alpha_{n}=\left[
\begin{matrix}
\alpha_{n,11} & \alpha_{n,12}\\
\alpha_{n,21} & \alpha_{n,22}%
\end{matrix}
\right]
\end{gather*}
then the coef\/f\/icient of $z^{n-1}$ on the left side of equation (\ref{eqnDz}) is
\begin{gather*}
\kappa ( \sigma\alpha_{n}-\alpha_{n}\sigma ) +n\alpha_{n}=\left[
\begin{matrix}
n\alpha_{n,11} & (n-2\kappa) \alpha_{n,12}\\
( n+2\kappa) \alpha_{n,21} & n\alpha_{n,22}
\end{matrix}
\right] ,
\end{gather*}
and on the right side it is
\begin{gather*}
\kappa S_{n}(x(u,0)) :=\kappa\sum_{i=0}^{n-1}\alpha_{n-1-i}B_{i}(x(u,0)) ,
\end{gather*}
for $n\geq1$. Arguing inductively suppose $\sigma\alpha_{m}\sigma= (-1 ) ^{m}\alpha_{m}$ for $0\leq m\leq n$, then $\sigma S_{n}\sigma
=\sum\limits_{i=0}^{n-1} ( \sigma\alpha_{n-1-i}\sigma ) ( \sigma B_{i}\sigma ) =\sum\limits_{i=0}^{n-1}(-1) ^{n-1-i+i-1}\alpha_{n-1-i}B_{i}$ and thus $\sigma S_{n}(u) \sigma= (-1) ^{n}S_{n}(u) $. In terms of the $\sigma$-block decomposition $\left[
\begin{matrix}
S_{n,11} & S_{n,12}\\
S_{n,21} & S_{n,22}%
\end{matrix}
\right] $ of $S_{n}(x(u,0)) $ this condition implies $S_{n,12}=O=S_{n,21}$ when $n$ is even, and $S_{n,11}=O=S_{n,22}$ when~$n$ is odd. This implies (for $n=1,2,3,\ldots$)
\begin{gather}
\alpha_{2n}(x(u,0)) =\frac{\kappa}{2n}S_{2n}(x(u,0)) ,\label{recurS}\\
\alpha_{2n-1} ( x ( u,0 )) =\rho\left( \frac{\kappa}{2n-1-2\kappa},\frac{\kappa}{2n-1+2\kappa}\right) S_{2n-1}(x(u,0)) ,\nonumber
\end{gather}
and thus $\sigma\alpha_{n}(x(u,0)) \sigma= ( -1) ^{n}\alpha_{n}(x(u,0)) $. In particular
\begin{gather*}
\alpha_{1}(x(u,0)) =\rho\left( \frac{\kappa }{1-2\kappa},\frac{\kappa}{1+2\kappa}\right) \alpha_{0}( x(u,0)) B_{0}(x(u,0)) ,
\end{gather*}
and all the coef\/f\/icients are determined; by hypothesis $\kappa\notin \mathbb{Z+}\frac{1}{2}$ and the denominators are of the form $2m+1\pm2\kappa$.

Henceforth denote the series (\ref{Lzseries}), solving (\ref{Lsys}) with the normalization $\alpha_{0}\big( x^{(0) }\big) =I$ by~$L_{1}(x) $. It is def\/ined for all $x(u,z) \in\mathcal{C}_{0}$ subject to $\vert z\vert <\min\limits_{1\leq j\leq N-2}\vert x_{j}-u \vert $, also $\vert z\vert <\vert u\vert$, $\operatorname{Im}\frac{z}{u}>0$. The radius of convergence depends on $x(u,0) $. Return to using $L ( x) $ to denote the solution from Def\/inition~\ref{DefL(x)T} (on all of $\mathbb{T}_{\rm reg}^{N}$ and $L(x_{0}) =I$). In terms of $x (u,z ) $ the point $x_{0}$ corresponds to $u=\frac{1}{2} \big(\omega^{-1}+\omega^{-2} \big) $, $z=\frac{1}{2}\big( \omega^{-1}-\omega^{-2}\big) $, $x(u,z) =\big( 1,\omega,\ldots,\omega^{N-3},u-z,u+z\big) $, then $\min\limits_{1\leq j\leq N-2} \vert u-x_{j}\vert =\sin\frac{\pi}{N}\big( 5+4\cos\frac{2\pi}{N}\big) $ and $\vert z\vert =\sin\frac{\pi}{N}$ (also $\frac{z}{u}=\mathrm{i}\tan\frac{\pi}{N}$) and $x_{0}$ is in the domain of convergence of the series $L_{1}(x) $. Thus the relation $L_{1} (x ) =L_{1}(x_{0}) L(x) $ holds in the domain of $L_{1}$ in $\mathcal{C}_{0}$. This implies the important fact that $L_{1}(x_{0}) $ is an analytic function of~$\kappa$, to be
exploited in Section~\ref{anlcmat}.

\subsection{Behavior on boundary}

The term $\rho ( z^{-\kappa},z^{\kappa} ) $ implies that $L_{1}(x) $ is not continuous at $z=0$, that is, on the boundary $ \{ x\colon x_{N-1}=x_{N}\}$. However there may be a weak type of continuity, specif\/ically
\begin{gather*}
\lim\limits_{x_{N-1}-x_{N}\rightarrow0} (K(x) -K( x(N-1,N))) =0.
\end{gather*}
With the aim of expressing the desired $K(x) $ in the form $L(x) ^{\ast}C^{\ast}CL(x) $ (and $C$ is unknown at this stage) we consider $CL(x) $ in series form, that is $CL_{1}(x_{0}) ^{-1}L_{1}(x) $ (recall $\det L(x) \neq0$ in $\mathcal{C}_{0}$). We analyze the ef\/fect of~$C$ on the weak continuity condition. Denote $C^{\prime}:=CL_{1}( x_{0}) ^{-1}$.

From Proposition \ref{L(xw)M} $L(x(N-1,N)) =\nu((N-1,N)) L(x) \tau( (N-1,N)) =L(x) \sigma$, because $w(1) =1$ for $w=(N-1,N) $, [for the special case $N=3$, $\tau=(2,1) $, $\mathbb{T}_{\rm reg}^{3}$ has two components and we def\/ine $L(x) =L( x(2,3)) \sigma$ for the component $\neq\mathcal{C}_{0}$]. By use of $x(u,z)(N-1,N) =x(u,-z) $ it follows that
\begin{gather*}
CL( x(u,z) (N-1,N)) =CL(x(u,z) ) \sigma =C^{\prime}\left( x_{N}x_{N-1}\right) ^{-\gamma\kappa}\rho\big(z^{-\kappa},z^{\kappa}\big) \sum_{n=0}^{\infty}\alpha_{n}(u)z^{n}\sigma\\
\hphantom{CL( x(u,z) (N-1,N))}{} =C^{\prime}\sigma ( x_{N}x_{N-1} ) ^{-\gamma\kappa}\rho\big(z^{-\kappa},z^{\kappa}\big) \sum_{n=0}^{\infty}\alpha_{n}(u)
(-1) ^{n}z^{n},
\end{gather*}
because $\sigma\alpha_{n}(u) \sigma=(-1)^{n}\alpha_{n}(u) $ and $\sigma=\rho(-1,1) $.

Recall $L^{\ast}(x) $ is def\/ined as $L(\phi x)^{T}$ with complex constants replaced by their conjugates. Then $\phi x(u,z) =\big( x_{1}^{-1},x_{2}^{-1},\ldots,x_{N-2}^{-1},\frac{1}{u-z},\frac{1}{u+z}\big)$. To compute $L_{1} ( \phi x (u,z )) $ replace $u$ by $u^{\prime}=\frac{u}{(u+z)
(u-z) }$ and replace $z$ by $z^{\prime}=-\frac{z}{(u+z) (u-z) }$. When restricted to the torus $u^{\prime}=\frac{1}{2}\big( \frac{1}{x_{N-1}}+\frac{1}{x_{N}}\big) =\overline{u}$ and $z^{\prime}=\frac{1}{2}\big( \frac{1}{x_{N-1}}-\frac{1}{x_{N}}\big) =\overline{z}$. The terms $\beta_{n}(u) :=\sum\limits_{j=1}^{N-2}\frac{\tau((j,N)) }{( u-x_{j}) ^{n+1}}$ in the intermediate formulae for $L_{1}$ are replaced by their complex conjugates when $x(u,z) \in\mathbb{T}^{N}$. Similarly $\widetilde{\beta}_{k}:=\sum\limits_{m=0}^{\infty}\sum\limits_{j=1}^{N-2}\frac{\tau((j,N)) }{(u_{0}-x_{j}) ^{k+1}}$ transforms to $\overline{( \widetilde{\beta}_{k}) }$ because the constant $u_{0}$ is conjugated. Thus for $x(u,z) \in\mathbb{T}_{\rm reg}^{N}$
\begin{gather*}
L_{1}(x(u,z)) ^{\ast}=\sum_{m=0}^{\infty}\alpha_{m}(u) ^{\ast}\overline{z}^{m}\rho \big( \overline {z}^{-\kappa},\overline{z}^{\kappa}\big) C^{^{\prime}\ast} (\overline{x_{N}x_{N-1}}) ^{-\gamma\kappa};
\end{gather*}
$\alpha_{m}(u) ^{\ast}$ denotes the adjoint of the matrix $\alpha_{m}(u) $. Then
\begin{gather*}
L_{1}( x(u,z) (N-1,N)) ^{\ast}=\sum_{m=0}^{\infty}(-1) ^{m}\alpha_{m}(u) ^{\ast
}\overline{z}^{m}\rho( \overline{z}^{-\kappa},\overline{z}^{\kappa}) \sigma C^{^{\prime}\ast}( \overline{x_{N}x_{N-1}})
^{-\gamma\kappa}.
\end{gather*}
Furthermore (recall $K(x(N-1,N)) =\sigma K(x) \sigma$ by def\/inition)
\begin{gather*}
K(x(u,z)) =\sum_{m,n=0}^{\infty}\overline
{z}^{m}z^{n}\alpha_{m}(u) ^{\ast}\rho\big( \overline
{z}^{-\kappa},\overline{z}^{\kappa}\big) C^{\prime\ast}C^{\prime}%
\rho\big( z^{-\kappa},z^{\kappa}\big) \alpha_{n}(u) ,\\
K(x(u,-z)) =\sum_{m,n=0}^{\infty}\overline
{z}^{m}z^{n}(-1) ^{m+n}\alpha_{m}(u) ^{\ast}\rho\big( \overline{z}^{-\kappa},\overline{z}^{\kappa}\big) \sigma
C^{\prime\ast}C^{\prime}\sigma\rho\big( z^{-\kappa},z^{\kappa}\big)\alpha_{n}(u) .
\end{gather*}
The term of lowest order in $z$ in $K(x(u,z)) -K( x(u,-z)) $ is
\begin{gather*}
\alpha_{0}(u) ^{\ast}\rho\big( \overline{z}^{-\kappa
},\overline{z}^{\kappa}\big) \big\{ C^{\prime\ast}C^{\prime}-\sigma
C^{\prime\ast}C^{\prime}\sigma\big\} \rho\big( z^{-\kappa},z^{\kappa}\big) \alpha_{0}(u) .
\end{gather*}
In terms of the $\sigma$-block decomposition, with%
\begin{gather*}
C^{\prime\ast}C^{\prime}=
\begin{bmatrix}
c_{11} & c_{12}\\
c_{12}^{\ast} & c_{22}
\end{bmatrix}
,\qquad \alpha_{0}(u) =
\begin{bmatrix}
a_{11}(u) & O\\
O & a_{22}(u)
\end{bmatrix}
\end{gather*}
the expression equals
\begin{gather*}
2
\begin{bmatrix}
O & \left( \dfrac{z}{\overline{z}}\right) ^{\kappa}a_{11}(u)
^{\ast}c_{12}a_{22}(u) \\
\left( \dfrac{\overline{z}}{z}\right) ^{\kappa}a_{22}(u)
^{\ast}c_{12}^{\ast}a_{11}(u) & O
\end{bmatrix},
\end{gather*}
which tends to zero as $z\rightarrow0$ if and only if $c_{12}=0$, that is, $\sigma C^{\ast}C\sigma=C^{\ast}C$.

\begin{Proposition}
Suppose $C^{\prime\ast}C^{\prime}$ commutes with $\sigma$ then
\begin{gather*}
K(x(u,z)) -K( x(u,z) (N-1,N)) =O\big( \vert z\vert ^{1-2\vert\kappa\vert }\big).
\end{gather*}
\end{Proposition}

\begin{proof}
The hypothesis implies $C^{\prime\ast}C^{\prime}$ commutes with $\rho(z^{-\kappa},z^{\kappa}) $, thus
\begin{gather*}
K(x(u,z)) -K ( x(u,z)(N-1,N)) \\
\qquad{} =\sum_{m,n=0}^{\infty}\overline{z}^{m}z^{n}\big( 1-(-1)^{m+n}\big) \alpha_{m}(u) ^{\ast}\rho\big(\vert z \vert ^{-2\kappa},\vert z\vert ^{2\kappa}\big)C^{\prime\ast}C^{\prime}\alpha_{n}(u) \\
\qquad{} =2z\alpha_{0}(u) ^{\ast}\rho\big( \vert z\vert^{-2\kappa},\vert z\vert ^{2\kappa}\big) C^{\prime\ast}
C^{\prime}\alpha_{1}(u) +2\overline{z}\alpha_{1}(u) ^{\ast}\rho\big( \vert z\vert ^{-2\kappa}, \vert z \vert ^{2\kappa}\big) C^{^{\prime}\ast}C^{\prime}\alpha_{0} (u) \\
\qquad\quad{} +\sum_{m+n\geq2}^{\infty}\overline{z}^{m}z^{n}\big( 1-(-1)^{m+n}\big) \alpha_{m}(u) ^{\ast}\rho\big(\vert
z \vert ^{-2\kappa},\vert z\vert ^{2\kappa}\big)C^{\prime\ast}C^{\prime}\alpha_{n}(u) .
\end{gather*}
The dominant terms come from $m=0,n=1$ and $m=1,n=0$; both of order $O\big(\vert z\vert ^{1-2\vert \kappa\vert }\big) $.
\end{proof}

We will see later for purpose of integration by parts, that the change in $K$ between the points $\big( x_{1},\ldots,x_{N-2},e^{\mathrm{i}\theta}, e^{\mathrm{i} ( \theta-\varepsilon ) }\big) $ and $\big(x_{1},\ldots,x_{N-2},e^{\mathrm{i}\theta},e^{\mathrm{i}( \theta+\varepsilon) }\big) $ is a key part of the analysis; this uses the relation $K\big( \big( x_{1},\ldots,x_{N-2},e^{\mathrm{i}\theta },e^{\mathrm{i}( \theta-\varepsilon) }\big) \big) =\sigma
K\big( \big( x_{1},\ldots,x_{N-2},e^{\mathrm{i} ( \theta -\varepsilon ) },e^{\mathrm{i}\theta}\big) \big) \sigma$.

\section{Bounds}\label{bnds}

In this section we derive bounds on $L(x) $ of global and local type. Throughout we adopt the normalization $L(x_{0}) =I$. The operator norm on $n_{\tau}\times n_{\tau}$ complex matrices is def\/ined by $\Vert M\Vert =\sup\{ \vert Mv\vert \colon \vert v \vert =1 \} $.

\begin{Theorem}\label{Lbnd}There is a constant $c$ depending on $\kappa$ such that $\Vert L(x) \Vert \leq c\prod\limits_{1\leq i<j\leq N}\vert x_{i}-x_{j}\vert ^{-\vert \kappa\vert }$ for each $x\in\mathbb{T}_{\rm reg}^{N}$.
\end{Theorem}

The proof is a series of steps starting with a general result which applies to matrix functions satisfying a linear dif\/ferential equation in one variable.

\begin{Lemma}\label{bdsM}Suppose $M(0) =I$, $\frac{d}{dt}M(t)
=M(t) F(t) $ and $ \Vert F(t) \Vert \leq f(t) $ for $0\leq t\leq1$ then $ \Vert M(t) -I\Vert \leq\exp\int_{0}^{t}f(s)\mathrm{d}s-1$ and $\Vert M(1)\Vert \leq\exp\int_{0}^{1}f(s) \mathrm{d}s$.
\end{Lemma}

\begin{proof}[{Proof from \cite[Theorem~7.1.11]{Stoer/Bulirsch1980}}] Let $\ell(t) := \Vert M(t) -I \Vert $ then the equation $M (t) -I=\int_{0}^{t}M(s) F(s) \mathrm{d}s$ and the inequalities $\Vert M(t) \Vert \leq\Vert M(t) -I \Vert + \Vert I \Vert $ (and $ \Vert I\Vert =1$) imply that $\ell(t) \leq\int_{0}^{t}(\ell(s) +1) f(s) \mathrm{d}s$. Def\/ine dif\/ferentiable functions~$b(t) $ and~$h(t) $ by
\begin{gather*}
h(t) :=\exp\int_{0}^{t}f(s) \mathrm{d}s,\\
b(t) h(t) =\int_{0}^{t}( \ell (s) +1) f(s) \mathrm{d}s+1.
\end{gather*}
Apply $\frac{d}{dt}$to the latter equation:%
\begin{gather*}
b^{\prime}(t) h(t) +b(t) f (t) h(t) =( \ell(t) +1)f(t) ,\\
b^{\prime}(t) h(t) =f(t)\left\{ \ell(t) +1-\int_{0}^{t} ( \ell(s)+1) f(s) \mathrm{d}s-1\right\} \\
 \hphantom{b^{\prime}(t) h(t)}{} =f(t) \left\{ \ell(t) -\int_{0}^{t}(\ell(s) +1) f(s) \mathrm{d}s\right\}\leq0.
\end{gather*}
Hence $b^{\prime}(t) \leq0$ and $b(t) \leq b(0) =1$ which implies
\begin{gather*}
\ell(t) \leq\int_{0}^{t} ( \ell(s) +1 ) f(s) \mathrm{d}s=b(t) h(t) -1\leq h(t) -1.
\end{gather*}
Finally $ \Vert M(1) \Vert \leq$ $ \Vert M ( 1 ) -I \Vert +1\leq\exp\int_{0}^{1}f(s) \mathrm{d}s$.
\end{proof}

Next we set up a dif\/ferentiable path $p(t) = ( p_{1} ( t ) ,\ldots,p_{N}(t)) $ in $\mathbb{C}_{\rm reg}^{N}$ starting at $x_{0}$ and obtain the equation
\begin{gather*}
\frac{\mathrm{d}}{\mathrm{d}t}L(p(t)) =\kappa L( p(t)) \sum_{i=1}^{N}\left\{ \sum_{j\neq i}\frac{p_{i}^{\prime}(t) }{p_{i}(t) -p_{j}(t) }
\tau((i,j)) -\gamma\frac{p_{i}^{\prime}(t) }{p_{i}(t) }I\right\} \\
\hphantom{\frac{\mathrm{d}}{\mathrm{d}t}L(p(t))}{} =\kappa L(p(t)) \left\{ \sum_{1\leq i<j\leq
N}\frac{p_{j}^{\prime}(t) -p_{i}^{\prime}(t)}{p_{j}(t) -p_{i}(t) }\tau ( (i,j)) -\gamma\sum_{i=1}^{N}\frac{p_{i}^{\prime} (t) }{p_{i}(t) }I\right\} .
\end{gather*}
Suppose $x=\big( e^{\mathrm{i}\theta_{1}},\ldots,e^{\mathrm{i}\theta_{N}}\big) \in\mathcal{C}_{0}$ and $\theta_{1}<\theta_{2}<\cdots<\theta
_{N}<\theta_{1}+2\pi$. Def\/ine the path $p(t) =\big( e^{\mathrm{i}g_{1}(t) },\ldots,e^{\mathrm{i}g_{N}(t) }\big) $ where $g_{j}(t) =(1-t)
\frac{2(j-1) \pi}{N}+t\theta_{j}$ for $1\leq j\leq N$. Then $p(t) \in\mathcal{C}_{0}$ for $0\leq t\leq1$ because $g_{i+1}(t) -g_{i}(t) =(1-t)
\frac{2\pi}{N}+t ( \theta_{i+1}-\theta_{i} ) >0$ for $1\leq i<N$ and $2\pi+g_{1}(t) -g_{N}(t) =2\pi+t\theta_{1}-(1-t) \frac{2(N-1) \pi}{N}-t\theta
_{N}=(1-t) \frac{2\pi}{N}+t ( 2\pi+\theta_{1}-\theta _{N}) >0$. The factor of $\tau((i,j)) $ in the equation is
\begin{gather*}
\mathrm{i}\frac{g_{j}^{\prime}(t) e^{\mathrm{i}g_{j} (t) }-g_{i}^{\prime}(t) e^{\mathrm{i}g_{i}(t) }}{e^{\mathrm{i}g_{j}(t) }-e^{\mathrm{i}g_{i}(t) }}=
\frac{1}{2}\left\{( g_{j}^{\prime}(t) -g_{i}^{\prime}(t)) \frac{\cos\big( \frac{1}{2}(g_{j}(t) -g_{i}(t) ) \big) }{\sin\big(
\frac{1}{2}( g_{j}(t) -g_{i}(t) )\big) }+\mathrm{i} ( g_{j}^{\prime}(t) +g_{i}^{\prime
}(t) ) \right\} .
\end{gather*}
We will apply Lemma \ref{bdsM} to $\widetilde{L}(x) =\prod\limits_{j=1}^{N}x_{j}^{\gamma\kappa}L(x) $; this only changes the phase and removes the $\sum\limits_{i=1}^{N}\frac{p_{i}^{\prime} (t) }{p_{i}(t) }$ term. In the notation of Lemma~\ref{bdsM}
\begin{gather*}
f(t) =\vert \kappa\vert \sum\limits_{i<j}\left\vert \mathrm{i}\frac{g_{j}^{\prime}(t) e^{\mathrm{i}g_{j}(t) }-g_{i}^{\prime}(t)
e^{\mathrm{i}g_{i}(t) }}{e^{\mathrm{i}g_{j}(t)}-e^{\mathrm{i}g_{i}(t) }}\right\vert
\end{gather*} (since $\Vert\tau((i,j))\Vert =1$). To set up the
integral
$\int_{0}^{1}f(t) \mathrm{d}t$ let \begin{gather*}\phi_{ij} (t) =\frac{1}{2}( g_{j}(t) -g_{i}(t)) =\frac{1}{2}\left\{ (1-t) \frac{2(j-i)
\pi}{N}+t(\theta_{j}-\theta_{i}) \right\}
\end{gather*} so that $\phi
_{ij}^{^{\prime}}(t) =\frac{1}{2}\big( \theta_{j}-\theta_{i}+\frac{2(j-i) \pi}{N}\big) $ and $0<\phi_{ij}(t) <\pi$ for $i<j$ and $0\leq t\leq1$. The terms $\vert \mathrm{i} ( g_{j}^{\prime}(t) +g_{i}^{\prime}(t) )\vert \leq 4\pi$ provide a simple bound (no singularities of\/f $\mathbb{T}_{\rm reg}^{N}$). The dominant terms come from $\int_{0}^{1}\frac{\vert \phi_{ij}^{\prime}\cos\phi_{ij}(t)\vert }{\sin\phi_{ij}(t) }\mathrm{d}t$. There are two cases. Let $\phi_{0},\phi_{1}$ satisfy $0<\phi_{0},\phi_{1}<\pi$ and let $\phi(t) =(1-t) \phi_{0}+t\phi_{1}$. The antiderivative $\int\frac{\phi^{\prime}\cos\phi(t) }{\sin \phi(t) }\mathrm{d}t=\log\sin\phi(t) $. The f\/irst case applies when either $0<\phi_{0},\phi_{1}\leq\frac{\pi}{2}$ or $\frac{\pi
}{2}\leq\phi_{0},\phi_{1}<\pi$ (assign $\phi_{0}=\frac{\pi}{2}=\phi_{1}$ to the f\/irst interval); then $\phi^{\prime}\cos\phi(t) \geq0$ if $0<\phi_{0}\leq\phi_{1}\leq\frac{\pi}{2}$ or $\frac{\pi}{2}\leq\phi_{1}\leq\phi_{0}<\pi$ and $\phi^{\prime}\cos\phi(t) <0$ otherwise. These imply
\begin{gather*}
\int_{0}^{1}\frac{ \vert \phi^{\prime}\cos\phi(t)\vert }{\sin\phi(t) }\mathrm{d}t=\left\vert \log\frac{\sin\phi_{1}}{\sin\phi_{0}}\right\vert .
\end{gather*}
The second case applies when either $0<\phi_{0}<\frac{\pi}{2}<\phi_{1}<\pi$ (thus $\phi^{\prime}>0$) or $0<\phi_{1}<\frac{\pi}{2}<\phi_{0}<\pi$. Let
$\phi(t_{0}) =\frac{\pi}{2}$ (that is, $t_{0}=\frac{\pi /2-\phi_{0}}{\phi_{1}-\phi_{0}})$. In the f\/irst situation
\begin{gather*}
\int_{0}^{1}\frac{\vert \phi^{\prime}\cos\phi(t)\vert }{\sin\phi(t) }\mathrm{d}t =\int_{0}^{t_{0}}
\frac{\phi^{\prime}\cos\phi(t) }{\sin\phi(t)}\mathrm{d}t-\int_{t_{0}}^{1}\frac{\phi^{\prime}\cos\phi(t)
}{\sin\phi(t) }\mathrm{d}t =-\log\sin\phi_{0}-\log\sin\phi_{1},
\end{gather*}
since $\log\sin\frac{\pi}{2}=0$; and the same value holds for the second situation. We obtain
\begin{gather*}
\int_{0}^{1}f(t) \mathrm{d}t\leq\vert \kappa\vert \sum_{1\leq i<j\leq N}\left\{ -\log\sin\frac{\theta_{j}-\theta_{i}}{2}-\log\sin\frac{(j-i) \pi}{N}+4\pi\right\} .
\end{gather*}
Taking exponentials and using the lemma (recall $\big\vert e^{\mathrm{i}\phi_{1}}-e^{\mathrm{i}\phi_{2}}\big\vert =2\sin\big\vert \frac{\phi
_{1}-\phi_{2}}{2}\big\vert $) we obtain%
\begin{gather*}
\Vert L(x)\Vert \leq c\prod\limits_{1\leq i<j\leq N}\vert x_{i}-x_{j}\vert ^{-\vert \kappa\vert }.
\end{gather*}
This bound applies to all of $\mathbb{T}_{\rm reg}^{N}$ when $L ( x_{0}) $ commutes with $\upsilon$ and $L$ is extended to $\mathbb{T}_{\rm reg}^{N}$ as in Def\/inition~\ref{DefL(x)T}. This completes the proof of Theorem~\ref{Lbnd}.

Next we f\/ind bounds on the series expansion from (\ref{Lzseries})
\begin{gather*}
L_{1}(x) =\left( \big( u^{2}-z^{2}\big) \prod_{j=1}^{N-2}x_{j}\right) ^{-\gamma\kappa}\rho\big( z^{-\kappa},z^{\kappa}\big)
\sum_{n=0}^{\infty}\alpha_{n}(x(u,0)) z^{n},
\end{gather*}
where $\vert z\vert <\delta_{0}:=\min\limits_{1\leq j\leq N-2}\left\vert u-x_{j}\right\vert $ and $\operatorname{Im}\frac{z}{u}>0$. Recall the recurrence~(\ref{recurS})
\begin{gather*}
S_{n} :=\sum_{i=0}^{n-1}\alpha_{n-1-i}\big\{ (-1) ^{i}\beta_{i}-\sigma\beta_{i}\sigma\big\} , \qquad \beta_{i}:=\sum_{j=0}^{N-2}\frac{\tau((j,N)) }{(u-x_{j}) ^{i+1}},\\
\alpha_{2n}(x(u,0)) =\frac{\kappa}{2n}S_{2n},\\
\alpha_{2n+1}(x(u,0)) =\rho\left(\frac{\kappa}{2n+1-2\kappa},\frac{\kappa}{2n+1+2\kappa}\right) S_{2n+1}.
\end{gather*}

\begin{Proposition}Suppose $\vert \kappa\vert \leq\kappa_{0}<\frac{1}{2}$ and $\lambda:=( N-2) \kappa_{0}$ then for $n\geq0$
\begin{gather}
\Vert \alpha_{2n}(x(u,0)) \Vert \leq \Vert \alpha_{0}(x(u,0)) \Vert \frac{(\lambda) _{n}\big( \lambda+\frac{1}{2}-\kappa
_{0}\big) _{n}}{n!\big( \frac{1}{2}-\kappa_{0}\big) _{n}}\delta _{0}^{-2n},\label{bndan}\\
 \Vert \alpha_{2n+1}(x(u,0)) \Vert \leq \Vert \alpha_{0}(x(u,0)) \Vert \frac{(\lambda) _{n+1}\big( \lambda+\frac{1}{2}-\kappa
_{0}\big) _{n}}{n!\big( \frac{1}{2}-\kappa_{0}\big) _{n+1}}\delta _{0}^{-2n-1}.\nonumber
\end{gather}
\end{Proposition}

\begin{proof} Suppose $n\geq1$ then $ \Vert S_{n} \Vert \leq\sum\limits_{i=0}^{n-1-i} \Vert \alpha_{n-1-i} \Vert ( 2N-4) \delta
_{0}^{-i-1}$ (since $ \Vert \tau((j,N-1))\Vert =1$). Furthermore, since $\big\vert \frac{\kappa}{n\pm2\kappa }\big\vert \leq\frac{\kappa_{0}}{n-2\kappa_{0}}$ for $n\geq2$, we f\/ind
\begin{gather*}
\Vert \alpha_{2n+1}\Vert \leq\frac{2\lambda}{2n+1-2\kappa_{0}}\sum\limits_{i=0}^{2n} \Vert \alpha_{2n-i}\Vert \delta_{0}^{-i-1},\\
\Vert \alpha_{2n}\Vert \leq\frac{\lambda}{n}\sum\limits_{i=0}^{2n-1} \Vert \alpha_{2n-1-i}\Vert \delta_{0}^{-i-1}.
\end{gather*}
To set up an inductive argument let $t_{n}$ denote the hypothetical bound on $\Vert \alpha_{n}(x(u,0))\Vert $ and set $v_{n}=\sum\limits_{i=0}^{n-1}t_{n-1-i}\delta_{0}^{-i-1}$; then $v_{n}=\delta_{0}^{-1}( t_{n-1}+v_{n-1}) $ for $n\geq2$. Setting $t_{2n}=\frac{\lambda}{n}v_{2n}$ and $t_{2n+1}=\frac{2\lambda}{2n+1-2\kappa_{0}}v_{2n+1}$ the recurrence relations become
\begin{gather*}
t_{2n} =\frac{\lambda}{n}\left( t_{2n-1}+\frac{2n-1-2\kappa_{0}}{2\lambda }t_{2n-1}\right) \delta_{0}^{-1}=\frac{2\lambda+2n-1-2\kappa_{0}}{2n}t_{2n-1}\delta_{0}^{-1},\\
t_{2n+1} =\frac{2\lambda}{2n+1-2\kappa_{0}}\left( t_{2n}+\frac{n}{\lambda }t_{2n}\right) \delta_{0}^{-1}=\frac{2\lambda+2n}{2n+1-2\kappa_{0}}t_{2n}\delta_{0}^{-1}.
\end{gather*}
Starting with $ \Vert \alpha_{1} \Vert \leq\frac{2\lambda}{1-2\kappa_{0}} \Vert \alpha_{0} \Vert \delta_{0}^{-1}=t_{1}$ the stated bounds are proved inductively.
\end{proof}

By use of Stirling's formula for $\frac{\Gamma ( n+a ) }{\Gamma(n+b) }\sim n^{a-b}$ we see that $t_{n}$ behaves like (a~multiple of) $n^{2\lambda-1}$ for large $n$. Also there is a constant~$c^{\prime}$ depending on $N$ and~$\kappa_{0}$ such that
\begin{gather}
\sum_{n=2}^{\infty}\Vert \alpha_{n}(x(u,0)) \Vert \vert z\vert ^{n}\leq c^{\prime}\Vert \alpha_{0}(x(u,0)) \Vert \left( \frac{\vert z\vert }{\delta_{0}}\right) ^{2}\left( 1-\frac{\vert z\vert }{\delta_{0}}\right) ^{-2\lambda-2}.\label{bndan2z}
\end{gather}

We also need to analyze the ef\/fect of small changes in $u$. Fix a point $x( \widetilde{u},0) $ and consider series expansions of $\alpha_{n}( x( \widetilde{u},0)) $ in powers of $ ( u-\widetilde{u} ) $. Let $\delta_{1}:=\min\limits_{1\leq j\leq N-2}\vert \widetilde{u}-x_{j}\vert $. Recall equation~(\ref{dua0x})
\begin{gather*}
\partial_{u}\alpha_{0}(x(u,0)) =\kappa\alpha _{0}(x(u,0)) \{ \beta_{0}( x ( u,0) ) +\sigma\beta_{0}(x(u,0))\sigma\} ,
\end{gather*}
and solve this in the form
\begin{gather*}
\alpha_{0}(x(u,0)) =\sum_{n=0}^{\infty}\alpha_{0,n} ( x ( \widetilde{u},0 ) ) (u-\widetilde{u}) ^{n}.
\end{gather*}
This leads to the recurrence (suppressing the arguments $x(\widetilde{u},0) $)
\begin{gather*}
\sum_{n=1}^{\infty}n\alpha_{0,n} ( u-\widetilde{u} ) ^{n-1} \\
\qquad{} =\kappa\sum_{n=0}^{\infty}\alpha_{0,n} ( u-\widetilde{u} ) ^{n}
 \sum_{m=0}^{\infty}(-1) ^{m} ( u-\widetilde {u} ) ^{m}\sum_{j=0}^{N-2}\frac{\tau ( (j,N-1) ) +\tau((j,N)) }{( \widetilde{u}-x_{j}) ^{m+1}},\\
( n+1) \alpha_{0,n+1} =\kappa\sum_{m=0}^{n}\alpha_{0,n-m}\widetilde{\beta}_{m}( x( \widetilde{u},0) ),\\
\widetilde{\beta}_{m}( x( \widetilde{u},0) ):=(-1) ^{m}\sum_{j=0}^{N-2}\frac{\tau((j,N-1)) +\tau((j,N)) }{(\widetilde{u}-x_{j}) ^{m+1}}.
\end{gather*}
Thus $\Vert \widetilde{\beta}_{m}\Vert \leq\frac{2(N-2) }{\delta_{1}^{m+1}}$ and by a similar method as above we f\/ind
\begin{gather}
 \Vert \alpha_{0,n} ( x ( \widetilde{u},0 ))\Vert \leq\frac{( 2\lambda) _{n}}{n!}\delta_{1}^{-n}\Vert \alpha_{0}( x( \widetilde{u},0))\Vert, \label{a0nbd}
\end{gather}
where $\lambda=( N-2) \kappa_{0}$. From
\begin{gather*}
\alpha_{1}(x(u,0)) =\rho\left( \frac{\kappa}{1-2\kappa},\frac{\kappa}{1+2\kappa}\right) \alpha_{0}( x(u,0) ) \sum_{j=0}^{N-2}\frac{\tau( (j,N-1)
) -\tau((j,N)) }{(u-x_{j})}\\
\hphantom{\alpha_{1}(x(u,0))}{} =\rho\left( \frac{\kappa}{1-2\kappa},\frac{\kappa}{1+2\kappa}\right)\sum_{n=0}^{\infty}\alpha_{0,n} ( u-\widetilde{u} )^{n}\\
\hphantom{\alpha_{1}(x(u,0))=}{} \times\sum_{m=0}^{\infty}(-1) ^{m} ( u-\widetilde {u}) ^{m}\sum_{j=0}^{N-2}\frac{\tau( (j,N-1)) -\tau((j,N)) }{(\widetilde
{u}-x_{j}) ^{m+1}},
\end{gather*}
we can derive a recurrence for the coef\/f\/icients in $\alpha_{1}( x ( u,0)) =\sum\limits_{n=0}^{\infty}\alpha_{1,n}( x(\widetilde{u},0))( u-\widetilde{u}) ^{n}$. Also
\begin{gather*}
\Vert \alpha_{1,n}( x( \widetilde{u},0) )\Vert \leq\frac{2\lambda}{1-2\kappa_{0}}\Vert \alpha_{0}(x( \widetilde{u},0) )\Vert \sum_{j=0}^{n}\frac{(2\lambda) _{n-j}}{(n-j) !}\delta_{1}^{j-n}\delta_{1}^{-j-1}\\
\hphantom{\Vert \alpha_{1,n}( x( \widetilde{u},0) )\Vert}{}
 =\frac{2\lambda(2\lambda+1) _{n}}{( 1-2\kappa_{0}) n!}\delta_{1}^{-n-1}\Vert \alpha_{0}( x(\widetilde{u},0))\Vert ;
\end{gather*}
note $2\lambda(2\lambda+1) _{n}=(2\lambda) _{n+1}$.

Essentially we are setting up bounds on behavior of $L( x (u,z)) $ for points near $x( \widetilde{u},0) $ in terms of $ \Vert \alpha_{0} (x(\widetilde{u},0))\Vert $ which is handled by the global bound.

In the series
\begin{gather*}
\sum_{n=0}^{\infty}\alpha_{n}(x(u,0)) z^{n}=\sum_{m,n=0}^{\infty}\alpha_{n,m}( x( \widetilde{u},0)) z^{n}( u-\widetilde{u}) ^{m}\end{gather*}
the f\/irst order terms are%
\begin{gather*}
\alpha_{00}( x( \widetilde{u},0) ) +\alpha
_{0,1}( x( \widetilde{u},0)) ( u-\widetilde{u}) +\alpha_{1,0}( x( \widetilde{u},0)) z,
\end{gather*}
and the bounds (\ref{bndan}) for the omitted terms
\begin{gather}
\sum_{n=2}^{\infty}\Vert \alpha_{n}(x(u,0))\Vert \vert z\vert ^{n}\leq c^{\prime}\Vert \alpha_{0}(x(u,0))\Vert \left( \frac{\vert z\vert }{\delta_{0}}\right) ^{2}\left( 1-\frac{\vert z\vert}{\delta_{0}}\right) ^{-2\lambda-2},\nonumber\\
\sum\limits_{n=2}^{\infty}\Vert \alpha_{0,n}( x(\widetilde{u},0) ) \Vert \vert u-\widetilde{u}\vert ^{n} \leq(2\lambda) _{2}\left( \frac
{\vert u-\widetilde{u}\vert }{\delta_{1}}\right) ^{2}\left(
1-\frac{\vert u-\widetilde{u}\vert }{\delta_{1}}\right)
^{-2\lambda-2}\Vert \alpha_{0}( x( \widetilde{u},0))\Vert ,\label{dblseries}\\
\vert z\vert \sum_{n=1}^{\infty}\Vert \alpha_{1,n}(
x( \widetilde{u},0)) \Vert \vert
u-\widetilde{u}\vert ^{n} \leq\frac{(2\lambda) _{2}
}{1-2\kappa_{0}}\left( \frac{\vert z( u-\widetilde{u})
\vert }{\delta_{1}^{2}}\right) \left( 1-\frac{\vert
u-\widetilde{u}\vert }{\delta_{1}}\right) ^{-2\lambda-2}\Vert
\alpha_{0}( x( \widetilde{u},0))\Vert.\nonumber
\end{gather}
Note there is a dif\/ference between $\delta_{0}$ and $\delta_{1}$: $\delta _{0},\delta_{1}$ are the distances from the nearest $x_{j}$ ($1\leq j\leq
N-2)$ to $u,\widetilde{u}$ respectively; thus the double series converges in $\vert z\vert +\vert u-\widetilde{u}\vert <\delta_{1}$ because this implies $\vert z\vert <\delta_{1}- \vert u-\widetilde{u} \vert \leq\delta_{0}$, by the triangle inequality: $\delta_{1}\leq\vert u-\widetilde{u}\vert +\delta_{0}$ .

\section{Suf\/f\/icient condition for the inner product property}\label{suffco}

In this section we will use the series
\begin{gather*}
L_{1}(y,u-z,u+z) =\left( \prod\limits_{j=1}^{N}x_{j}\right) ^{-\gamma\kappa}\rho\big( z^{-\kappa},z^{\kappa}\big) \sum_{n=0}^{\infty
}\alpha_{n}(x(u,0)) z^{n},
\end{gather*}
normalized by $\alpha_{0}\big( x^{(0) }\big) =I$ where $x^{(0) }=\big( 1,\omega,\omega^{2},\ldots,\omega^{N-3},\omega^{-3/2},\omega^{-3/2}\big) $, $\omega=e^{2\pi\mathrm{i}/N}$. The hypothesis is that there exists a Hermitian matrix~$H$ such that $\upsilon H=H\upsilon$ (recall $\upsilon:=\tau(w_{0})$) and the matrix~$H_{1}$ def\/ined by
\begin{gather}
L_{1}(x) ^{\ast}H_{1}L_{1}(x) =L(x) ^{\ast}HL(x) \label{LHL}
\end{gather}
commutes with $\sigma=\tau(N-1,N) $ (recall $L (x_{0}) =I$). Setting $x=x_{0}$ we f\/ind that $H=L_{1}(x_{0}) ^{\ast}H_{1}L_{1}(x_{0}) $. The analogous
condition has to hold for each face of $\mathcal{C}_{0}$ and any such face can be obtained from $\{ x_{N-1}=x_{N}\} $ by applying $x\mapsto xw_{0}^{m}$ with suitable $m$. For notational simplicity we will work out only the $\{ x_{N-2}=x_{N-1}\} $ case. From the general relation $w(i,j) w^{-1}=(w(i),w(j))$ we obtain $w_{0}^{-1}(N-1,N) w_{0}=(N-2,N-1)$. A matrix $M$ commutes with $\tau(N-2,N-1)$ if and only if $\upsilon M\upsilon^{-1}$ commutes with $\sigma$. Let
\begin{gather*}
x^{\prime} =\big( x_{1}^{\prime},\ldots,x_{N-3}^{\prime},u-z,u+z,x_{N}^{\prime}\big) ,\\
x^{\prime}w_{0}^{-1} =\big( x_{N}^{\prime},x_{1}^{\prime},\ldots ,x_{N-3}^{\prime},u-z,u+z\big) =x,\\
L_{2}( x^{\prime}) :=\upsilon^{-1}L_{1}\big( x^{\prime}w_{0}^{-1}\big) \upsilon =\left( \prod\limits_{j=1}^{N}x_{j}\right) ^{-\gamma\kappa}\upsilon
^{-1}\rho\big( z^{-\kappa},z^{\kappa}\big) \sum_{n=0}^{\infty}\alpha_{n}(y,u) \upsilon z^{n}.
\end{gather*}
This is a solution of (\ref{Lsys}) by Proposition~\ref{L(xw)}. This has the analogous behavior to $L_{1}$; writing
\begin{gather*}
\upsilon^{-1}\rho\big( z^{-\kappa},z^{\kappa}\big) \alpha_{n} (y,u) \upsilon=\big\{ \upsilon^{-1}\rho\big( z^{-\kappa},z^{\kappa
}\big) \upsilon\big\} \big\{ \upsilon^{-1}\alpha_{n} (y,u) \upsilon\big\}
\end{gather*}
implies the relations
\begin{gather*}
\tau(N-2,N-1) \big\{ \upsilon^{-1}\rho\big( z^{-\kappa },z^{\kappa}\big) \upsilon\big\} =\big\{ \upsilon^{-1}\rho\big(
z^{-\kappa},z^{\kappa}\big) \upsilon\big\} \tau(N-2,N-1),\\
\tau(N-2,N-1) \big\{ \upsilon^{-1}\alpha_{n} (y,u) \upsilon\big\} \tau(N-2,N-1) = (-1) ^{n}\big\{ \upsilon^{-1}\alpha_{n}(y,u)
\upsilon\big\} ,\qquad n\geq0.
\end{gather*}
We claim that the Hermitian matrix $H_{2}$ def\/ined by
\begin{gather}
L_{2}(x) ^{\ast}H_{2}L_{2}(x) =L(x)^{\ast}HL(x) \label{L2H2L2}
\end{gather}
commutes with $\tau(N-2,N-1) $. There is a subtle change: the base point $x^{(0)}=\big( 1,\omega,\ldots,\omega^{N-2}$, $\omega^{-3/2},\omega^{-3/2}\big) $ is replaced by $\big( \omega ,\ldots,\omega^{N-2},\omega^{-3/2},\omega^{-3/2},1\big) $ and now $\omega x_{0}=\big( \omega,\ldots,\omega^{N-1},1\big) $ is in the domain of convergence of~$L_{2}$. Set $x=\omega x_{0}$ in~(\ref{L2H2L2}) to obtain
\begin{gather*}
L_{2}(\omega x_{0}) =\upsilon^{-1}L_{1}\big( \omega x_{0}w_{0}^{-1}\big) \upsilon=\upsilon^{-1}L_{1}(x_{0})\upsilon,\\
H_{2} =( L_{2}(\omega x_{0}) ^{\ast})^{-1}HL_{2}(\omega x_{0}) ^{-1}=\upsilon^{-1}(L_{1}(x_{0}) ^{\ast}) ^{-1}\upsilon H\upsilon^{-1}L_{1}(x_{0}) ^{-1}\upsilon\\
\hphantom{H_{2}}{} =\upsilon^{-1} ( L_{1}(x_{0}) ^{\ast} )^{-1}HL_{1}(x_{0}) ^{-1}\upsilon=\upsilon^{-1}H_{1}\upsilon,
\end{gather*}
because $H$ commutes with $\upsilon$ (and $L(\omega x_{0}) =L(x_{0}) =I$ by the homogeneity). Thus $H_{2}$ commutes with $\tau(N-2,N-1) $.

From Theorem \ref{Lbnd} we have the bound
\begin{gather*}
\Vert L(x) ^{\ast}HL(x)\Vert \leq c\prod\limits_{i<j}\vert x_{i}-x_{j} \vert ^{-2\vert \kappa\vert }.
 \end{gather*}
 Denote $K(x) =L(x) ^{\ast}HL(x) $. We will show that there is an interval $-\kappa_{1}<\kappa<\kappa_{1}$ where $\kappa_{1}$ depends on $N$ such that
\begin{gather*}
\int_{\mathbb{T}^{N}}\big\{ ( x_{N}\mathcal{D}_{N}f ) ^{\ast}(x) K(x) g(x) -f^{\ast} (x t) K(x) x_{N}\mathcal{D}_{N}g(x) \big\}\mathrm{d}m(x) =0,
\end{gather*}
for each $f,g\in C^{1}\big( \mathbb{T}^{N};V_{\tau}\big) $. Consider the Haar measure of $\big\{ x\colon \min\limits_{i<j}\vert x_{i}-x_{j}\vert
<\varepsilon\big\} $; let $\sin\frac{\varepsilon^{\prime}}{2} =\frac{\varepsilon}{2}$ and $i<j$ then $m\big\{ x\colon \vert x_{i}-x_{j}\vert \leq\varepsilon\big\} =\frac{1}{\pi}\varepsilon^{\prime}$, thus $m\big\{ x\colon \min\limits_{i<j}\vert x_{i}-x_{j}\vert <\varepsilon\big\} \leq \binom{N}{2}\frac{\varepsilon^{\prime}}{\pi}$. The integral is broken up into three pieces. The aim is to let $\delta \rightarrow0$; where $\delta$ satisf\/ies an upper bound $\delta<\min\big(\big( 2\sin\frac{\pi}{N}\big) ^{2},\frac{1}{9}\big)$; the f\/irst term comes from the maximum spacing of~$N$ points on~$\mathbb{T}$ and the second is equivalent to $3\delta<\delta^{1/2}$. Also $\delta^{\prime}:=2\arcsin\frac{\delta}{2}$.
\begin{enumerate}\itemsep=0pt
\item $\min\limits_{i<j}\vert x_{i}-x_{j}\vert <\delta$, done with the integrability of $\prod\limits_{i<j}\vert x_{i}-x_{j}\vert
^{-2\vert \kappa\vert }$ for $\vert \kappa\vert <\frac{1}{N}$ (from the Selberg integral $\int_{\mathbb{T}^{N}}\prod \limits_{i<j}\vert x_{i}-x_{j}\vert ^{-2 \vert \kappa \vert }\mathrm{d}m(x) =\frac{\Gamma 1-N \vert \kappa\vert) }{\Gamma( 1-\vert \kappa\vert) ^{N}}$), and the measure of the set is~$O( \delta) $. The limit as $\delta\rightarrow0$ is zero by the dominated convergence theorem.

\item $\delta\leq\min\limits_{1\leq i<j\leq N-1}\vert x_{i}-x_{j}\vert <\delta^{1/2}$ and $\delta\leq\min\limits_{1\leq i\leq N-1}\vert x_{i}-x_{N}\vert $; this case uses the same bound on $K$ and the $\mathbb{T}^{N-1}$-Haar measure of $\big\{ x\in \mathbb{T}^{N-1} \colon \min\limits_{1\leq i<j\leq N-1} \vert x_{i} -x_{j} \vert <\delta^{1/2}\big\} ;$

\item $\min\limits_{1\leq i<j\leq N-1}\vert x_{i}-x_{j}\vert \geq\delta^{1/2}$ and $\min\limits_{1\leq i\leq N-1}\vert x_{i}-x_{N}\vert \geq\delta$. This is done with a detailed analysis using the double series from~(\ref{dblseries}).
\end{enumerate}

The total of parts (2) and (3), that is, the integral over $\Omega_{\delta}$, equals
\begin{gather*}
\int_{\Omega_{\delta}}x_{N}\partial_{N}\{ f(x)^{\ast}K(x) g(x)\} \mathrm{d}m (x).
\end{gather*}
We use the coordinates $x_{j}=e^{\mathrm{i}\theta_{j}}$, $1\leq j\leq N$; thus $x_{N}\partial_{N}=-\mathrm{i}\frac{\partial}{\partial
\theta_{N}}$. For f\/ixed $( \theta_{1},\ldots,\theta_{N-1}) $ the condition $x\in\Omega_{\delta}$ implies that the set of $\theta_{N}$-values is a~union of disjoint closed intervals (it is possible there is only one, in the extreme case $\theta_{j}=j\delta^{\prime}$ for $1\leq j\leq N-1$ the interval is $N\delta^{\prime}\leq\theta_{N}\leq2\pi$). In case (2) the $\theta_{N}$-integration results in a sum of terms $( f^{\ast}Kg) \big(e^{\mathrm{i}\theta_{1}}, \ldots,e^{\mathrm{i}\theta_{N-1}},e^{\mathrm{i}\phi }\big) $ with coef\/f\/icients $\pm1$ where $\min\limits_{1\leq i\leq N-1} \vert \phi-\theta_{i} \vert =\delta$. Each such sum is bounded by $2(N-1) c\Vert f\Vert _{\infty}\Vert g \Vert _{\infty}\delta^{-N(N-1) \vert \kappa \vert }$, because $\prod\limits_{i<j}\vert x_{i}-x_{j}\vert \geq\delta^{N(N-1) /2}$ on $\Omega_{\delta}$. Thus the integral for part~(2) is bounded by
\begin{gather*}
2(N-1)\Vert f\Vert _{\infty}\Vert g\Vert _{\infty}\delta^{-N(N-1) \vert \kappa\vert }\binom{N-1}{2}\left( 2\arcsin\frac{\delta^{1/2}}{2}\right) \leq c^{\prime
}\delta^{1/2-N(N-1) \vert \kappa\vert },
\end{gather*}
for some f\/inite constant $c^{\prime}$ (depending on $f$, $g$). This term tends to zero as $\delta\rightarrow0$ if $\vert \kappa\vert <\frac {1}{2N(N-1) }$.

In part (3) the intervals $[ \theta_{i}-\delta^{\prime},\theta_{i}+\delta^{\prime}] $ are pairwise disjoint because $\vert\theta_{i}-\theta_{j}\vert \geq3\delta^{\prime}$ for $i\neq j$ (recall $\sqrt{\delta}>3\delta$). To simplify the notation assume $\theta_{1}<\theta_{2}<\cdots$ (the other cases follow from the group invariance of the setup). Then the $\theta_{N}$-integration yields
\begin{gather*}
(2\pi) ^{1-N}\sum_{j=1}^{N-1}\int_{R_{\delta}}\left\{
\begin{matrix}
( f^{\ast}Kg) \big( e^{\mathrm{i}\theta_{1}},\ldots
,e^{\mathrm{i}\theta_{j}},\ldots,e^{\mathrm{i} ( \theta_{j}-\delta^{\prime} ) }\big) \\
- ( f^{\ast}Kg ) \big( e^{\mathrm{i}\theta_{1}},\ldots,e^{\mathrm{i}\theta_{j}},\ldots,e^{\mathrm{i} ( \theta_{j}+\delta^{\prime} ) }\big)
\end{matrix}
\right\} \mathrm{d}\theta_{1}\cdots\mathrm{d}\theta_{N-1},
\end{gather*}
where $R_{\delta}:=\big\{ ( \theta_{1},\ldots,\theta_{N-1} ) :\theta_{1}<\theta_{2}<\cdots<\theta_{N-1}<\theta_{1}+2\pi,\min \big\vert
e^{\mathrm{i}\theta_{j}}-e^{\mathrm{i}\theta_{k}}\big\vert \geq\sqrt{\delta }\big\}$. It suf\/f\/ices to deal with the term with $j=N-1$; this allows the
use of the double series. It is fairly easy to show that $ ( f^{\ast}Kg) \big( e^{\mathrm{i}\theta_{1}},\ldots,e^{\mathrm{i}\theta_{N-1}},e^{\mathrm{i}( \theta_{N-1}-\delta^{\prime}) }\big) - (f^{\ast}Kg) \big( e^{\mathrm{i}\theta_{1}},\ldots,e^{\mathrm{i}\theta_{N-1}},e^{\mathrm{i}( \theta_{N-1}+\delta^{\prime})}\big) $ tends to zero with $\delta$ but this is not enough to control the integral. The idea is to show that
\begin{gather*}
 \big\vert ( f^{\ast}Kg ) \big( e^{\mathrm{i}\theta_{1}},\ldots,e^{\mathrm{i}\theta_{N-1}},e^{\mathrm{i} ( \theta_{N-1}-\delta^{\prime} ) }\big) - (f^{\ast}Kg) \big(e^{\mathrm{i}\theta_{1}},\ldots,e^{\mathrm{i}\theta_{N-1}},e^{\mathrm{i} ( \theta_{N-1}+\delta^{\prime} ) }\big) \big\vert \\
\qquad{} \leq c''\delta^{1/2-2\vert \kappa\vert }\big\vert (f^{\ast}Kg) \big( e^{\mathrm{i}\theta_{1}},\ldots,e^{\mathrm{i}\theta_{N-1}},e^{\mathrm{i} ( \theta_{N-1}+\delta^{\prime} )}\big) \big\vert
\end{gather*}
for some constant $c''$. This can then be bounded using the $ \Vert K\Vert $ bound for suf\/f\/iciently small $\vert \kappa\vert $. Fix $y=\big( e^{\mathrm{i}\theta_{1}},\ldots,e^{\mathrm{i}\theta_{N-2}}\big) $ and let $x ( y,u-v,u+v ) $ denote $\big(e^{\mathrm{i}\theta_{1}},\ldots,e^{\mathrm{i}\theta_{N-2}},u-z,u+z\big)$. We will use the form $K=L_{1}^{\ast}H_{1}L_{1}$ from~(\ref{LHL}) with two
pairs of values along with $\widetilde{u}=e^{\mathrm{i}\theta_{N-1}}$, and set $\zeta=e^{\mathrm{i\delta}^{\prime}}$
\begin{enumerate}\itemsep=0pt
\item[1)] $\eta^{(1) }=x ( y,u_{1}-z_{1},u_{1}+z_{1} ) =x\big( y,e^{\mathrm{i}\theta_{N-1}},e^{\mathrm{i} ( \theta
_{N-1}+\delta^{\prime} ) }\big) $, then $u_{1}=\frac{1}{2}\widetilde{u} ( 1+\zeta ) $, $z_{1}=\frac{1}{2}\widetilde{u} (\zeta-1 ) $, $u_{1}-\widetilde{u}=z_{1}$, $ \vert z_{1} \vert =\delta$,

\item[2)] $\eta^{(2) }=x ( y,u_{2}-z_{2},u_{2}+z_{2} ) =x\big( y,e^{\mathrm{i} ( \theta_{N-1}-\delta^{\prime} ) },e^{\mathrm{i}\theta_{N-1}}\big) $, then $u_{2}=\frac{1}{2}\widetilde {u}\big( 1+\zeta^{-1}\big) $, $z_{2}=\frac{1}{2}\widetilde{u}\big(1-\zeta^{-1}\big) =\zeta^{-1}z_{1}$, $u_{2}-\widetilde{u}=-z_{2}$, $ \vert z_{2} \vert =\delta$.
\end{enumerate}

Let $\eta^{(3) }\!=\!x\big( y,e^{\mathrm{i}\theta_{N-1}},e^{\mathrm{i} ( \theta_{N-1}-\delta^{\prime} ) }\big)\! =\!\eta^{(2) }(N-1,N) $, then by construction
$K\big( \eta^{(3) }\big) \!=\!\sigma K\big( \eta^{(2) }\big) \sigma$. We start by disposing of the $f$ and $g$ factors: by uniform continuous dif\/ferentiability there is a constant $c^{\prime \prime\prime}$ such that $\big\Vert f\big( \eta^{(1)}\big) -f\big( \eta^{(3) }\big) \big\Vert \leq c^{\prime\prime\prime}\delta^{\prime}$ and $\big\Vert g\big( \eta^{(1) }\big) -g\big( \eta^{(3) }\big) \big\Vert \leq c^{\prime\prime\prime}\delta^{\prime}$ (same constant for all of~$\mathbb{T}^{N}$). So the error made by assuming~$f$ and~$g$ are constant is bounded by $c^{\prime\prime\prime}\delta^{\prime}\big( \big\Vert K\big(\eta^{(1) }\big) \big\Vert +\big\Vert K\big(\eta^{(2) }\big) \big\Vert \big) $. The problem is reduced to bounding $K\big( \eta^{(1) }\big) -\sigma K\big( \eta^{(2) }\big) \sigma$. To add more detail about the ef\/fect of the $\ast$-operation on $u$ and $z$ we compute
\begin{gather*}
z^{\ast}=\frac{1}{2}\left( \frac{1}{u+z}-\frac{1}{u-z}\right) =-\frac{z}{u^{2}-z^{2}}, \qquad u^{\ast}=\frac{1}{2}\left( \frac{1}{u+z}+\frac{1}{u-z}\right)=\frac{u}{u^{2}-z^{2}}
\end{gather*} and if $u-z=e^{i\theta_{N-1}}$, $u+z=e^{i\theta_{N}}$ then
\begin{gather*}
z =\frac{1}{2}e^{\mathrm{i} ( \theta_{N-1}+\theta_{N} )
/2}\big( e^{\mathrm{i} ( \theta_{N}-\theta_{N-1} ) /2}-e^{\mathrm{i} ( \theta_{N-1}-\theta_{N} ) /2}\big)
=\mathrm{i}e^{\mathrm{i} ( \theta_{N-1}+\theta_{N} ) /2}\sin \frac{\theta_{N}-\theta_{N-1}}{2},\\
z^{\ast} =\mathrm{i}e^{-\mathrm{i} ( \theta_{N-1}+\theta_{N} )/2}\sin\frac{\theta_{N-1}-\theta_{N}}{2}=\overline{z},\\
u =\frac{1}{2}e^{\mathrm{i} ( \theta_{N-1}+\theta_{N} ) /2}\big( e^{\mathrm{i} ( \theta_{N}-\theta_{N-1} ) /2}+e^{\mathrm{i} ( \theta_{N-1}-\theta_{N} ) /2}\big) =e^{\mathrm{i}( \theta_{N-1}+\theta_{N}) /2}\cos\frac{\theta_{N}-\theta_{N-1}}{2},\\
u^{\ast} =e^{-\mathrm{i}( \theta_{N-1}+\theta_{N}) /2}\cos\frac{\theta_{N-1}-\theta_{N}}{2}=\overline{u};
\end{gather*}
the $\ast$-operation agrees with complex conjugate on the torus and $\rho( z^{-\kappa},z^{\kappa}) ^{\ast}=\rho( \overline{z}^{-\kappa},\overline{z}^{\kappa}) $. The reason for this is to emphasize that $L(x) ^{\ast}$ is an analytic function agreeing with the (Hermitian) adjoint of~$L(x)$. Thus
\begin{gather}
K\big( \eta^{(1) }\big) =\sum_{n,m=0}^{\infty}\alpha_{n}( x(u_{1},0)) ^{\ast}\rho\big(z_{1}^{-\kappa},z_{1}^{\kappa}\big) ^{\ast}H_{1}\rho\big( z_{1}^{-\kappa
},z_{1}^{\kappa}\big) \alpha_{m} ( x(u_{1},0) ) ( z_{1}^{\ast} ) ^{m}z_{1}^{n}\nonumber\\
\hphantom{K\big( \eta^{(1) }\big)}{} =\sum_{n,m=0}^{\infty}\alpha_{n} ( x(u_{1},0) )
^{\ast}H_{1}\rho \big( \vert z_{1} \vert ^{-2\kappa}, \vert z_{1} \vert ^{2\kappa}\big) \alpha_{m} ( x(u_{1},0) ) \overline{z_{1}}^{m}z_{1}^{n},\label{Keta1}
\end{gather}
because $H_{1}$ commutes with $\sigma$ and hence with $\rho ( z_{1}^{-\kappa},z_{1}^{\kappa}) $, and
\begin{gather}
K\big( \eta^{(3) }\big) =\sigma K\big( \eta^{(2) }\big) \sigma\nonumber\\
\hphantom{K\big( \eta^{(3) }\big)}{} =\sum_{n,m=0}^{\infty}(-1)
^{m+n}\alpha_{n}(x(u_{2},0)) ^{\ast}H_{1}\rho\big(\vert z_{2}\vert ^{-2\kappa},\vert z_{2}\vert ^{2\kappa}\big) \alpha_{m}(x( u_{2},0)) \overline{z_{2}}^{m}z_{2}^{n},\label{Keta3}
\end{gather}
because $\sigma\alpha_{n}(x(u_{2},0)) \sigma=(-1) ^{n}\alpha_{n}(x(u_{2},0)) $ for $n\geq0$. Now we use the expansion in powers of $(u-\widetilde{u}) ^{n}$ to evaluate $K\big( \eta^{(1)}\big) -K\big( \eta^{(3) }\big) $. From the inequality~(\ref{bndan2z})
\begin{gather}
\sum_{n=2}^{\infty} \Vert \alpha_{n}(x(u,0))
 \Vert \vert z\vert ^{n} \leq c^{\prime} \Vert
\alpha_{0}(x(u,0)) \Vert \left(
\frac{\vert z\vert }{\delta_{0}}\right) ^{2}\left( 1-\frac
{\vert z\vert }{\delta_{0}}\right) ^{-2\lambda-2}\nonumber\\
\hphantom{\sum_{n=2}^{\infty} \Vert \alpha_{n}(x(u,0)) \Vert \vert z\vert ^{n}}{} =c^{\prime}\delta \Vert \alpha_{0}(x(u,0))
 \Vert \big( 1-\delta^{1/2}\big) ^{-2\lambda-2}\label{bnd1}
\end{gather}
with $\delta_{0}=\min\limits_{1\leq j\leq N-2} \vert u-x_{j} \vert =\delta^{1/2}$ we can restrict the problem to $0\leq n,m\leq1$. The omitted terms in $K\big( \eta^{(1) }\big) -K\big( \eta^{(3) }\big) $ are bounded by $c^{\prime\prime}\delta^{1-2 \vert \kappa \vert } \Vert B_{1} \Vert \Vert \alpha_{0} (x( u_{2},x) ) \Vert ^{2}$, for some constant~$c^{\prime\prime}$. Then
\begin{gather*}
L_{1}\big(\eta^{(1)}\big) =\left( \prod \limits_{j=1}^{N}x_{j}^{(1) }\right) ^{-\gamma\kappa}\rho\big( z_{1}^{-\kappa},z_{1}^{\kappa}\big) \left\{
\begin{matrix}
\alpha_{00}(x(\widetilde{u},0)) +\alpha
_{0,1}(x(\widetilde{u},0)) ( u_{1}-\widetilde{u}) \\
+\alpha_{1,0}(x(\widetilde{u},0)) z_{1}+O(\delta)
\end{matrix}
\right\} ,\\
\sigma L_{1}\big(\eta^{(2)}\big) \sigma =\left(\prod\limits_{j=1}^{N}x_{j}^{(2) }\right) ^{-\gamma\kappa}\rho\big( z_{2}^{-\kappa},z_{2}^{\kappa}\big) \left\{
\begin{matrix}
\alpha_{00}(x(\widetilde{u},0)) +\alpha_{0,1}(x(\widetilde{u},0)) ( u_{2}-\widetilde{u}) \\
-\alpha_{1,0}(x(\widetilde{u},0)) z_{2}+O(\delta)
\end{matrix}
\right\} ,
\end{gather*}
because $\sigma\alpha_{1}(x(u,0)) \sigma= (-1) ^{n}\alpha_{1}(x(u,0)) $. The terms $O(\delta) $ correspond to the bound in~(\ref{bnd1}). Drop the
argument $x(\widetilde{u},0) $ for brevity. Combining these with~(\ref{Keta1}) and~(\ref{Keta3}) we obtain
\begin{gather*}
K\big(\eta^{(1)}\big) -K\big( \eta^{(3)}\big) = \{ \alpha_{0,1} ( u_{1}-\widetilde{u} )
+\alpha_{1,0}z_{1} \} ^{\ast}H_{1}\rho\big( \vert z_{1}\vert ^{-2\kappa},\vert z_{1}\vert ^{2\kappa}\big)\alpha_{00}\\
\hphantom{K\big(\eta^{(1)}\big) -K\big( \eta^{(3)}\big) =}{} +\alpha_{00}^{\ast}H_{1}\rho\big(\vert z_{1}\vert ^{-2\kappa
}, \vert z_{1} \vert ^{2\kappa}\big) \{ \alpha_{0,1} (u_{1}-\widetilde{u}) +\alpha_{1,0}z_{1}\} \\
\hphantom{K\big(\eta^{(1)}\big) -K\big( \eta^{(3)}\big) =}{}
+ \{ \alpha_{0,1} ( u_{1}-\widetilde{u} ) +\alpha_{1,0}z_{1} \} ^{\ast}H_{1}\rho\big(\vert z_{1}\vert ^{-2\kappa},\vert z_{1}\vert ^{2\kappa}\big)\\
\hphantom{K\big(\eta^{(1)}\big) -K\big( \eta^{(3)}\big) =}{}
\times \{ \alpha_{0,1}(u_{1}-\widetilde{u}) +\alpha_{1,0}z_{1}\} \\
\hphantom{K\big(\eta^{(1)}\big) -K\big( \eta^{(3)}\big) =}{}
- \{ \alpha_{0,1} ( u_{2}-\widetilde{u} ) -\alpha_{1,0}z_{2} \} ^{\ast}H_{1}\rho\big(\vert z_{2}\vert ^{-2\kappa},\vert z_{2}\vert ^{2\kappa}\big) \alpha_{00}\\
\hphantom{K\big(\eta^{(1)}\big) -K\big( \eta^{(3)}\big) =}{}
-\alpha_{00}^{\ast}H_{1}\rho\big( \vert z_{2}\vert ^{-2\kappa},\vert z_{2}\vert ^{2\kappa}\big) \{ \alpha_{0,1}(
u_{1}-\widetilde{u}) -\alpha_{1,0}z_{1}\} \\
\hphantom{K\big(\eta^{(1)}\big) -K\big( \eta^{(3)}\big) =}{}
- \{ \alpha_{0,1} ( u_{2}-\widetilde{u} ) +\alpha_{1,0}z_{2}\} ^{\ast}H_{1}\rho\big(\vert z_{2}\vert ^{-2\kappa
},\vert z_{2}\vert ^{2\kappa}\big)\\
\hphantom{K\big(\eta^{(1)}\big) -K\big( \eta^{(3)}\big) =}{}
\times \{ \alpha_{0,1}(u_{2}-\widetilde{u}) +\alpha_{1,0}z_{2}\} +O(\delta) .
\end{gather*}
The key fact is that the $\alpha_{00}^{\ast}H_{1}\rho\big( \vert z_{1} \vert ^{-2\kappa}, \vert z_{1} \vert ^{2\kappa}\big) \alpha_{00}$ terms cancel out ($ \vert z_{1} \vert = \vert z_{2} \vert $).

From $\Vert \alpha_{1,0}(x(\widetilde{u},0)) \Vert \leq\frac{(2\lambda) _{2}}{( 1-2\kappa_{0}) }\delta_{1}^{-1}\Vert \alpha_{0}( x(\widetilde{u},0)) \Vert $ and $\Vert \alpha_{0,1}(x(\widetilde{u},0)) \Vert\leq2\lambda\delta_{1}^{-1} \Vert \alpha_{0} ( x ( \widetilde {u},0))\Vert $ (from~(\ref{a0nbd})) $\delta_{1}=\delta_{0}-\vert u\vert =\delta^{1/2}-\delta=\delta^{1/2} (1-\delta^{1/2}) $. Thus the sum of the f\/irst order terms in $K\big(
\eta^{(1) }\big) -K\big( \eta^{(3)}\big)$ is bounded by $c^{\prime\prime\prime} \Vert \alpha_{0} ( x (\widetilde{u},0 ) ) \Vert ^{2}\delta^{1/2-2\vert
\kappa \vert }\big(1-\delta^{1/2}\big) ^{-1} \Vert H_{1}\Vert $, where the constant~$c^{\prime\prime\prime}$ is independent of $x(\widetilde{u},0) $ (but is dependent on $\kappa_{0}$ and~$N$). Note $ \vert u_{1}-\widetilde{u} \vert = \vert u_{2}-\widetilde{u}\vert =\vert z_{1}\vert =\vert z_{2}\vert =\delta$. The second last step is to relate $\Vert \alpha_{0}(x(\widetilde{u},0))\Vert $ to $\Vert L_{1}\big(\eta^{(1)}\big) \Vert $; indeed
\begin{gather*}
L_{1}\big(\eta^{(1)}\big) =\left( \prod\limits_{j=1}^{N}x_{j}^{(1) }\right) ^{-\gamma\kappa}\rho\big(z_{1}^{-\kappa},z_{1}^{\kappa}\big) \left\{ \alpha_{0}( x (\widetilde{u},0)) +\sum_{n=1}^{\infty}\alpha_{n}(x(\widetilde{u},0)) z_{1}^{n}\right\} .
\end{gather*}
Similarly to (\ref{bnd1})
\begin{gather*}
\sum_{n=1}^{\infty}\Vert \alpha_{n}( x( \widetilde{u},0) ) \Vert \vert z\vert ^{n} \leq
c^{\prime}\Vert \alpha_{0}( x(\widetilde{u},0) ) \Vert \left( \frac{\vert z\vert }{\delta_{0}}\right) \left( 1-\frac{\vert z\vert }{\delta_{0}}\right)
^{-2\lambda-1}\\
\hphantom{\sum_{n=1}^{\infty}\Vert \alpha_{n}( x( \widetilde{u},0) ) \Vert \vert z\vert ^{n}}{}
 =c^{\prime} \Vert \alpha_{0} ( x(\widetilde{u},0)) \Vert \delta^{1/2}\big(1-\delta^{1/2}\big) ^{-2\lambda-1}\leq c^{\prime\prime}\Vert \alpha_{0}( x( \widetilde{u},0) ) \Vert \delta^{1/2},
\end{gather*}
(if $\delta<\frac{1}{9}$ then $1-\delta^{1/2}>\frac{2}{3}$); thus
\begin{gather*}
\Vert \alpha_{0}(x(\widetilde{u},0)) \Vert \big( 1-c^{\prime\prime}\delta^{1/2}\big) \leq\delta^{-\vert \kappa\vert } \big\Vert L_{1}\big( \eta^{(1)}\big) \big\Vert .
\end{gather*}
By Theorem \ref{Lbnd}
\begin{gather*}
\big\Vert L\big(\eta^{(1)}\big) \big\Vert \leq
c\prod\limits_{1\leq i<j\leq N-1}\vert x_{i}-x_{j}\vert
^{-\vert \kappa\vert }\prod\limits_{j=1}^{N-2}\big\vert
e^{\mathrm{i}\theta_{j}}-e^{\mathrm{i} ( \theta_{N-1}+\delta^{\prime
} ) }\big\vert ^{-\vert \kappa\vert }\big\vert
e^{\mathrm{i}\theta_{N-1}}-e^{\mathrm{i} ( \theta_{N-1}+\delta^{\prime}) }\big\vert ^{-\vert \kappa\vert }\\
\hphantom{\big\Vert L\big(\eta^{(1)}\big) \big\Vert}{}
 \leq c\delta^{-\vert \kappa\vert \{ (N+1)(N-2) /2+1\} },
\end{gather*}
because the f\/irst two groups of terms satisfy the bound $\vert x_{i}-x_{j} \vert \geq\delta^{1/2}$. Combining everything we obtain the bound
\begin{gather*}
\big\Vert K\big(\eta^{(1)}\big) -K\big( \eta^{(3) }\big) \big\Vert \leq c^{\prime\prime\prime} \Vert
\alpha_{0}(x(\widetilde{u},0)) \Vert^{2}\delta^{1/2-2\vert \kappa\vert }\big(1-\delta^{1/2}\big)
^{-1} \Vert B_{1} \Vert \\
\hphantom{\big\Vert K\big(\eta^{(1)}\big) -K\big( \eta^{(3) }\big) \big\Vert }{}
\leq c^{\prime\prime} \Vert H_{1} \Vert \delta^{1/2-2 \vert \kappa \vert -\vert \kappa\vert \{ (N+1)(N-2) +2\} }.
\end{gather*}
The constant is independent of $\eta^{(1) }$ and the exponent on~$\delta$ is $\frac{1}{2}-\vert \kappa\vert \big( N^{2}-N+2\big)$. Thus the integral of part~(3) goes to zero as $\delta \rightarrow0$ if $\vert \kappa\vert <\big( 2\big(N^{2}-N+2\big) \big) ^{-1}$. This is a crude bound, considering that we
know everything works for $-1/h_{\tau}<\kappa<1/h_{\tau}$, but as we will see, an open interval of $\kappa$ values suf\/f\/ices.

\begin{Theorem}\label{suffctH}If there exists a Hermitian matrix $H$ such that
\begin{gather*} \upsilon H=H\upsilon \qquad \text{and} \qquad ( L_{1}(x_{0}) ^{\ast})^{-1}HL_{1}(x_{0}) ^{-1}
 \end{gather*}
 commutes with~$\sigma$, and $-\big(2\big( N^{2}-N+2\big) \big) ^{-1}<\kappa<\big( 2\big( N^{2}-N+2\big) \big) ^{-1}$ then
\begin{gather*}
\int_{\mathbb{T}^{N}}\{ ( x_{i}\mathcal{D}_{i}f(x)
) ^{\ast}L(x) ^{\ast}HL(x) g(x) -f(x) ^{\ast}L(x) ^{\ast}HL(x) x_{i}\mathcal{D}_{i}g(x) \} \mathrm{d}m(x) =0
\end{gather*}
for $f,g\in C^{(1) }\big( \mathbb{T}^{N};V_{\tau}\big)$ and $1\leq i\leq N$.
\end{Theorem}

It is important that we can derive uniqueness of $H$ from the relation, because the conditions $\langle wf,wg\rangle = \langle f,g\rangle $, $\langle x_{i}f,x_{i}g\rangle =\langle f,g \rangle $, and $ \langle x_{i}\mathcal{D}_{i}f,g \rangle = \langle f,x_{i}\mathcal{D}_{i}g \rangle $ for $w\in\mathcal{S}_{N}$ and $1\leq i\leq N$ determine the Hermitian form uniquely up to multiplication by a constant. Thus the measure $K(x) \mathrm{d}m(x)$ is similarly determined, by the density of Laurent polynomials.

\section{The orthogonality measure on the torus}\label{orthmu}

At this point there are two logical threads in the development. On the one hand there is a suf\/f\/icient condition implying the desired orthogonality measure is of the form $L^{\ast}HL\mathrm{d}m$, specif\/ically if $H$ commutes with $\upsilon$, $( L_{1}(x_{0}) ^{\ast})^{-1}HL_{1}(x_{0}) ^{-1}$ commutes with $\sigma$, and $\vert \kappa\vert <( 2( N^{2}-N+2))^{-1}$. However we have not yet proven that $H$ exists. On the other hand in~\cite{Dunkl2016} we showed that there does exist an orthogonality measure of the form $\mathrm{d}\mu=\mathrm{d}\mu_{S}+L^{\ast}HL\mathrm{d}m$ where $\operatorname{spt}\mu_{S}\subset\mathbb{T}^{N}\backslash\mathbb{T}_{\rm reg}^{N}$, $H$~commutes with $\upsilon$, and $-1/h_{\tau}<\kappa<1/h_{\tau}$ (the support of a Baire measure $\nu$, denoted by $\operatorname{spt}\nu$, is the smallest compact set whose complement has $\nu$-measure zero). In the next sections we will show that $( L_{1}(x_{0}) ^{\ast}) ^{-1}HL_{1}(x_{0}) ^{-1}$ commutes with $\sigma$ and that $H$ is an analytic function of $\kappa$ in a complex neighborhood of this interval. Combined with the above suf\/f\/icient condition this is enough to show that there is no singular part, that is, $\mu_{S}=0$. The proof involves the formal dif\/ferential equation satisf\/ied by the Fourier--Stieltjes series of~$\mu$, which is used to show $\mu_{S}=0$ on $\big\{ x\in\mathbb{T}^{N}\colon \#\{x_{j}\} _{j=1}^{N}=N-1\big\} $ (that is, $x$ has at least $N-1$ distinct components). In turn this implies $( L_{1}^{\ast}(x_{0})) ^{-1}HL_{1}(x_{0}) ^{-1}$ commutes with~$\sigma$. The proofs unfortunately are not short. In the sequel $H$ refers to the Hermitian matrix in the formula for $\mathrm{d}\mu$ and $K$ denotes~$L^{\ast}HL$. Also $H$ is positive-def\/inite since the measure~$\mu$ is positive (else there exists a vector $v$ with $Hv=0$ and then the $C^{(1) }\big( \mathbb{T}_{\rm reg}^{N};V_{\tau}\big) $ function given by $f(x) :=L(x) ^{-1}vg(x) $ where~$g$ is a~smooth scalar nonnegative function with support in a~suf\/f\/iciently small neighborhood of $x_{0}$, has norm $ \langle f,f \rangle =0$, a~contradiction). Thus $H$ has a~positive-def\/inite square root $C$ which commutes with $\upsilon$. Now extend~$CL(x) $ from $\mathcal{C}_{0}$ to all of $\mathbb{T}_{\rm reg}^{N}$ by Def\/inition~\ref{DefL(x)T} and so $K(x) =L^{\ast}(x) C^{\ast}CL(x) $ for all $x\in\mathbb{T}_{\rm reg}^{N}$ (this follows from $K(xw) =\tau(w) ^{-1}K(x) \tau(w) $).

Furthermore $\int_{\mathbb{T}^{N}} \Vert K(x)\Vert \mathrm{d}m(x) <\infty$ because $K\mathrm{d}m$ is the absolutely continuous part of the f\/inite Baire measure~$\mu$.

We will show that $(L_{1}^{\ast}(x_{0})) ^{-1}C^{\ast}CL_{1}(x_{0}) ^{-1}$ commutes with $\sigma$. The proof begins by establishing a recurrence relation for the Fourier coef\/f\/icients of $K(x) $, which comes from equation~(\ref{Kdieq}). For~$F(x) $ integrable on~$\mathbb{T}^{N}$, possibly matrix-valued, and $\alpha\in\mathbb{Z}^{N}$ let $\widehat{F}_{\alpha}=\int_{\mathbb{T}^{N}}F(x) x^{-\alpha}\mathrm{d}m(x)$. Clearly $\int_{\mathbb{T}^{N}}x^{\beta}F (x) x^{-\alpha}\mathrm{d}m(x) =\widehat{F}_{\alpha-\beta}$; and if $\partial_{i}F(x) $ is also integrable then
(integration-by-parts)
\begin{gather}
\int_{\mathbb{T}^{N}}x_{i}\partial_{i}F(x) x^{-\alpha}\mathrm{d}m(x) =\alpha_{i}\int_{\mathbb{T}^{N}}F (x) x^{-\alpha}\mathrm{d}m(x) .\label{diffFC}
\end{gather}
For a subset $J\subset\{1,2,\ldots,N\} $ let $\varepsilon_{J}\in\mathbb{N}_{0}^{N}$ be def\/ined by $( \varepsilon_{J}) _{i}=1$ if $i\in J$ and $=0$ otherwise; also $\varepsilon_{i}:=\varepsilon_{\{i\} }$. For $1\leq i\leq N$ let
\begin{gather*}
E_{i} :=\{1,2,\ldots,N\} \backslash\{i\} ,\qquad E_{ij} :=E_{i}\backslash\{j\} ,\\
p_{i}(x) :=\prod\limits_{j\neq i}(x_{i}-x_{j}) =\sum\limits_{\ell=0}^{N-1}(-1) ^{\ell}x_{i}^{N-1-\ell}\sum\limits_{J\subset E_{i},\#J=\ell}x^{\varepsilon_{J}}.
\end{gather*}
Equation (\ref{Kdieq}) can be rewritten as
\begin{gather}
p_{i}(x) x_{i}\partial_{i}K(x) =\kappa\sum_{j\neq i}\prod\limits_{\ell\neq i,j}(x_{i}-x_{\ell}) \{ x_{j}\tau((i,j)) K(x) +K(x) \tau((i,j)) x_{i}\};\label{Kdiffeq1}
\end{gather}
this is a polynomial relation which shows that $p_{i}(x) x_{i}\partial_{i}K(x) $ is integrable and which has implications for the Fourier coef\/f\/icients of~$K$.

\begin{Proposition}
For $1\leq i\leq N$ and $\alpha\in\mathbb{Z}^{N}$ the Fourier coefficients $\widehat{K}$ satisfy
\begin{gather}
 \sum_{\ell=0}^{N-1}(-1) ^{\ell}(\alpha_{i}+\ell) \sum\limits_{J\subset E_{i},\, \#J=\ell}\widehat{K}_{\alpha+\ell\varepsilon_{i}-\varepsilon_{J}}\nonumber\\
\qquad{} =\kappa\sum_{j\neq i}\sum_{\ell=0}^{N-2}(-1) ^{\ell}\sum_{J\subset E_{ij},\, \#J=\ell}\big\{ \tau ( (i,j)) \widehat{K}_{\alpha+\ell\varepsilon_{i}-\varepsilon_{j}-\varepsilon_{J}}+\widehat{K}_{\alpha+(l+1) \varepsilon
_{i}-\varepsilon_{J}}\tau((i,j))\big\}.\label{FCrec}
\end{gather}
\end{Proposition}

\begin{proof} Multiply both sides of (\ref{Kdiffeq1}) by $x_{i}^{1-N}$; this makes the terms homogeneous of degree zero. Suppose $j\neq i$ then
\begin{gather*}
x_{i}^{1-N}\prod\limits_{\ell\neq i,j}(x_{i}-x_{\ell}) =\prod\limits_{\ell\neq i,j}\left( 1-\frac{x_{\ell}}{x_{i}}\right)
=\sum_{\ell=0}^{N-2}(-1) ^{\ell}x_{i}^{-\ell}\sum_{J\subset E_{ij},\, \#J=\ell}x^{\varepsilon_{J}}.
\end{gather*}
Multiply the right side by $x^{-\alpha}\mathrm{d}m(x) $ and integrate over $\mathbb{T}^{N}$ to obtain%
\begin{gather*}
\kappa\sum_{j\neq i}\sum_{\ell=0}^{N-2}(-1) ^{\ell}\sum_{J\subset E_{ij},\, \#J=\ell}\big\{ \tau( (i,j)) \widehat{K}_{\alpha+\ell\varepsilon_{i}-\varepsilon_{j}%
-\varepsilon_{J}}+\widehat{K}_{\alpha+(l+1) \varepsilon_{i}-\varepsilon_{J}}\tau((i,j)) \big\} .
\end{gather*}
The sum is zero unless $\alpha\in\boldsymbol{Z}_{N}$ where $\boldsymbol{Z}_{N}:=\Big\{ \alpha\in\mathbb{Z}^{N}\colon \sum\limits_{j=1}^{N}\alpha_{j}=0\Big\} $, by the homogeneity. For the left side start with~(\ref{diffFC}) applied to $x_{i}^{1-N}p_{i}(x) x_{i}\partial_{i}K(x) $
\begin{gather*}
 (\alpha_{i}+N-1) \int_{\mathbb{T}^{N}}p_{i}(x)K(x) x_{i}^{1-N}x^{-\alpha}\mathrm{d}m(x) \\
\qquad{} =\int_{\mathbb{T}^{N}}\big\{ ( x_{i}\partial_{i}p_{i}(
x) ) K(x) +p_{i}(x) (x_{i}\partial_{i}K(x) ) \big\} x_{i}^{1-N}x^{-\alpha}\mathrm{d}m(x) ,\\
 \int_{\mathbb{T}^{N}}p_{i}(x)( x_{i}\partial _{i}K(x)) x_{i}^{1-N}x^{-\alpha}\mathrm{d}m (x) \\
\qquad{} =\int_{\mathbb{T}^{N}}( (\alpha_{i}+N-1) p_{i} ( x) -x_{i}\partial_{i}p_{i}(x) ) K(x)
x_{i}^{1-N}x^{-\alpha}\mathrm{d}m(x) \\
\qquad{} =\int_{\mathbb{T}^{N}}\sum_{\ell=0}^{N-1}(-1) ^{\ell} (\alpha_{i}+\ell ) x_{i}^{-\ell}\sum\limits_{J\subset E_{i},\#J=\ell
}x^{\varepsilon_{J}}K(x) x^{-\alpha}\mathrm{d}m(x)\\
\qquad{} =\sum_{\ell=0}^{N-1}(-1) ^{\ell} ( \alpha_{i}+\ell ) \sum\limits_{J\subset E_{i},\#J=\ell}\widehat{K}_{\alpha+\ell\varepsilon_{i}-\varepsilon_{J}}.
\end{gather*}

Combining the two sides f\/inishes the proof. If $\alpha\notin\boldsymbol{Z}_{N}$ then both sides are trivially zero.
\end{proof}

This system of recurrences has the easy (and quite undesirable) solution~$\widehat{K}_{\alpha}=I$ for all $\alpha\in\boldsymbol{Z}_{N}$ and~$0$ otherwise. The right side becomes $2\kappa\sum\limits_{j\neq i}\tau ((i,j)) \sum\limits_{\ell=0}^{N-2}(-1) ^{\ell}\binom {N-2}{\ell}=0$ (for $N\geq3$, an underlying assumption), and the left side is $\sum\limits_{\ell=0}^{N-1}(-1) ^{\ell}(\alpha_{i}+\ell) \binom{N-1}{\ell}I=0$. This~$\widehat{K}$ corresponds to the measure $\frac {1}{2\pi}\mathrm{d}\theta$ on the circle $\big\{ e^{\mathrm{i}\theta} (1,\ldots,1 ) \colon -\pi<\theta\leq\pi\big\}$. Next we show that $\widehat{\mu}_{\alpha}:=\int_{\mathbb{T}^{N}}x^{-\alpha}\mathrm{d}\mu (x ) $ satisf\/ies the same recurrences. Proposition~5.2 of~\cite{Dunkl2016} asserts that if $\alpha,\beta\in\mathbb{N}_{0}^{N}$ and $\sum\limits_{j=1}^{N}(\alpha_{j}-\beta_{j}) =0$ then
\begin{gather}
( \alpha_{i}-\beta_{i}) \widehat{\mu}_{\alpha-\beta}
=\kappa\sum_{\alpha_{j}>\alpha_{i}}\sum_{\ell=1}^{\alpha_{j}-\alpha_{i}}\tau((i,j)) \widehat{\mu}_{\alpha+\ell(
\varepsilon_{i}-\varepsilon_{j}) -\beta}\nonumber\\
\hphantom{( \alpha_{i}-\beta_{i}) \widehat{\mu}_{\alpha-\beta} }{}
 -\kappa\sum_{\alpha_{i}>\alpha_{j}}\sum_{\ell=0}^{\alpha_{i}-\alpha_{j}
-1}\tau((i,j)) \widehat{\mu}_{\alpha+\ell(
\varepsilon_{j}-\varepsilon_{i}) -\beta} -\kappa\sum_{\beta_{j}>\beta_{i}}\sum_{\ell=1}^{\beta_{j}-\beta_{i}}\widehat{\mu}_{\alpha-\ell( \varepsilon_{i}-\varepsilon_{j})
-\beta}\tau((i,j)) \nonumber\\
\hphantom{( \alpha_{i}-\beta_{i}) \widehat{\mu}_{\alpha-\beta} }{}
 +\kappa\sum_{\beta_{i}>\beta_{j}}\sum_{\ell=0}^{\beta_{i}-\beta_{j}-1}\widehat{\mu}_{\alpha-\ell(\varepsilon_{j}-\varepsilon_{i})
-\beta}\tau((i,j)) .\label{A(a-b)}
\end{gather}
The relation $\tau(w) ^{\ast}\widehat{\mu}_{w\alpha}\tau (w) =\widehat{\mu}_{\alpha}$ is shown in \cite[Theorem~4.4]{Dunkl2016}. Introduce Laurent series $\sum\limits_{\alpha\in\boldsymbol{Z}_{N}}B_{\alpha}^{(i,j) }x^{\alpha}$ ($i\neq j$) satisfying
\begin{gather*}
B_{\alpha}^{(i,j) }-B_{\alpha+\varepsilon_{i}-\varepsilon_{j}}^{(i,j) } =\widehat{\mu}_{\alpha},\qquad
B_{\alpha-\alpha_{j}(\varepsilon_{j}-\varepsilon_{i}) }^{(i,j) } =0,
\end{gather*}
note
\begin{gather*}
\alpha-\alpha_{j}(\varepsilon_{j}-\varepsilon_{i}) = \big(\ldots,\overset{i}{\alpha_{i}+\alpha_{j}},\ldots,\overset{j}{0},\ldots
,\overset{\ell}{\alpha_{\ell}},\ldots \big) , \qquad \ell\neq i,j.
\end{gather*}
The purpose of the def\/inition is to produce a formal Laurent series satisfying
\begin{gather*}
\left( 1-\frac{x_{j}}{x_{i}}\right) \sum_{\alpha}B_{\alpha}^{(i,j) }x^{\alpha}=\sum_{\alpha}\widehat{\mu}_{\alpha}x^{\alpha}.
\end{gather*} The ambiguity in the solution is removed by the second condition (note that $\sum_{\alpha}\big( B_{\alpha}^{(i,j) }-cI\big) x^{\alpha}$ also solves the f\/irst equation for any constant~$c$).

\begin{Proposition}\label{Bij-Bji}Suppose $i\neq j$ and $\alpha\in\boldsymbol{Z}_{N}$ then $B_{\alpha}^{(i,j)}\tau((i,j))=\tau((i,j))B_{(i,j)\alpha}^{(j,i)}$.
\end{Proposition}

\begin{proof} Start with $\widehat{\mu}_{\alpha}\tau((i,j)) =\tau((i,j)) \widehat{\mu}_{(i,j) \alpha}$ and the def\/ining relations
\begin{gather*}
B_{\alpha}^{(i,j) }\tau((i,j))-B_{\alpha+\varepsilon_{i}-\varepsilon_{j}}^{(i,j) }\tau((i,j)) =\widehat{\mu}_{\alpha}\tau((i,j)) ,\\
\tau((i,j)) B_{(i,j) \alpha}^{(j,i) }-\tau((i,j)) B_{(i,j) \alpha+\varepsilon_{j}-\varepsilon_{i}}^{(j,i)}=\tau((i,j)) \widehat{\mu}_{(i,j)\alpha};
\end{gather*}
subtract the second equation from the f\/irst:
\begin{gather*}
B_{\alpha}^{(i,j) }\tau((i,j))-\tau((i,j)) B_{(i,j) \alpha}^{(j,i) }=B_{\alpha+\varepsilon_{i}-\varepsilon_{j}}^{(i,j) }\tau((i,j)) -\tau((i,j)) B_{(i,j) \alpha+\varepsilon_{j}-\varepsilon_{i}}^{(j,i) }.
\end{gather*}
By two-sided induction
\begin{gather*}
B_{\alpha}^{(i,j) }\tau((i,j))-\tau((i,j)) B_{(i,j) \alpha}^{(j,i) }=B_{\alpha+s( \varepsilon_{i}-\varepsilon_{j}) }^{(i,j) }\tau((i,j))-\tau((i,j)) B_{(i,j)
\alpha+s(\varepsilon_{j}-\varepsilon_{i}) }^{(j,i)}
\end{gather*}
for all $s\in\mathbb{Z}$, in particular for $s=\alpha_{j}$ where the right hand side vanishes by def\/inition.
\end{proof}

\begin{Theorem}For $\gamma\in\boldsymbol{Z}_{N}$ and $1\leq i\leq N$
\begin{gather}
\gamma_{i}\widehat{\mu}_{\gamma}=\kappa\sum_{j\neq i}\big\{{-}\tau ((i,j)) B_{\gamma}^{(j,i) }+B_{\gamma}^{(i,j) }\tau((i,j)) \big\}.\label{Btomu}
\end{gather}
\end{Theorem}

\begin{proof}
The proof involves a number of cases (for each $(i,j) $ whether $\gamma_{i}\geq0$ or $\gamma_{i}<0$, $\gamma_{j}\geq0$ or $\gamma_{j}<0$). Consider equation~(\ref{A(a-b)}), in the terms on the f\/irst line (with $\tau((i,j))$ acting on the left) use the substitution $\widehat{\mu}_{\delta}=B_{\delta}^{(j,i)}-B_{\delta-\varepsilon_{i}+\varepsilon_{j}}^{(j,i) }$, and for the terms on the second line (with $\tau((i,j))$ acting on the right) use the substitution $\widehat{\mu}_{\delta}=B_{\delta }^{(i,j) }-B_{\delta+\varepsilon_{i}-\varepsilon_{j}}^{(i,j)}$. Set $\alpha_{\ell}=\max(\gamma_{\ell},0) $ and $\beta_{\ell}=\max ( 0,-\gamma_{\ell} ) $ for $1\leq\ell\leq N$, thus $\gamma=\alpha-\beta$. The left hand side is $( \alpha_{i}-\beta_{i}) \widehat{\mu}_{\alpha-\beta}=\gamma_{i}\widehat{\mu}_{\gamma}$. We consider two possibilities separately: (i) $\alpha_{i}\geq0$, $\beta_{i}=0$; (ii) $\alpha_{i}=0$, $\beta_{i}>0$; and describe the typical $\tau((i,j)) $ terms. The sums over~$\ell$ telescope. In the following any term of the form $\tau((i,j)) \widehat{\mu}_{\cdot}$ or $\widehat{\mu}_{\cdot}\tau((i,j)) $ not mentioned explicitly is zero. Proposition~\ref{Bij-Bji} is used in each case. For case (i) and $\alpha_{j}>\alpha_{i}$
\begin{gather*}
\tau((i,j)) \sum_{\ell=1}^{\alpha_{j}-\alpha_{i}}\widehat{\mu}_{\alpha+\ell ( \varepsilon_{i}-\varepsilon_{j} )-\beta}=\tau((i,j)) \sum_{\ell=1}^{\alpha
_{j}-\alpha_{i}}\big( B_{\gamma+\ell( \varepsilon_{i}-\varepsilon_{j}) }^{(j,i) }-B_{\gamma+\ell ( \varepsilon_{i}-\varepsilon_{j}) -\varepsilon_{i}+\varepsilon_{j}}^{(j,i) }\big) \\
\hphantom{\tau((i,j)) \sum_{\ell=1}^{\alpha_{j}-\alpha_{i}}\widehat{\mu}_{\alpha+\ell ( \varepsilon_{i}-\varepsilon_{j} )-\beta}}{}
=\tau((i,j)) \sum_{\ell=1}^{\alpha_{j}-\alpha
_{i}}\big( B_{\gamma+\ell ( \varepsilon_{i}-\varepsilon_{j} )
}^{(j,i) }-B_{\gamma+(\ell-1) (\varepsilon_{i}-\varepsilon_{j}) }^{(j,i) }\big)\\
\hphantom{\tau((i,j)) \sum_{\ell=1}^{\alpha_{j}-\alpha_{i}}\widehat{\mu}_{\alpha+\ell ( \varepsilon_{i}-\varepsilon_{j} )-\beta}}{}
=\tau((i,j)) \big( B_{\gamma+ ( \alpha_{j}-\alpha_{i})( \varepsilon_{i}-\varepsilon_{j})}^{(j,i) }-B_{\gamma}^{(j,i) }\big) \\
\hphantom{\tau((i,j)) \sum_{\ell=1}^{\alpha_{j}-\alpha_{i}}\widehat{\mu}_{\alpha+\ell ( \varepsilon_{i}-\varepsilon_{j} )-\beta}}{}
=\tau((i,j)) \big( B_{(i,j)\gamma}^{(j,i) }-B_{\gamma}^{(j,i) }\big)=-\tau((i,j)) B_{\gamma}^{(j,i)}+B_{\gamma}^{(i,j) }\tau((i,j))
\end{gather*}
For case (i) and $\alpha_{i}>\alpha_{j}\geq0=\beta_{j}$
\begin{gather*}
-\tau((i,j)) \sum_{\ell=0}^{\alpha_{i}-\alpha_{j}-1}\widehat{\mu}_{\alpha+\ell ( \varepsilon_{j}-\varepsilon_{i}) -\beta} =-\tau((i,j)) \sum
_{\ell=0}^{\alpha_{i}-\alpha_{j}-1}\big( B_{\gamma+\ell (\varepsilon_{j}-\varepsilon_{i}) }^{(j,i) }-B_{\gamma+( \ell+1)( \varepsilon_{i}-\varepsilon
_{j}) }^{(j,i) }\big) \\
\hphantom{-\tau((i,j)) \sum_{\ell=0}^{\alpha_{i}-\alpha_{j}-1}\widehat{\mu}_{\alpha+\ell ( \varepsilon_{j}-\varepsilon_{i}) -\beta}}{}
 =-\tau((i,j)) \big( B_{\gamma}^{(j,i) }-B_{(i,j) \gamma}^{(j,i) }\big)\\
\hphantom{-\tau((i,j)) \sum_{\ell=0}^{\alpha_{i}-\alpha_{j}-1}\widehat{\mu}_{\alpha+\ell ( \varepsilon_{j}-\varepsilon_{i}) -\beta}}{}
=-\tau((i,j)) B_{\gamma}^{(j,i)}+B_{\gamma}^{(i,j) }\tau((i,j)) ,
\end{gather*}
note $\gamma+(\alpha_{i}-\alpha_{j}) ( \varepsilon_{j}-\varepsilon_{i}) =(i,j) \gamma$. For case (i) and $\alpha_{i}>\alpha_{j}=0>-\beta_{j}$
\begin{gather*}
-\tau((i,j)) \sum_{\ell=0}^{\alpha_{i}-1}\widehat{\mu}_{\alpha+\ell(\varepsilon_{j}-\varepsilon_{i})-\beta} =-\tau((i,j)) \sum_{\ell=0}^{\alpha_{i}-1}\big( B_{\gamma+\ell ( \varepsilon_{j}-\varepsilon_{i}) }^{(j,i) }-B_{\gamma+(\ell+1) (\varepsilon_{j}-\varepsilon_{i}) }^{(j,i) }\big) \\
\hphantom{-\tau((i,j)) \sum_{\ell=0}^{\alpha_{i}-1}\widehat{\mu}_{\alpha+\ell(\varepsilon_{j}-\varepsilon_{i})-\beta}}{}
 =-\tau((i,j)) \big( B_{\gamma}^{(j,i)}-B_{\gamma+\gamma_{i}( \varepsilon_{j}-\varepsilon_{i}) }^{(j,i) }\big) ,
\\
-\sum_{\ell=1}^{\beta_{j}}\widehat{\mu}_{\alpha-\ell(\varepsilon_{i}-\varepsilon_{j}) -\beta}\tau((i,j))=
-\sum_{\ell=1}^{\beta_{j}}\big( B_{\gamma-\ell ( \varepsilon_{i}-\varepsilon_{j}) }^{(i,j) }-B_{\gamma-\ell(\varepsilon_{i}-\varepsilon_{j})+\varepsilon_{i}-\varepsilon_{j}}^{(i,j) }\big) \tau((i,j)) \\
\hphantom{-\sum_{\ell=1}^{\beta_{j}}\widehat{\mu}_{\alpha-\ell(\varepsilon_{i}-\varepsilon_{j}) -\beta}\tau((i,j))}{}
=-\sum_{\ell=1}^{\beta_{j}}\big( B_{\gamma-\ell ( \varepsilon_{i}-\varepsilon_{j}) }^{(i,j) }-B_{\gamma-(\ell-1) ( \varepsilon_{i}-\varepsilon_{j}) }^{(
i,j) }\big) \tau((i,j))\\
\hphantom{-\sum_{\ell=1}^{\beta_{j}}\widehat{\mu}_{\alpha-\ell(\varepsilon_{i}-\varepsilon_{j}) -\beta}\tau((i,j))}{}
 =\big(B_{\gamma}^{(i,j) }-B_{\gamma+\gamma_{j}( \varepsilon_{i}-\varepsilon_{j}) }^{(i,j) }\big) \tau((i,j))
\end{gather*}
let $\delta=\gamma+\gamma_{i}(\varepsilon_{j}-\varepsilon_{i}) $ then $\delta_{k}=\gamma_{k}$ for $k\neq i,j$, $\delta_{i}=0$, and $\delta_{j}=\gamma_{i}+\gamma_{j}$; also $(i,j) \delta=\gamma+\gamma_{j}( \varepsilon_{i}-\varepsilon_{j})$. Thus the sum of the terms for this case is
\begin{gather*}
-\tau((i,j)) B_{\gamma}^{(j,i)}+B_{\gamma}^{(i,j) }\tau((i,j)) +\tau((i,j)) B_{\delta}^{(j,i)}+B_{(i,j) \delta}^{(i,j) }\tau( (i,j))\\
\qquad{} =-\tau((i,j)) B_{\gamma}^{(j,i) }+B_{\gamma}^{(i,j) }\tau((i,j)) .
\end{gather*}
For case (ii) and $\beta_{j}=-\gamma_{j}>\beta_{i}=-\gamma_{i}>0$
\begin{gather*}
-\sum_{\ell=1}^{\beta_{j}-\beta_{i}}\widehat{\mu}_{\alpha-\ell (\varepsilon_{i}-\varepsilon_{j}) -\beta}\tau( (i,j)) =-\sum_{\ell=1}^{\beta_{j}-\beta_{i}}\big( B_{\gamma-\ell (\varepsilon_{i}-\varepsilon_{j}) }^{(i,j) }-B_{\gamma-(\ell-1) ( \varepsilon_{i}-\varepsilon_{j}) }^{(i,j) }\big) \tau((i,j)) \\
\hphantom{-\sum_{\ell=1}^{\beta_{j}-\beta_{i}}\widehat{\mu}_{\alpha-\ell (\varepsilon_{i}-\varepsilon_{j}) -\beta}\tau( (i,j))}{}
=\big( {-}B_{\gamma- ( \gamma_{i}-\gamma_{j} ) (\varepsilon_{i}-\varepsilon_{j}) }^{(i,j) }+B_{\gamma}^{(i,j) }\big) \tau((i,j))\\
\hphantom{-\sum_{\ell=1}^{\beta_{j}-\beta_{i}}\widehat{\mu}_{\alpha-\ell (\varepsilon_{i}-\varepsilon_{j}) -\beta}\tau( (i,j))}{}
=\big( {-}B_{(i,j) \gamma}^{(i,j) }+B_{\gamma}^{(i,j) }\big) \tau((i,j))\\
\hphantom{-\sum_{\ell=1}^{\beta_{j}-\beta_{i}}\widehat{\mu}_{\alpha-\ell (\varepsilon_{i}-\varepsilon_{j}) -\beta}\tau( (i,j))}{}
=-\tau((i,j)) B_{\gamma}^{(j,i)}+B_{\gamma}^{(i,j) }\tau((i,j)) .
\end{gather*}
For case (ii) and $\beta_{i}>\beta_{j}=-\gamma_{j}\geq0$ (and $\alpha_{j}=0$)
\begin{gather*}
\sum_{\ell=0}^{\beta_{i}-\beta_{j}-1}\widehat{\mu}_{\alpha-\ell (\varepsilon_{j}-\varepsilon_{i}) -\beta}\tau( (i,j)) =\sum_{\ell=0}^{\beta_{i}-\beta_{j}-1}\big( B_{\gamma-\ell(\varepsilon_{j}-\varepsilon_{i}) }^{(i,j) }-B_{\gamma-(\ell+1)( \varepsilon_{j}-\varepsilon
_{i}) }^{(i,j) }\big) \tau ( (i,j)) \\
\hphantom{\sum_{\ell=0}^{\beta_{i}-\beta_{j}-1}\widehat{\mu}_{\alpha-\ell (\varepsilon_{j}-\varepsilon_{i}) -\beta}\tau( (i,j))}{}
=\big( B_{\gamma}^{(i,j) }-B_{\gamma- ( \gamma_{j}-\gamma_{i}) (\varepsilon_{j}-\varepsilon_{i}) }^{(i,j) }\big) \tau((i,j))\\
\hphantom{\sum_{\ell=0}^{\beta_{i}-\beta_{j}-1}\widehat{\mu}_{\alpha-\ell (\varepsilon_{j}-\varepsilon_{i}) -\beta}\tau( (i,j))}{}
 =\big(B_{\gamma}^{(i,j) }-B_{(i,j) \gamma}^{(i,j) }\big) \tau((i,j)) \\
\hphantom{\sum_{\ell=0}^{\beta_{i}-\beta_{j}-1}\widehat{\mu}_{\alpha-\ell (\varepsilon_{j}-\varepsilon_{i}) -\beta}\tau( (i,j))}{}
=-\tau((i,j)) B_{\gamma}^{(j,i)}+B_{\gamma}^{(i,j) }\tau((i,j)) .
\end{gather*}
For case (ii) and $-\beta_{i}=\gamma_{i}<0<\gamma_{j}=\alpha_{j}$ (and $\beta_{j}=0$)
\begin{gather*}
 \tau((i,j)) \sum_{\ell=1}^{\alpha_{j}}\widehat{\mu}_{\alpha+\ell( \varepsilon_{i}-\varepsilon_{j})
-\beta}+\sum_{\ell=0}^{\beta_{i}-1}\widehat{\mu}_{\alpha-\ell(\varepsilon_{j}-\varepsilon_{i}) -\beta}\tau( (i,j)) \\
\qquad{} =\tau((i,j)) \sum_{\ell=1}^{\alpha_{j}}\big(B_{\gamma+\ell ( \varepsilon_{i}-\varepsilon_{j} ) }^{(j,i) }-B_{\gamma+(\ell-1) ( \varepsilon
_{i}-\varepsilon_{j}) }^{(j,i) }\big) \\
\qquad\quad{} +\sum_{\ell=0}^{\beta_{i}-1}\big( B_{\gamma-\ell ( \varepsilon_{j}-\varepsilon_{i}) }^{(i,j) }-B_{\gamma-(\ell+1) (\varepsilon_{j}-\varepsilon_{i}) }^{(i,j) }\big) \tau((i,j)) \\
\qquad{} =\tau((i,j)) \big( B_{\gamma+\gamma_{j}( \varepsilon_{i}-\varepsilon_{j}) }^{(j,i)}-B_{\gamma}^{(j,i) }\big) +\big( B_{\gamma}^{(i,j)}-B_{\gamma+\gamma_{i}( \varepsilon_{j}-\varepsilon_{i}) }^{(i,j) }\big) \tau ( (i,j)) \\
\qquad{} =-\tau((i,j)) B_{\gamma}^{(j,i) }+B_{\gamma}^{(i,j) }\tau((i,j)) ,
\end{gather*}
because $(i,j) ( \gamma+\gamma_{j}( \varepsilon_{i}-\varepsilon_{j}) ) =\gamma+\gamma_{i}( \varepsilon_{j}-\varepsilon_{i}) $. In the trivial case $\gamma_{i}=\gamma_{j}$ so that $(i,j) \gamma=\gamma$ where are no nonzero $\tau ((i,j)) $ terms the equation $-\tau( (i,j)) B_{\gamma}^{(j,i) }-B_{\gamma}^{(
i,j) }\tau((i,j)) =0$ applies. Thus in each case and for each $j\neq i$ the right hand side contains the expression $-\kappa\big( \tau((i,j)) B_{\gamma}^{(
j,i) }-B_{\gamma}^{(i,j) }\tau( (i,j)) \big)$.
\end{proof}

In the following there is no implied claim about convergence, because any term $x^{\alpha}$ appears only a f\/inite number of times in the equation.

\begin{Theorem}
For $1\leq i\leq N$ the formal Laurent series $F(x) :=\sum\limits_{\alpha\in\boldsymbol{Z}_{N}}\widehat{\mu}_{\alpha}x^{\alpha}$ satisfies the equation
\begin{gather}
p_{i}(x) x_{i}\partial_{i}F(x) =\kappa\sum_{j\neq i}\prod\limits_{\ell\neq i,j}(x_{i}-x_{\ell}) \{ x_{j}\tau((i,j)) F(x) +F(x) \tau((i,j)) x_{i}\}
.\label{diffKLs}
\end{gather}
\end{Theorem}

\begin{proof}
Start with multiplying equation (\ref{Btomu}) by $x_{i}^{1-N}p_{i}(x) x^{\gamma}$ and sum over $\gamma\in\boldsymbol{Z}_{N}$ to obtain
\begin{gather*}
\prod\limits_{j=1,\, j\neq i}^{N}\left( 1-\frac{x_{j}}{x_{i}}\right)\sum_{\gamma\in\boldsymbol{Z}_{N}}\gamma_{i}\widehat{\mu}_{\gamma}x^{\gamma}
 =\kappa\sum_{j\neq i}\prod\limits_{k\neq i,j}\left( 1-\frac{x_{k}}{x_{i}}\right) \left( 1-\frac{x_{j}}{x_{i}}\right) \\
 \hphantom{\prod\limits_{j=1,j\neq i}^{N}\left( 1-\frac{x_{j}}{x_{i}}\right)\sum_{\gamma\in\boldsymbol{Z}_{N}}\gamma_{i}\widehat{\mu}_{\gamma}x^{\gamma}}{}
 \times\left\{ -\tau((i,j)) \sum_{\gamma\in\boldsymbol{Z}_{N}}B_{\gamma}^{(j,i) }x^{\gamma}+\sum_{\gamma\in\boldsymbol{Z}_{N}}B_{\gamma}^{(i,j) }x^{\gamma}%
\tau((i,j)) \right\} .
\end{gather*}
By construction
\begin{gather*}
\left( 1-\frac{x_{j}}{x_{i}}\right) \sum_{\gamma\in\boldsymbol{Z}_{N}}B_{\gamma}^{(i,j) }x^{\gamma}=\sum_{\gamma\in\boldsymbol{Z}_{N}}\big( B_{\gamma}^{(i,j) }-B_{\gamma+\varepsilon_{i}-\varepsilon_{j}}^{(i,j) }\big) x^{\gamma}=\sum_{\gamma\in\boldsymbol{Z}_{N}}\widehat{\mu}_{\gamma}x^{\gamma}
\end{gather*}
and
\begin{gather*}
\left( 1-\frac{x_{j}}{x_{i}}\right) \sum_{\gamma\in\boldsymbol{Z}_{N}}B_{\gamma}^{(j,i) }x^{\gamma}=-\frac{x_{j}}{x_{i}}\left(1-\frac{x_{i}}{x_{j}}\right) \sum_{\gamma\in\boldsymbol{Z}_{N}}B_{\gamma}^{(j,i) }x^{\gamma}=-\frac{x_{j}}{x_{i}}\sum_{\gamma\in\boldsymbol{Z}_{N}}\widehat{\mu}_{\gamma}x^{\gamma}.
\end{gather*}
Thus the equation becomes
\begin{gather*}
 \prod\limits_{j=1,\, j\neq i}^{N}\left( 1-\frac{x_{j}}{x_{i}}\right)
\sum_{\gamma\in\boldsymbol{Z}_{N}}\gamma_{i}\widehat{\mu}_{\gamma}x^{\gamma}\\
 \qquad{} =\kappa\sum_{j\neq i}\prod\limits_{k\neq i,j}\left( 1-\frac{x_{k}}{x_{i}}\right) \left\{ \frac{x_{j}}{x_{i}}\tau((i,j))
\sum_{\gamma\in\boldsymbol{Z}_{N}}\widehat{\mu}_{\gamma}x^{\gamma}+\sum_{\gamma\in\boldsymbol{Z}_{N}}\widehat{\mu}_{\gamma}x^{\gamma}\tau((i,j))\right\} .
\end{gather*}
This completes the proof.
\end{proof}

\begin{Corollary}
The coefficients $\{ \widehat{\mu}_{\alpha}\} $ satisfy the same recurrences as $\big\{ \widehat{K}_{\alpha}\big\}$ in~\eqref{FCrec}.
\end{Corollary}

\subsection{Maximal singular support}

Above we showed that $\mu$ and $K$ satisfy the same Laurent series dif\/ferential systems~(\ref{Kdiffeq1}) and~(\ref{diffKLs}), thus the singular part $\mu_{S}$ also satisf\/ies this relation. The singular part $\mu_{S}$ is the restriction of $\mu$ to $\bigcup\limits_{i<j}\big\{ x\in\mathbb{T}^{N} \colon x_{i}=x_{j}\big\} $, a closed set. For each pair $ \{k,\ell\} $ let $E_{k\ell}=\big\{ x\in\mathbb{T}^{N}\colon x_{k}\neq x_{\ell}\big\} $, an open subset of $\mathbb{T}^{N}$. For $i\neq j$ let
\begin{gather*}
T_{i,j}=\big\{ x\in\mathbb{T}^{N}\colon x_{i}=x_{j}\big\} \cap\bigcap \limits_{\{ k,\ell \} \cap\{i,j\} =\varnothing} \{ E_{k\ell}\cap E_{ik}\cap E_{jk}\};
\end{gather*}
this is an intersection of a closed set and an open set, hence $T_{i,j}$ is a~Baire set and the restriction~$\mu_{i,j}$ of $\mu$ to $T_{i,j}$ is a Baire
measure. Informally $T_{i,j}=\big\{ x\in\mathbb{T}^{N}\colon x_{i}=x_{j},\# \{ x_{k} \} =N-1\big\}$. We will prove that $\mu_{i,j}=0$ for all $i\neq j$. That is, $\mu_{S}$ is supported by $\big\{ x\in \mathbb{T}^{N}\colon \# \{ x_{k} \} \leq N-2\big\} $ (the number of distinct coordinate values is $\leq N-2$). In \cite[Corollary~4.15]{Dunkl2016} there is an approximate identity
\begin{gather*}
\sigma_{n}^{N-1}(x) :=\sum_{k=0}^{n}\frac{(-n) _{k}}{( 1-n-N) _{k}}\sum_{\alpha\in\boldsymbol{Z}_{N},\, \vert \alpha \vert =2k}x^{\alpha},
\end{gather*}
which satisf\/ies $\sigma_{n}^{N-1}(x) \geq0$ and $\sigma _{n}^{N-1}\ast\nu\rightarrow\nu$ as $n\rightarrow\infty$, in the weak-$\ast$ sense for any f\/inite Baire measure $\nu$ on $\mathbb{T}^{N}/\mathbb{D}$ (referring to functions and measures on $\mathbb{T}^{N}$ homogeneous of degree zero as Laurent series). The set $T_{i,j}$ is pointwise invariant under $(i,j)$ thus $\mathrm{d}\mu_{i,j}(x)=\mathrm{d}\mu_{i,j}(x(i,j))=\tau ((i,j)) \mathrm{d}\mu_{i,j}(x)\tau((i,j))$.

\begin{Remark}
The density of Laurent polynomials in $C^{(1) }\big(\mathbb{T}^{N}\big) $ can be shown by using an approximate identity, for example: $u_{n}(x) =\Big\{ \frac{1}{n+1}\sum\limits_{j=-n}^{n}( n- \vert j \vert +1) x_{1}^{j}\Big\} \sigma _{n}^{N-1}(x) $; for any $\alpha\in\mathbb{Z}^{N}$ the coef\/f\/icient of $x^{\alpha}$ in $u_{n}(x) $ tends to $1$ as $n\rightarrow\infty$ (express $\alpha=(\alpha_{1}-m) \varepsilon_{1}+( -m,\alpha_{2},\ldots,\alpha_{N}) $ where
$m=\sum\limits_{j=2}^{N}\alpha_{j}$). Then $f\ast u_{n}\rightarrow f$ in the $C^{(1) }\big( \mathbb{T}^{N}\big) $ norm.
\end{Remark}

Let $K_{n}^{s}=\sigma_{n}^{N-1}\ast\mu_{S}$ (convolution), a Laurent polynomial, f\/ix $\ell$ in $1\leq\ell\leq N$, and consider the functionals $F_{\ell,n}$, $G_{\ell,n}$ on scalar functions $p\in C^{(1)}\big( \mathbb{T}^{N}\big) $
\begin{gather*}
F_{\ell,n}(p) :=\int_{\mathbb{T}^{N}}p(x)
\prod\limits_{j\neq\ell}\left( 1-\frac{x_{j}}{x_{\ell}}\right) x_{\ell}\partial_{\ell}K_{n}^{s}(x) \mathrm{d}m(x), \\
G_{\ell,n}(p) :=\kappa\sum_{i\neq\ell}\int_{\mathbb{T}^{N}}p(x) \prod\limits_{j\neq\ell,i}\left( 1-\frac{x_{j}}{x_{\ell}}\right) \left\{ \frac{x_{i}}{x_{\ell}}\tau\left( \ell,i\right) K_{n}^{s}(x) +K_{n}^{s}(x) \tau( \ell,i)\right\} \mathrm{d}m(x) .
\end{gather*}
By construction the functionals annihilate $x^{\alpha}$ for $\alpha\notin\mathbf{Z}_{N}$. For a f\/ixed $\alpha\in\mathbf{Z}_{N}$ the value $F_{\ell,n}( x^{-\alpha}) -G_{\ell,n}( x^{-\alpha})$ is
\begin{gather*}
 \alpha_{\ell}A_{\alpha}b_{n}(\alpha) +\sum_{i=1}^{N-1} (-1) ^{i}\sum_{J\subset E_{\ell},\, \#J=i}( \alpha_{\ell}+i)A_{\alpha+i\varepsilon_{\ell}-\varepsilon_{J}}b_{n}( \alpha+i\varepsilon_{\ell}-\varepsilon_{J}) \\
 -\kappa\sum_{j=1,\, j\neq\ell}^{N}\sum_{i=0}^{N-2}(-1) ^{\ell}\sum_{J\subset E_{\ell,j},\, \#J=i}\left\{
\begin{matrix}
\tau(\ell,j) A_{\alpha+(i+1) \varepsilon_{\ell
}-\varepsilon_{j}-\varepsilon_{J}}b_{n}( \alpha+(i+1)
\varepsilon_{\ell}-\varepsilon_{j}-\varepsilon_{J}) \\
{} +A_{\alpha+i\varepsilon_{\ell}-\varepsilon_{J}}\tau(\ell,j)
b_{n}( \alpha+i\varepsilon_{\ell}-\varepsilon_{J})
\end{matrix}
\right\},
\end{gather*}
where $b_{n}(\gamma) :=\frac{(-n) _{\vert \gamma\vert /2}}{(1-N-n) _{\vert \gamma\vert/2}}$ (from the Laurent series of $\sigma_{n}^{N-1})$, and $A_{\gamma}:=\int_{\mathbb{T}^{N}}x^{-\gamma}\mathrm{d}\mu_{S}$ . Thus for f\/ixed $\alpha$ the coef\/f\/icients $b_{n}(\cdot) \rightarrow1$ as $n\rightarrow
\infty$ and the expression tends to the dif\/ferential system~\ref{Kdiffeq1} and%
\begin{gather*}
\lim_{n\rightarrow\infty}\big( F_{\ell,n} ( x^{-\alpha} )-G_{\ell,n}( x^{-\alpha}) \big) =0.
\end{gather*}
This result extends to any Laurent polynomial by linearity. From the approximate identity property
\begin{gather*}
\lim_{n\rightarrow\infty}G_{\ell,n}(p) =\kappa\sum_{i\neq\ell
}\int_{\mathbb{T}^{N}}p(x) \prod\limits_{j\neq\ell,i}\left(1-\frac{x_{j}}{x_{\ell}}\right) \left\{ \frac{x_{i}}{x_{\ell}}\tau (\ell,i) \mathrm{d}\mu_{S}(x) +\mathrm{d}\mu_{S}(x) \tau( \ell,i) \right\} ,
\end{gather*}
and
\begin{gather*}
\Vert G_{\ell,n}(p)\Vert \leq M\sup_{x,\, i}\left\vert p(x) \prod\limits_{j\neq\ell,i}\left( 1-\frac{x_{j}}{x_{\ell}}\right) \right\vert ,
\end{gather*}
where $M$ depends on $\mu_{S}$. Also
\begin{gather*}
F_{\ell,n}(p) =-\int_{\mathbb{T}^{N}}x_{\ell}\partial_{\ell}\left\{ p(x) \prod\limits_{j\neq\ell}\left( 1-\frac{x_{j}}{x_{\ell}}\right) \right\} K_{n}(x) \mathrm{d}m(x) ,
\end{gather*}
and
\begin{gather*}
\lim_{n\rightarrow\infty}F_{\ell,n}(p) =-\int_{\mathbb{T}^{N}}x_{\ell}\partial_{\ell}\left\{ p(x) \prod\limits_{j\neq\ell
}\left( 1-\frac{x_{j}}{x_{\ell}}\right) \right\} \mathrm{d}\mu_{S}(x)
\end{gather*}
for Laurent polynomials $p$. By density of Laurent polynomials in $C^{(1) }\big( \mathbb{T}^{N}\!/\mathbb{D}\big) $ ($\mathbb{D}\! =\!\{( u,u,\ldots,u) \colon\!$ $\vert u\vert =1 \} $ thus functions homogeneous of degree zero on~$\mathbb{T}^{N}$ can be considered as functions on the quotient group $\mathbb{T}^{N}/\mathbb{D}$) we obtain
\begin{gather}
 -\int_{\mathbb{T}^{N}}x_{\ell}\partial_{\ell}\left\{ p(x)\prod\limits_{j\neq\ell}\left( 1-\frac{x_{j}}{x_{\ell}}\right) \right\}\mathrm{d}\mu_{S}(x)\nonumber\\
\qquad{} =\kappa\sum_{i\neq\ell}\int_{\mathbb{T}^{N}}p(x) \prod\limits_{j\neq\ell,i}\left( 1-\frac{x_{j}}{x_{\ell}}\right) \left\{
\frac{x_{i}}{x_{\ell}}\tau\left( \ell,i\right) \mathrm{d}\mu_{S}(x) +\mathrm{d}\mu_{S}(x) \tau( \ell,i)\right\} ,\label{Tijformula}
\end{gather}
for all $p\in C^{(1) }\big( \mathbb{T}^{N}/\mathbb{D}\big)$.

\begin{Theorem}
For $1\leq i<j\leq N$ the restriction $\mu_{S}|T_{i,j}=0$.
\end{Theorem}

\begin{proof} It suf\/f\/ices to take $i=1,j=2$. Let $E$ be an open neighborhood of a point in $T_{1,2}$ such that if $x\in\overline{E}$ (the closure) and $x_{i}=x_{j}$ for some pair $i<j$ then $i=1$ and $j=2$. Let $f(x) \in C^{(1) }\big( \mathbb{T}^{N}/\mathbb{D}\big) $ have support $\subset E$. Thus $f(x) =0=\partial_{1}f(x) $ at each point~$x$ such that $x_{i}=x_{j}$ for some pair $\{i,j\} \neq \{1,2\} $ ($f=0$ on a neighborhood of $\bigcup\limits_{i<j}\{
x\colon x_{i}=x_{j} \} \backslash T_{1,2}$). Then in formula~(\ref{Tijformula}) (with $\ell=1$) applied to $f$ the measure~$\mu_{S}$ can be replaced with~$\mu_{1,2}$. Evaluate the derivative
\begin{gather*}
x_{1}\partial_{1}\left\{ f(x) \prod\limits_{j\neq1}\left(
1-\frac{x_{j}}{x_{1}}\right) \right\} =\left( 1-\frac{x_{2}}{x_{1}}\right) f(x) x_{1}\partial_{1}\left\{ \prod\limits_{j>2}
\left( 1-\frac{x_{j}}{x_{1}}\right) \right\} \\
\qquad{} +f(x) \frac{x_{2}}{x_{1}}\prod\limits_{j>2}\left( 1-\frac
{x_{j}}{x_{1}}\right) +\left( x_{1}\partial_{1}f(x) \right) \left( 1-\frac{x_{2}}{x_{1}}\right) \prod\limits_{j>2}\left(
1-\frac{x_{j}}{x_{1}}\right) .
\end{gather*}
Each term vanishes on $\bigcup\limits_{i<j}\{x\colon x_{i}=x_{j}\} \backslash T_{1,2}$, and restricted to $T_{1,2}$ the value is $f(x) \prod\limits_{j>2}\big( 1-\frac{x_{j}}{x_{1}}\big) $. Thus
\begin{gather*}
-\int_{\mathbb{T}^{N}}x_{1}\partial_{1}\left\{ f(x)
\prod\limits_{j\neq1}\left( 1-\frac{x_{j}}{x_{1}}\right) \right\} \mathrm{d}\mu_{S}(x) =-\int_{\mathbb{T}^{N}}x_{1}
\partial_{1}\left\{ f(x) \prod\limits_{j\neq1}\left(1-\frac{x_{j}}{x_{1}}\right) \right\} \mathrm{d}\mu_{1,2}(x)
\\
\hphantom{-\int_{\mathbb{T}^{N}}x_{1}\partial_{1}\left\{ f(x)
\prod\limits_{j\neq1}\left( 1-\frac{x_{j}}{x_{1}}\right) \right\} \mathrm{d}\mu_{S}(x)}{}
 =-\int_{\mathbb{T}^{N}}f(x) \prod\limits_{j>2}\left(1-\frac{x_{j}}{x_{1}}\right) \mathrm{d}\mu_{1,2}(x).
\end{gather*}
The right hand side of the formula reduces to
\begin{align*}
 \kappa\int_{\mathbb{T}^{N}}f(x) \prod\limits_{j>2}\left(1-\frac{x_{j}}{x_{1}}\right) \left\{ \frac{x_{2}}{x_{1}}\tau (1,2) \mathrm{d}\mu_{1,2}(x) +\mathrm{d}\mu_{1,2} (x) \tau(1,2) \right\} \\
\qquad {} =2\kappa\tau(1,2) \int_{\mathbb{T}^{N}}f(x) \prod\limits_{j>2}\left( 1-\frac{x_{j}}{x_{1}}\right) \mathrm{d}\mu_{1,2}(x) ,
\end{align*}
since $\mathrm{d}\mu_{1,2}(x) \tau(1,2) =\tau(1,2) \mathrm{d}\mu_{1,2}(x) $. Thus the integral is a matrix $F(f) $ such that
\begin{gather*}
( I+2\kappa \tau(1,2)) F(f) =0,
\end{gather*} which implies $F(f) =0$ provided $\kappa\neq\pm\frac{1}{2}$. Replacing $f(x) $ by $f(x) \prod\limits_{j>2}\big(1-\frac{x_{j}}{x_{1}}\big) ^{-1}$ shows that $\mu_{1,2}=0$, since $E$ was arbitrarily chosen.
\end{proof}

\subsection{Boundary values for the measure}

In this subsection we will show that $K$ satisf\/ies the weak continuity condition
\begin{gather*}
\lim\limits_{x_{N-1}-x_{N}\rightarrow0} ( K(x)-K(x(N-1,N))) =0
\end{gather*}
at the faces of~$\mathcal{C}_{0}$ and then deduce that $H_{1}$ commutes with~$\sigma$ (as described in Theorem~\ref{suffctH}). The idea is to use the inner product property of $\mu$ on functions supported in a small enough neighborhood of $x^{(0)}=\big(1,\omega,\ldots,\omega^{N-3},\omega^{-3/2},\omega^{-3/2}\big)$ where~$\mu_{S}$ vanishes, so that only~$K$ is involved, then argue that a~failure of the continuity condition leads to a contradiction.

Let $0<\delta\leq\frac{2\pi}{3N}$ and def\/ine the boxes
\begin{gather*}
\Omega_{\delta} =\big\{ x\in\mathbb{T}^{N}\colon \big\vert x_{j} -x_{j}^{(0) }\big\vert \leq2\sin\tfrac{\delta}{2},\, 1\leq j\leq
N\big\} ,\\
\Omega_{\delta}^{\prime} =\big\{ x\in\mathbb{T}^{N-1}\colon \big\vert x_{j}-x_{j}^{(0) }\big\vert \leq2\sin\tfrac{\delta}{2},\, 1\leq
j\leq N-1\big\}
\end{gather*}
(so if $x_{j}=e^{\mathrm{i\theta}_{j}}$ then $\big\vert \theta_{j}-\frac{2\pi(j-1) }{N}\big\vert \leq\delta$, for $1\leq j\leq N-2$ and $\big\vert \theta_{j}-\frac{( 2N-3) \pi}{2}\big\vert $ for $N-1\leq j\leq N$). Then $x\in\Omega_{\delta}$ implies $ \vert x_{i}-x_{j} \vert \geq2\sin\frac{\delta}{2}$ for $1\leq i<j\leq N$ except for $i=N-1$, $j=N$ (that is, $ \vert \theta_{i}-\theta_{j} \vert \geq\delta $). Further $\Omega_{\delta}$ is invariant under $(N-1,N)$, while $\Omega_{\delta}\cap\Omega_{\delta}(i,N) =\varnothing$ for $1\leq i\leq N-2$. For brevity set $\phi_{0}=\frac{(2N-3) \pi
}{2}$, $e^{\mathrm{i\phi}_{0}}=\omega^{-3/2}$. We consider the identity
\begin{gather*}
\int_{\mathbb{T}^{N}} ( x_{N}\mathcal{D}_{N}f(x)) ^{\ast}\mathrm{d}\mu(x) g(x) -\int_{\mathbb{T}^{N}}f(x) ^{\ast}\mathrm{d}\mu(x) x_{N}\mathcal{D}_{N}g(x) =0
\end{gather*}
for $f,g\in C^{(1) }\big( \mathbb{T}^{N};V_{\tau}\big) $ whose support is contained in $\Omega_{\delta}$. Then $\operatorname{spt}((x_{N}\mathcal{D}_{N}f(x) ) ^{\ast}g(x)) \subset\Omega_{\delta}$ and $\operatorname{spt}( f(x) ^{\ast}x_{N}\mathcal{D}_{N}g(x)) \subset \Omega_{\delta}$.

The support hypothesis and the construction of $\Omega_{\delta}$ imply that $\Omega_{\delta}\cap\big( \mathbb{T}^{N}\backslash\mathbb{T}_{\rm reg}^{N}\big) \subset T_{N-1,N}$ and thus $\mathrm{d}\mu$ can be replaced by $K(x) \mathrm{d}m(x) $ in the formula. Recall the general identity~(\ref{dfKg})
\begin{gather*}
-( x_{N}\mathcal{D}_{N}f(x)) ^{\ast}K(x) g(x) +f(x) ^{\ast}K(x)x_{N}\mathcal{D}_{N}g(x) \\
\qquad{} =x_{N}\partial_{N} \{ f(x) ^{\ast}K(x)g(x) \} -\kappa\sum_{1\leq j\leq N-1}\frac{1}{x_{N}-x_{j}}\big\{ x_{j}f(x(j,N)) ^{\ast}\tau((j,N))
K(x) g(x) \\
\qquad\quad{} +x_{N}f(x) ^{\ast}K(x) \tau((j,N)) g( x( j,N)) \big\} .
\end{gather*}

Specialize to $\operatorname{spt}(f) \subset\Omega_{\delta}$ and $\operatorname{spt}(g) \subset\Omega_{\delta}$ and $x\in\Omega_{\delta}$ then only the $j=N-1$ term in the sum remains, and this term changes sign under $x\mapsto x(N-1,N) $.

For $\varepsilon>0$ let $\Omega_{\delta,\varepsilon}=\big\{ x\in \Omega_{\delta}\colon \vert x_{N-1}-x_{N} \vert \geq2\sin\frac {\varepsilon}{2}\big\} $, then
\begin{gather*}
\int_{\Omega_{\delta,\varepsilon}}\big\{ x_{N}\partial_{N} ( f (x ) ^{\ast}K(x) g(x) ) \\
\qquad{} + (x_{N}\mathcal{D}_{N}f(x) ) ^{\ast}K(x)g(x) -f(x) ^{\ast}K(x) x_{N}\mathcal{D}_{N}g(x) \big\} \mathrm{d}m(x)=0,
\end{gather*}
because $\Omega_{\delta,\varepsilon}$ is $(N-1,N) $-invariant (similar argument to Proposition~\ref{xdfKg-fKxdg}). By integrability
\begin{gather*}
\lim_{\varepsilon\rightarrow0_{+}}\int_{\Omega_{\delta,\varepsilon}}\big\{ ( x_{N}\mathcal{D}_{N}f(x) ) ^{\ast}K(x) g(x) -f(x) ^{\ast}K(x)
x_{N}\mathcal{D}_{N}g(x) \mathrm{d}m(x) \big\} =0,
\end{gather*}
hence
\begin{gather*}
\lim_{\varepsilon\rightarrow0_{+}}\int_{\Omega_{\delta,\varepsilon}}x_{N}\partial_{N} ( f(x) ^{\ast}K(x) g (x )) \mathrm{d}m(x) =0.
\end{gather*}
Now we use iterated integration. For f\/ixed $\theta_{N-1}$ the ranges for $\theta_{N}$ are obtained by inserting suitable gaps into the interval
$ [ \phi_{0}-\delta,\phi_{0}+\delta ] $ (as usual, $x=\big(e^{\mathrm{i}\theta_{1}},\ldots,e^{\mathrm{i}\theta_{N}}\big) $):
\begin{enumerate}\itemsep=0pt
\item[1)] $\phi_{0}-\delta\leq\theta_{N-1}\leq\phi_{0}-\delta+\varepsilon \colon [\theta_{N-1}+\varepsilon,\phi_{0}+\delta] $,

\item[2)] $\phi_{0}-\delta+\varepsilon\leq\theta_{N-1}\leq\phi_{0}+\delta-\varepsilon\colon [ \phi_{0}-\delta,\theta_{N-1}-\varepsilon ]\cup [ \theta_{N-1}+\varepsilon,\phi_{0}+\delta ] $,

\item[3)] $\phi_{0}+\delta-\varepsilon\leq\theta_{N-1}\leq\phi_{0}+\delta\colon [\phi_{0}-\delta,\theta_{N-1}-\varepsilon]$.
\end{enumerate}

From $x_{N}\partial_{N}=-\mathrm{i}\frac{\partial}{\partial\theta_{N}}$ it follows that
\begin{gather*}
\frac{1}{2\pi}\int_{a}^{b}x_{N}\partial_{N}(f^{\ast}Kg)
\big( \big( e^{\mathrm{i}\theta_{1}},\ldots,e^{\mathrm{i}\theta_{N}}\big) \big) \mathrm{d}\theta_{N}\\
\qquad{} =\frac{1}{2\pi\mathrm{i}}\big\{ (f^{\ast}Kg) \big( \big( e^{\mathrm{i}\theta_{1}},\ldots,e^{\mathrm{i}\theta_{N-1}},e^{\mathrm{i}b}\big)\big) -(f^{\ast}Kg) \big( \big(e^{\mathrm{i}\theta_{1}},\ldots,e^{\mathrm{i}\theta_{N-1}},e^{\mathrm{i}a}\big) \big) \big\} .
\end{gather*}
Since $f$ and $g$ are at our disposal we can take their supports contained in $\Omega_{\delta/2}$ then for $0<\varepsilon\leq\frac{\delta}{4}$ the
$\mathrm{d}\theta_{N}$-integrals for~(1) and~(3) vanish and the integrals in~(2) have the value
\begin{gather*}
\frac{1}{2\pi\mathrm{i}}\big\{ (f^{\ast}Kg) \big( \big(e^{\mathrm{i}\theta_{1}},\ldots,e^{\mathrm{i}\theta_{N-1}},e^{\mathrm{i} ( \theta_{N-1}-\varepsilon) }\big) \big) - ( f^{\ast}Kg ) \big( \big( e^{\mathrm{i}\theta_{1}},\ldots,e^{\mathrm{i}\theta_{N-1}},e^{\mathrm{i}( \theta_{N-1}+\varepsilon ) }\big)
\big) \big\} . %\label{fKgdiff}
\end{gather*}

We use the power series (from (\ref{Lzseries}))
\begin{gather*}
L_{1}(x(u,z)) =\left( \big( u^{2}-z^{2}\big) \prod_{j=1}^{N-2}x_{j}\right) ^{-\gamma\kappa}\rho\big(z^{-\kappa},z^{\kappa}\big) \sum_{n=0}^{\infty}\alpha_{n} ( x (u,0 ) ) z^{n},
\end{gather*}
with the notation $x(u,z) = ( x_{1},\ldots,x_{N-2},u-z,u+z ) $ for $x\in\Omega_{\delta}$. Recall $\alpha_{n} (x(u,0)) $ is analytic for a~region including
$\Omega_{\delta}$ and $\sigma\alpha_{n}(x(u,0)) \sigma=(-1) ^{n}\alpha_{n}(x(u,0))$. Also $\alpha_{0}(x(u,0)) $ is invertible. As in Section~\ref{locps} def\/ine $C_{1}:=CL_{1}(x_{0}) ^{-1}$ so that $L_{1}(x) ^{\ast}C_{1}^{\ast}C_{1}L_{1}(x) =L(x) ^{\ast}HL(x) $ on their common domain, and set $H_{1}=C_{1}^{\ast}C_{1}$. It suf\/f\/ices to use the approximation $\sum\limits_{n=0}^{\infty}\alpha_{n}(x(u,0))z^{n}=\alpha_{0}(x(u,0)) +O (\vert z \vert)$, uniformly in $\Omega_{2\pi/3N}$.

Let
\begin{gather*}
\eta^{(1) }(x,\theta,\varepsilon) :=\big(x_{1},\ldots,x_{N-2},e^{\mathrm{i}\theta},e^{\mathrm{i}(\theta+\varepsilon)}\big),\\
\eta^{(2) }(x,\theta,\varepsilon) :=\big(x_{1},\ldots,x_{N-2},e^{\mathrm{i}(\theta-\varepsilon)},e^{\mathrm{i}\theta}\big),\\
\eta^{(3) }(x,\theta,\varepsilon) :=\big(x_{1},\ldots,x_{N-2},e^{\mathrm{i}\theta},e^{\mathrm{i}(\theta-\varepsilon)}\big)
\end{gather*}
with $\eta^{(1) },\eta^{(2) }\in\Omega_{\delta}\cap\mathcal{C}_{0}$ and $\eta^{(3) }=\eta^{(2)}(N-1,N) $. Set $\zeta=e^{\mathrm{i}\varepsilon}$. Then
\begin{gather*}
\eta^{(1) } =x( u_{1}-z_{1},u_{1}+z_{1}), \qquad u_{1}=\frac{1}{2}e^{\mathrm{i}\theta} ( 1+\zeta ) , \qquad z_{1}=\frac
{1}{2}e^{\mathrm{i}\theta}(\zeta-1) ,\\
\eta^{(2) } =x ( u_{2}-z_{2},u_{2}+z_{2} ), \qquad u_{2}=\frac{1}{2}e^{\mathrm{i}\theta}\big( 1+\zeta^{-1}\big), \qquad z_{2}=\frac{1}{2}e^{\mathrm{i}\theta}\big( 1-\zeta^{-1}\big) =\zeta^{-1}z_{1}.
\end{gather*}

The invariance properties of $K$ imply $K\big( \eta^{(3) }\big) =\sigma K\big(\eta^{(2)}\big) \sigma$. Then
\begin{gather*}
K\big(\eta^{(1)}\big) =\alpha_{0} ( x (u_{1},0 ) ) ^{\ast}\rho\big( z_{1}^{-\kappa},z_{1}^{\kappa }\big) ^{\ast}H_{1}\rho\big( z_{1}^{-\kappa},z_{1}^{\kappa}\big) \alpha_{0}( x(u_{1},0)) +O\big( \vert z_{1} \vert ^{1-2\vert \kappa\vert }\big) ,\\
\sigma K\big(\eta^{(2)}\big) \sigma =\alpha_{0} (x(u_{2},0) ) ^{\ast}\rho\big( z_{2}^{-\kappa},z_{2}^{\kappa}\big) ^{\ast}\sigma H_{1}\sigma\rho\big( z_{2}^{-\kappa},z_{2}^{\kappa}\big) \alpha_{0}(x(u_{2},0)) +O\big( \vert z_{2} \vert ^{1-2\vert \kappa\vert}\big),
\end{gather*}
because $\sigma\alpha_{0}(x(u,0)) \sigma =\alpha_{0}(x(u,0)) $ and $\sigma=\rho(-1,1) $ commutes with $\rho( z_{2}^{-\kappa},z_{2}^{\kappa}) $. To express $K\big(\eta^{(1)}\big) -\sigma K\big(\eta^{(2)}\big) \sigma$ let
\begin{gather*}
A_{1}=\rho\big( z_{1}^{-\kappa},z_{1}^{\kappa}\big)^{\ast}H_{1}\rho\big(z_{1}^{-\kappa},z_{1}^{\kappa}\big)=O\big(\vert z_{1}\vert^{-2\vert \kappa\vert}\big),\\
A_{2} =\rho\big( z_{2}^{-\kappa},z_{2}^{\kappa}\big) ^{\ast}\sigma H_{1}\sigma\rho\big( z_{2}^{-\kappa},z_{2}^{\kappa}\big) =O\big(\vert z_{2}\vert^{-2\vert \kappa\vert }\big),
\end{gather*}
then
\begin{gather*}
K\big(\eta^{(1)}\big) -\sigma K\big( \eta^{(2) }\big) \sigma\\
\qquad{} =\alpha_{0} ( x(u_{1},0) ) ^{\ast}A_{1}\alpha_{0}( x(u_{1},0)) -\alpha_{0}(x(u_{2},0)) ^{\ast}A_{2}\alpha_{0}(x(u_{2},0)) +O\big(\vert z_{1}\vert ^{1-2\vert \kappa\vert}\big) .
\end{gather*}
Also $u_{2}-u_{1}=\frac{1}{2}x_{N-1}\xi_{1}\big( \zeta^{-1}-\zeta\big) =O(\vert z_{1}\vert) $ (since $\vert z_{1} \vert = \vert 1-\zeta \vert $) thus $\alpha_{0} ( x(u_{1},0) ) -\alpha_{0}( x(u_{2},0)) =O( \vert z_{1}\vert) $ and
\begin{gather*}
K\big(\eta^{(1)}\big) -\sigma K\big( \eta^{(2)}\big) \sigma=\alpha_{0} ( x(u_{1},0) ) ^{\ast}( A_{1}-A_{2}) \alpha_{0}( x(u_{1},0)) +O\big(\vert z_{1}\vert ^{1-2\vert \kappa \vert }\big) .
\end{gather*}
Using the $\sigma$-decomposition write
\begin{gather*}
H_{1}:=%
\begin{bmatrix}
H_{11} & H_{12}\\
H_{12}{}^{\ast} & H_{11}%
\end{bmatrix}
,
\end{gather*}
then
\begin{gather*}
A_{1}-A_{2} =
\begin{bmatrix}
H_{11}\vert z_{1}\vert ^{-2\kappa} & H_{12}\left( \dfrac{z_{1}%
}{\overline{z_{1}}}\right) ^{\kappa}\\
H_{12}{}^{\ast}\left( \dfrac{z_{1}}{\overline{z_{1}}}\right) ^{-\kappa} &
H_{11}\vert z_{1}\vert ^{2\kappa}%
\end{bmatrix}
-
\begin{bmatrix}
H_{11}\vert z_{2}\vert ^{-2\kappa} & -H_{12}\left( \dfrac{z_{2}%
}{\overline{z_{2}}}\right) ^{\kappa}\\
-H_{12}{}^{\ast}\left( \dfrac{z_{2}}{\overline{z_{2}}}\right) ^{-\kappa} &
H_{11}\vert z_{2}\vert ^{2\kappa}%
\end{bmatrix}
\\
\hphantom{A_{1}-A_{2}}{} =
\begin{bmatrix}
O & H_{12}\left\{ \left( \dfrac{z_{1}}{\overline{z_{1}}}\right) ^{\kappa
}+\left( \dfrac{z_{2}}{\overline{z_{2}}}\right) ^{\kappa}\right\} \\
H_{12}{}^{\ast}\left\{ \left( \dfrac{z_{1}}{\overline{z_{1}}}\right)
^{-\kappa}+\left( \dfrac{z_{2}}{\overline{z_{2}}}\right) ^{-\kappa}\right\}
& O
\end{bmatrix},
\end{gather*}
because $ \vert z_{2} \vert = \vert z_{1} \vert $. Next
\begin{gather*}
\frac{z_{1}}{\overline{z_{1}}} =e^{2\mathrm{i}\theta}\frac{e^{\mathrm{i}\varepsilon}-1}{e^{-\mathrm{i}\varepsilon}-1}=-e^{\mathrm{i} (
2\theta+\varepsilon ) },\qquad \frac{z_{2}}{\overline{z_{2}}}=e^{2\mathrm{i}\theta}\frac{1-e^{-\mathrm{i}\varepsilon}}{1-e^{\mathrm{i}\varepsilon}
}=-e^{\mathrm{i} ( 2\theta-\varepsilon ) },\\
\left( \frac{z_{1}}{\overline{z_{1}}}\right) ^{\kappa}+\left( \frac{z_{2}}{\overline{z_{2}}}\right) ^{\kappa} =\big( {-}e^{2\mathrm{i}\theta
}\big) ^{\kappa}\big\{ e^{\mathrm{i}\varepsilon\kappa}+e^{-\mathrm{i}\varepsilon\kappa}\big\} =2\big({-}e^{2\mathrm{i}\theta}\big) ^{\kappa
}\cos\varepsilon\kappa,
\end{gather*}
where some branch of the power function is used (the interval where it is applied is a small arc of the unit circle), and $\phi_{0}-\delta_{1}<\theta<\phi_{0}-\delta_{1}$.

We show that $H_{12}=O$, equivalently $H_{1}$ commutes with $\sigma$. By way of contradiction suppose some entry $h_{ij}\neq0$ ($1\leq i\leq m_{\tau}<j\leq
n_{\tau}$). There exist $r>0$, $\delta_{1}\geq\delta_{2}>0$ and $c\in \mathbb{C}$ with $\vert c\vert =1$ such that
\begin{gather*}
\operatorname{Re}\big( 2c\big( {-}e^{2\mathrm{i}\theta}\big) ^{\kappa}h_{ij}\big) >r
\end{gather*}
for $\phi_{0}-\delta_{2}\leq\theta\leq\phi_{0}+\delta_{2}$. Let $p (x) \in C^{(1) }\big( \mathbb{T}^{N}\big) $ such that $\operatorname{spt}(p) \subset\Omega_{\delta_{2}/2}$, $0\leq p (x) \leq1$ and $p(x) =1$ for $x\in\Omega_{\delta_{2}/4}$. Let $f(x) =p(x) \alpha_{0}( x(u,0))^{-1}\varepsilon_{i}$ and $g(x) =cp ( x) \alpha_{0}(x(u,0)) ^{-1}\varepsilon _{j}$ (for $x\in\Omega_{\delta_{2}/2}$). Also impose the bound $0<\varepsilon <\frac{\delta_{2}}{4}$. Then
\begin{gather*}
f\big(\eta^{(1)}\big) ^{\ast}K\big( \eta^{(1) }\big) g\big(\eta^{(1)}\big) =p\big(\eta^{(1) }\big) ^{2}c\left( \frac{z_{1}}{\overline{z_{1}}
}\right) ^{\kappa}h_{ij}+O\big( \vert z_{1}\vert ^{1-2 \vert \kappa \vert }\big), \\
f\big( \eta^{(3) }\big) ^{\ast}\sigma K\big(\eta^{(2) }\big) \sigma g\big( \eta^{(3)}\big) =-p\big( \eta^{(3) }\big) ^{2}c\left(
\frac{z_{2}}{\overline{z_{2}}}\right) ^{\kappa}h_{ij}+O\big( \vert z_{2} \vert ^{1-2\vert \kappa\vert }\big) ,
\end{gather*}

Suppose $x=\big( e^{\mathrm{i}\theta_{1}},\ldots,e^{\mathrm{i}\theta_{N}}\big) \in\Omega_{\delta_{2}/2}$ then $p(x) =1$ for
$\phi_{N-1}-\frac{\delta_{2}}{4}\leq\theta_{N-1},\theta_{N}\leq\phi _{N-1}-\frac{\delta_{2}}{4}$, thus $p\big( \eta^{(1) } (x,\theta,\varepsilon) \big) =1$ for $\phi_{0}-\frac{\delta_{2}}{4}\leq\theta\leq\phi_{0}-\frac{\delta_{2}}{4}-\varepsilon$ and $p\big( \eta^{(3) }(x,\theta,\varepsilon) \big) =1$ for
$\phi_{0}-\frac{\delta_{2}}{4}+\varepsilon\leq\theta\leq\phi_{0}-\frac {\delta_{2}}{4}$. By the continuous dif\/ferentiability it follows that for
$\phi_{0}-\frac{\delta_{2}}{4}\leq\theta\leq\phi_{0}-\frac{\delta_{2}}{4}$ both $p\big(\eta^{(1)}\big) =1+O( \varepsilon) $ and $p\big( \eta^{(3) }\big) =1+O(\varepsilon)$. Thus
\begin{gather*}
 p\big(\eta^{(1)}\big) f\big( \eta^{(1)}\big) ^{\ast}K\big(\eta^{(1)}\big) g\big(\eta^{(1)}\big) p\big(\eta^{(1)}\big)
-p\big( \eta^{(3) }\big) f\big(\eta^{(3)}\big) ^{\ast}\sigma K\big(\eta^{(2)}\big) \sigma g\big( \eta^{(3) }\big) p\big(\eta^{(3)}\big) \\
 \qquad {} =p\big(\eta^{(1)}\big) ^{2}c\left\{ \left(
\frac{z_{1}}{\overline{z_{1}}}\right) ^{\kappa}+\left( \frac{z_{2}}{\overline{z_{2}}}\right) ^{\kappa}\right\} h_{ij}+O\big(\vert z_{1}\vert ^{1-2\vert \kappa\vert }\big) +O(\varepsilon) .
\end{gather*}
By construction
\begin{gather*}
\operatorname{Re}\left( c\left\{ \left( \frac{z_{1}}{\overline{z_{1}}}\right) ^{\kappa}+\left( \frac{z_{2}}{\overline{z_{2}}}\right) ^{\kappa
}\right\} h_{ij}\right) >r\cos\varepsilon\kappa,
\end{gather*}
multiply the inequality by $p\big( \big( e^{\mathrm{i}\theta_{1}},\ldots,e^{\mathrm{i}\theta_{N-1}},e^{\mathrm{i}\theta_{N-1}}\big) \big)^{2}$ and integrate over the $(N-1) $-box $\Omega_{\delta_{2}}^{\prime}$ with respect to $\mathrm{d}m_{N-1}=\big( \frac{1}{2\pi}\big) ^{N-1}\mathrm{d}\theta_{1}\cdots\mathrm{d}\theta_{N-1}$. This integral dominates the integral over~$\Omega_{\delta_{2}/4}^{\prime}$, thus
\begin{gather*}
\operatorname{Re}\int_{\Omega_{\delta_{2}}^{\prime}}p\big( \eta^{(1) }\big) ^{2}c\left\{ \left( \frac{z_{1}}{\overline{z_{1}}%
}\right) ^{\kappa}+\left( \frac{z_{2}}{\overline{z_{2}}}\right) ^{\kappa}\right\} h_{ij}\mathrm{d}m_{N-1}\geq r\cos\varepsilon\kappa\left(\frac{\delta_{2}}{2\pi}\right) ^{N-1}.
\end{gather*}
This contradicts the limit of the integral being zero as $\varepsilon\rightarrow0$. The ignored parts of the integral are $O\big( \varepsilon
^{1-2\vert \kappa\vert }\big) $ and $ \vert \kappa \vert <\frac{1}{2}$. We have proven the following:

\begin{Theorem}\label{H1comm}For $-1/h_{\tau}<\kappa<1/h_{\tau}$ the matrix $H_{1}= (L_{1}^{\ast}(x_{0})) ^{-1}HL_{1}(x_{0})^{-1}$ commutes with~$\sigma$.
\end{Theorem}

\section{Analytic matrix arguments}\label{anlcmat}

In this section we set up some tools from linear algebra dealing with matrices whose entries are analytic functions of one variable. The aim is to establish the existence of an analytic solution for the matrices described in Theorem~\ref{H1comm}. The key fact is that the solution $L_{1}(x;\kappa)$ of~(\ref{Lsys}) is analytic in~$\kappa$ for $\vert \kappa\vert <\frac{1}{2}$, in fact for $\kappa\in\mathbb{C}\backslash ( \mathbb{Z}+\frac{1}{2})$; the series expansion in~(\ref{Lzseries}) does not apply to $\kappa\in\mathbb{Z}+\frac{1}{2}$ and a logarithmic term has to be included for this case. Set $b_{N}:= ( 2 ( N^{2}-N+2 )) ^{-1}$, the bound from Section~\ref{suffco}. One would like use analytic continuation to extend the inner product property of $L^{\ast}HL$ from the interval $-b_{N}<\kappa<b_{N}$ to $-1/h_{\tau}<\kappa<1/h_{\tau}$ but the Bochner theorem argument for the existence of $\mu$ does not allow~$\kappa$ to be a complex variable. The following arguments work around this obstacle.

\begin{Theorem}\label{Matrixeq}Suppose $M(\kappa) $ is an $m\times n$ complex matrix with $m\geq n-1$ such that the entries are analytic in $\kappa\in
D_{r}:= \{ \kappa\in\mathbb{C}\colon \vert \kappa\vert <r \}$, some $r>0$ and $\operatorname{rank}(M(\kappa)) =n-1$ for a~real interval $-r_{1}<\kappa<r_{1}$ then $\operatorname{rank}( M (\kappa)) =n-1$ for all $\kappa\in D_{r}$ except possibly at isolated points~$\lambda$ where $\operatorname{rank} ( M(\lambda)) <n-1$, and there is a nonzero vector function $v (\kappa) $, analytic on $D_{r}$ such that $M(\kappa)v(\kappa) =0$ and $v(\kappa) $ is unique up to
multiplication by a scalar function.
\end{Theorem}

\begin{proof} Let $M^{\prime}(\kappa) $ be any $n\times n$ submatrix of~$M(\kappa) $ (when $m\geq n$), that is, $M^{\prime}$ is composed of~$n$ rows of $M(\kappa) $, then $\det M^{\prime} (\kappa) $ is analytic for $\kappa\in D_{r}$ and by hypothesis $\det M^{\prime}(\kappa) =0$ for $-r_{1}<\kappa<r_{1}$. This implies $\det M^{\prime}(\kappa) \equiv0$ for all $\kappa$, by analyticity. Thus $\operatorname{rank}(M(\kappa)) \leq n-1$ for all $\kappa\in D_{r}$. For each subset $J= \{ j_{1},\ldots,j_{n-1} \} $ with $1\leq j_{1}<\cdots<j_{n-1}\leq m$ let $M_{J,k}(\kappa) $ be the $(n-1) \times(n-1) $ submatrix of $M(\kappa) $ consisting of rows $\#$ $j_{1},\ldots,j_{n-1}$ and deleting column $\#k$, and $X_{J} (\kappa) :=[ \det M_{J,1}(\kappa) ,\ldots,\det M_{J,n}(\kappa)] $, an $n$-vector of analytic functions. There exists at least one set $J$ for which $X_{J}(0) \neq [ 0,\ldots,0 ] $ otherwise $\operatorname{rank} ( M (0)) <n-1$. By continuity there exists $\delta>0$ such that at least one $\det M_{J,k}(\kappa) \neq0$ for $ \vert \kappa \vert <\delta$ and $v(\kappa) = \big[(-1)^{k-1}\det M_{J,k}(\kappa) \big] _{k=1}^{n}$ is a~nonzero vector in the null-space of~$M(\kappa)$ (by Cramer's rule and the rank hypothesis). The analytic equation $M(\kappa) v(\kappa) =0$ holds in a neighborhood of $\kappa=0$ and thus for all of~$D_{r}$. If~\smash{$v(\kappa) =0$} for isolated points $\kappa_{1},\ldots,\kappa_{\ell}$ in $\vert \kappa\vert \leq r_{2}<r$ then $v(\kappa) $ can be multiplied by $\prod \limits_{j=1}^{\ell}\big( 1-\frac{\kappa}{\kappa_{j}}\big) ^{-a_{j}}$ for suitable positive integers $a_{1},\ldots,a_{\ell}$ to produce a~solution never zero in $\vert \kappa\vert \leq r_{2}<r$. (It may be possible that there are inf\/initely many zeros in the open set~$D_{r}$.)
\end{proof}

We include the parameter in the notations for $L$ and $L_{1}$. The $\ast$ operation replaces $x_{j}$ by $x_{j}^{-1}$, the constants by their conjugates, and transposing, but $\kappa$ is unchanged to preserve the analytic dependence, see Def\/inition~\ref{defadj}. For $x\in\mathbb{T}_{\rm reg}^{N}$ and real $\kappa$ the Hermitian adjoint of $L_{1} (x_{0};\kappa) $) agrees with $L_{1}(x_{0};\kappa) ^{\ast}$. The matrix $M(\kappa) $ is implicitly def\/ined by the linear
system with the unknown $B_{1}$
\begin{gather}
B_{1} =\sigma B_{1}\sigma,\nonumber\\
\upsilon L_{1}^{\ast}(x_{0};\kappa) B_{1}L_{1} (x_{0};\kappa ) =L_{1}^{\ast}(x_{0};\kappa) B_{1}L_{1}(x_{0};\kappa) \upsilon.\label{Mkeqn}
\end{gather}
(Recall $\upsilon=\tau(w_{0}) $.) The entries of $M (\kappa) $ are analytic in $\vert \kappa\vert <\frac{1}{2}$. The equation $B_{1}=\sigma B_{1}\sigma$ implies that $B_{1}$ has $n:=m_{\tau}^{2}+ ( n_{\tau}-m_{\tau} ) ^{2}$ possible nonzero entries, by the $\sigma$-block decomposition. The number of equations $m= n_{\tau}^{2}-\dim \{ A\colon A\upsilon=\upsilon A \} $. Because $w_{0}$ and $(N-1,N) $ generate~$\mathcal{S}_{N}$ and~$\tau$ is irreducible $A\sigma=\sigma A$ and $A\upsilon=\upsilon A$ imply $A=cI$ for $c\in \mathbb{C}$ by Schur's lemma. This implies $n\geq m-1$ (else there are two linearly independent solutions). By a result of Stembridge \cite[Section~3]{Stembridge1989} $n$ can be computed from the following: (recall $\omega:= \exp\frac{2\pi\mathrm{i}}{N}$) for $0\leq j\leq N-1$ set~$e_{j}$ equal to the multiplicity of $\omega^{j}$ in the list of the $n_{\tau}$ eigenvalues of $\upsilon$ and set $F_{\tau}(q) :=q^{e_{0}}+q^{e_{1}}+\dots+q^{e_{N-1}}$ then
\begin{gather*}
F_{\tau}(q) =\left\{ q^{n(\tau) }\prod_{i=1}^{N}\big( 1-q^{i}\big) \prod_{(i,j) \in\tau}\big(1-q^{h(i,j) }\big) ^{-1}\right\} \operatorname{mod}\big(1-q^{N}\big),
\end{gather*}
where $n(\tau) :=\sum\limits_{i=1}^{\ell(\tau) }(i-1) \tau_{i}$ and $h(i,j) $ is the hook length at $(i,j) $ in the diagram of $\tau$ (note $F_{\tau} (1) =n_{\tau}$). Thus $\dim \{ A\colon A\upsilon=\upsilon A \} =\sum\limits_{j=0}^{N-1}e_{j}^{2}$. For example let $\tau=(4,2) $ then $n_{\tau}=9$, $m_{\tau}=3$ and $n=45$ while $F_{\tau}(q) =2+q+2q^{2}+q^{3}+2q^{4}+q^{5}$ and $\dim \{ A\colon A\upsilon=\upsilon A\} =15$, $m=66$.

\begin{Theorem}
For $-1/h_{\tau}<\kappa<1/h_{\tau}$ there exists a unique Hermitian matrix $H$ such that $\mathrm{d}\mu=L^{\ast}HL\mathrm{d}m$. Also $( L_{1}(x_{0}) ^{\ast}) ^{-1}HL_{1}(x_{0}) ^{-1}$ commutes with $\sigma$.
\end{Theorem}

\begin{proof}
For any Hermitian $n_{\tau}\times n_{\tau}$ matrix $B$ def\/ine the Hermitian form
\begin{gather*}
\langle f,g\rangle _{B}:=\int_{\mathbb{T}^{N}}f(x)^{\ast}L(x) ^{\ast}BL(x) g(x)\mathrm{d}m(x)
\end{gather*}
for $f,g\in C^{(1) }\big( \mathbb{T}^{N};V_{\tau}\big) $. If the form satisf\/ies $\langle wf,wg\rangle _{B}=\langle f,g \rangle _{B}$ for all $w\in\mathcal{S}_{N}$ and $ \langle x_{i}\mathcal{D}_{i}f,g\rangle _{B}$ $= \langle f,x_{i}\mathcal{D}_{i}g \rangle _{B}$ for $1\leq i\leq N$ then $B$ is determined up to multiplication by a constant. This follows from the density of the span of the nonsymmetric Jack (Laurent) polynomials in $C^{(1) }\big(\mathbb{T}^{N};V_{\tau}\big)$. By Theorem~\ref{H1comm} there exists a~nontrivial solution of the system~(\ref{Mkeqn}) for $-1/h_{\tau}<\kappa<1/h_{\tau}$. Thus $\operatorname{rank}M(\kappa) \leq n-1$ in this interval. Now suppose that $B_{1}$ is a nontrivial solution of~(\ref{Mkeqn}) for some $\kappa$ such that $-b_{N}<\kappa<b_{N}$. Then both $B^{(1) }:=B_{1}+B_{1}^{\ast}$ and $B^{(2) }:=\mathrm{i}(B_{1}-B_{1}^{\ast}) $ are also solutions (by the invariance of~(\ref{Mkeqn}) under the adjoint operation). Let $H^{(i) }:=L_{1}^{\ast}(x_{0};\kappa) B^{(i)}L_{1}(x_{0};\kappa) $ for $i=1,2$ then by Theorem~\ref{suffctH} the forms $\langle \cdot,\cdot\rangle _{H^{(1) }}$ and $\langle \cdot,\cdot\rangle _{H^{(2) }}$ satisfy the above uniqueness condition. Hence either $B_{1}$ is Hermitian or $B^{(1) }=rB^{(2) }$ for some $r\neq0$ which implies $B_{1}=\frac{1}{2}(r-\mathrm{i}) B^{(2) }$, that is, $B_{1}$ is a scalar multiple of a Hermitian matrix. Thus there is a~unique (up to scalar multiplication) solution of~(\ref{Mkeqn}) which implies $\operatorname{rank}M(\kappa) \geq n-1$ in $-b_{N}<\kappa<b_{N}$.

Hence the hypotheses of Theorem~\ref{Matrixeq} are satisf\/ied, and there exists a nontrivial solution $B_{1}(\kappa) $ which is analytic in $\vert \kappa\vert <\frac{1}{2}$. Since the Hermitian form is positive def\/inite for $-1/h_{\tau}<\kappa<1/h_{\tau}$ we can use the fact that $B_{1}(\kappa) $ is a multiple of a positive-def\/inite matrix when $\kappa$ is real (in fact, of the matrix $H_{1}$ arising from $\mu$ as in Section~\ref{orthmu}) and its trace is nonzero (at least on a complex neighborhood of $\{ \kappa\colon -1/h_{\tau}<\kappa<1/h_{\tau}\} $ by continuity). Set $B_{1}^{\prime}(\kappa) :=\Big( n_{t}/\sum\limits_{i=1}^{n_{\tau}}B_{1}(\kappa) _{ii}\Big) B_{1}(\kappa) $, analytic and $\operatorname{tr}( B_{1}^{\prime}(\kappa)) =1$ thus the normalization produces a unique analytic (and Hermitian for real~$\kappa$) matrix in the null-space of~$M(\kappa) $. Let $H(\kappa) =L_{1}(x_{0};\kappa) ^{\ast}B_{1}(\kappa) L_{1}(x_{0};\kappa) $ then for f\/ixed $f,g\in C^{(1) }\big(\mathbb{T}^{N};V_{\tau}\big) $ and $1\leq i\leq N$
\begin{gather*}
\int_{\mathbb{T}^{N}}\left\{
\begin{matrix}
( x_{i}\mathcal{D}_{i}f(x) ) ^{\ast}L^{\ast}(x;\kappa) H(\kappa) L(x;\kappa) g(x) \\
-f(x) ^{\ast}L^{\ast}(x;\kappa) H(\kappa) L(x;\kappa) x_{i}\mathcal{D}_{i}g(x)
\end{matrix}\right\} \mathrm{d}m(x)
\end{gather*}
is an analytic function of $\kappa$ which vanishes for $-b_{N}<\kappa<b_{N}$ hence for all $\kappa$ in $-1/h_{\tau}<\kappa<1/h_{\tau}$; this condition is required for integrability. This completes the proof.
\end{proof}

By very complicated means we have shown that the torus Hermitian form for the vector-valued Jack polynomials is given by the measure $L^{\ast}HL\mathrm{d}m$. The orthogonality measure we constructed in~\cite{Dunkl2016} is absolutely continuous with respect to the Haar measure. We conjecture that $L^{\ast} ( x;\kappa) H(\kappa) L(x;\kappa)$ is integrable for $-1/\tau_{1}<\kappa<1/\ell(\tau) $ but $H(\kappa)$ is not positive outside $\vert \kappa \vert <1/h_{\tau}$ (the length of $\tau$ is $\ell(\tau) :=\max\{ i\colon \tau_{i}\geq1\}$). In as yet unpublished work we have found explicit formulas for $L^{\ast}HL$ for the two-dimensional representations $(2,1)$ and $(2,2) $ of $\mathcal{S}_{3}$ and $\mathcal{S}_{4}$ respectively, using hypergeometric functions. It would be interesting to f\/ind the normalization constant, that is, determine the scalar multiple of $H(\kappa) $ which results in $\langle 1\otimes T,1\otimes T \rangle _{H(\kappa)}= \langle T,T \rangle _{0}$ (see~(\ref{admforms})) the ``initial condition'' for the form. In \cite[Theorem~4.17(3)]{Dunkl2016} there is an inf\/inite series for $H(\kappa)$ but it involves all the Fourier coef\/f\/icients of~$\mu$.

\subsection*{Acknowledgement}

Some of these results were presented at the conference ``Dunkl operators, special functions and harmonic analysis'' held at Universit\"{a}t Paderborn, Germany, August 8--12, 2016.

\LastPageEnding

\end{document}